\documentclass[11pt]{amsart}
\usepackage{xcolor}
\usepackage[a4paper,hmargin=3cm,vmargin=3cm]{geometry}
\usepackage{amsfonts,amssymb,amscd,amstext,hyperref}
\usepackage{graphicx}
\usepackage{color}
\usepackage[dvips]{epsfig}
\usepackage[utf8]{inputenc}
\usepackage{comment}

\usepackage{fancyhdr}
\usepackage{times}

\usepackage{enumerate}
\usepackage{titlesec}
\usepackage{mathrsfs}
\usepackage{stmaryrd}

\def\R{\mathbb{R}}

\def\M{\mathbb{M}}

\newcommand{\ben}{\begin{enumerate}}
\newcommand{\bit}{\begin{itemize}}
\newcommand{\een}{\end{enumerate}}
\newcommand{\eit}{\end{itemize}}

\newcommand{\ed}{\end{document}}

\def\cO{\mathcal{O}}
\DeclareMathOperator{\cotanh}{cotanh}

\let\8=\infty \let\0=\emptyset  
\let\hat=\widehat

\let\tilde=\widetilde

\def\S{\Sigma}

\def\mA{\mathcal{A}}
\def\mB{\mathcal{B}}

\def\cte.{\mathop{\rm cte.}\nolimits}

\def\cosh{\mathop{\rm cosh }\nolimits}
\def\tanh{\mathop{\rm tanh }\nolimits}
\def\arcsinh{\mathop{\rm arcsinh }\nolimits}

\def\R{\mathbb{R}}

\def\H{\mathbb{H}}
\def\S{\mathbb{S}}

\newfont{\bb}{msbm10 at 12pt}

\titleformat{\section}
{\filcenter\bfseries\large} {\thesection{.}}{0.2cm}{}
\titleformat{\subsection}[runin]
{\bfseries} {\thesubsection{.}}{0.15cm}{}[.]
\titleformat{\subsubsection}[runin]
{\em}{\thesubsubsection{.}}{0.15cm}{}[.]

\usepackage[up,bf]{caption}

\newtheorem{theorem}{Theorem}[section]
\newtheorem{lemma}[theorem]{Lemma}
\newtheorem{proposition}[theorem]{Proposition}

\newtheorem{remark}[theorem]{Remark}
\newtheorem{corollary}[theorem]{Corollary}
\newtheorem{definition}[theorem]{Definition}

\theoremstyle{definition}

\numberwithin{equation}{section}
\numberwithin{figure}{section}

\usepackage{color}

\begin{document}
\fancyhead[LO]{Free boundary minimal annuli}
\fancyhead[RE]{Alberto Cerezo}
\fancyhead[RO,LE]{\thepage}

\thispagestyle{empty}

\begin{center}
{\bf \LARGE Free boundary minimal annuli in geodesic balls of \texorpdfstring{$\H^3$}{}}

\vspace*{5mm}

\hspace{0.2cm} {\Large Alberto Cerezo}
\end{center}

\vspace{0.5cm}


\footnote[0]{
\mbox{} \hspace{0.25cm} \emph{Keywords}: minimal surfaces, free boundary, spherical curvature lines, critical catenoid}

\vspace*{7mm}

\begin{quote}
{\small
\noindent {\bf Abstract}\hspace*{0.1cm}
We construct a countable collection of one-parameter families of non-rotational minimal annuli with free boundary in geodesic balls of hyperbolic 3-space. Every surface within a given family shares a common prismatic symmetry group, and they appear as bifurcations from certain free boundary hyperbolic catenoids.

\vspace*{0.1cm}
}
\end{quote}

\section{Introduction}
A classical question in differential geometry is to study the existence of minimal surfaces in compact Riemannian three-manifolds $M$, as these surfaces appear as critical points for the area functional. In the case of manifolds with boundary $\partial M$, it is natural to consider the set of surfaces $\Sigma$ with boundary $\partial \Sigma$ contained in $\partial M$. The critical points of the area functional within this set correspond with free boundary minimal surfaces, that is, minimal surfaces $\Sigma$ which meet $\partial M$ orthogonally along $\partial \Sigma$. More generally, free boundary surfaces with constant mean curvature (not necessarily minimal) appear as critical points of the area functional under volume constraints; see \cite{GHN}. One case of particular interest is when $M$ is a geodesic ball of one of the space forms $\mathbb{R}^3$, $\mathbb{S}^3$ or $\mathbb{H}^3$. 

\medskip

In 1985, Nitsche \cite{Nit} proved that any free boundary minimal or constant mean curvature (CMC) disk in the unit ball of $\mathbb{R}^3$ is totally umbilical, i.e., an equatorial disk or a spherical cap. Ros and Souam \cite{RS,S} later generalized this result to free boundary disks in geodesic balls of $\mathbb{S}^3$ and $\mathbb{H}^3$. All of these disk-type surfaces are rotational, i.e., they are invariant under a 1-parameter group of isometries that fix pointwise a geodesic of the ambient space. For the Euclidean case, Nitsche claimed without proof that any free boundary minimal or CMC annulus in the unit ball of $\mathbb{R}^3$ should be rotational. However, in 1995, Wente \cite{W2} constructed examples of immersed, non-rotational CMC ($H \neq 0$) annuli with free boundary in the unit ball. This led Wente to ask whether any embedded, free boundary CMC annulus in $\mathbb{R}^3$ should be rotational. Regarding the minimal case, Fernández, Hauswirth and Mira \cite{FHM} recently constructed immersed, non-rotational minimal annuli with free boundary in the unit ball. This result was based on the use of the Weierstrass representation and the study of the Liouville equation with an overdetermined condition. Shortly after, a similar construction was done by Kapouleas and McGrath \cite{KM} using doubling methods.

\medskip

The results in \cite{FHM,KM} left open the {\em critical catenoid conjecture}, which claims that the critical catenoid is the only embedded, free boundary minimal annulus in the unit ball of $\mathbb{R}^3$; see \cite{F,FL,L}. However, it showed that the embeddedness condition cannot be omitted from the conjecture. In a more general context, the problem of whether an embedded, free boundary minimal annulus in a geodesic ball of $\mathbb{S}^3$ or $\mathbb{H}^3$ should be rotational has been recently studied by Lima and Menezes \cite{LM2} and Medvedev \cite{Me}. In fact, Medvedev conjectured that any embedded, free boundary minimal annulus in a geodesic ball of these space forms should be rotational, and predicted that immersed, non-rotational examples should exist; see \cite[Section 1.1]{Me}.

\medskip

In \cite{CFM1}, Fernández, Mira and the author constructed embedded, non-rotational CMC ($H \neq 0$) annuli in the unit ball of $\mathbb{R}^3$, answering in the negative the question posed by Wente \cite{W2}. In the same spirit as \cite{FHM}, these examples rely on the construction of a set of solutions to the sinh-Gordon equation with an overdetermined condition. Later on, in \cite{CFM2} they used these solutions to construct a wide family of non-rotational minimal and CMC annuli with free boundary in geodesic balls of $\mathbb{S}^3$ and $\mathbb{H}^3$. More specifically, they constructed embedded free boundary annuli with constant mean curvature $H > 1$ in $\mathbb{H}^3$ and $H \geq \frac{1}{\sqrt{3}}$ in $\mathbb{S}^3$. For the minimal case, immersed examples were constructed exclusively in $\mathbb{S}^3$.

\subsection{Result of the paper and discussion}
We answer in the affirmative the question of whether non-rotational, immersed, free boundary minimal annuli exist in geodesic balls of $\mathbb{H}^3$, posed by Medvedev \cite{Me}. This result shows that the embeddedness assumption is necessary if one aims to extend the {\em critical catenoid conjecture} to the hyperbolic context.

\medskip

We now introduce the main result in this paper. We model the hyperbolic space $\mathbb{H}^3$ as the hyperboloid
\begin{equation}\label{eq:espaciohiperbolico1}
    \mathbb{H}^3 := \{(x_1,x_2,x_3,x_4) \in \mathbb{L}^4 \; : \; x_1^2 + x_2^2 + x_3^2 - x_4^2 = -1, x_4> 0\},
\end{equation}
where $\mathbb{L}^4$ is the 4-dimensional Lorentz space, i.e., $\mathbb{R}^4$ equipped with the pseudo-Riemannian metric $ds^2 = dx_1^2 + dx_2^2 + dx_3^2 - dx_4^2$. Then, it holds:

\begin{theorem}\label{thm:main}
    There exists an open interval $\mathcal{J} \subset \left(-\frac{1}{\sqrt{2}},-\frac{1}{\sqrt{3}}\right)$ satisfying the following property: let $q \in \mathcal{J} \cap \mathbb{Q}$, which we express as an irreducible fraction $q = -m/n$, $m,n \in \mathbb{N}$. Then, there exists $\varepsilon = \varepsilon(q)$ and a real analytic 1-parameter family $\{\mathbb{A}_q(\eta) \; : \; \eta \in [0,\varepsilon)\}$ of minimal compact annuli $\mathbb{A}_q(\eta)$ in $\mathbb{H}^3$ such that:
    \begin{enumerate}
        \item The annulus $\mathbb{A}_q(\eta)$ is free boundary in some geodesic ball ${\bf B} = {\bf B}(q,\eta) \subset \mathbb{H}^3$ whose geodesic center is $\mathbf{e}_4 = (0,0,0,1) \in \mathbb{H}^3$.
        \item $\mathbb{A}_q(\eta)$ is invariant under the symmetry with respect to the totally geodesic surface $\mathbf{S}:= \{x_3 = 0\} \cap \mathbb{H}^3$ and also with respect to $n$ totally geodesic surfaces which meet equiangularly along the geodesic ${\mathcal{L}}:= \{x_1 = x_2 = 0\} \cap \mathbb{H}^3$.
        \item The curve $\mathbb{A}_q(\eta) \cap \mathbf{S}$ of $\mathbb{A}_q(\eta)$ is a geodesic of $\mathbb{A}_q(\eta)$ with rotation index $-m$ in ${\bf S}$.
        \item The annulus $\mathbb{A}_q(\eta)$ is foliated by spherical curvature lines.
        \item If $\eta > 0$, the annulus $\mathbb{A}_q(\eta)$ has a prismatic symmetry group of order $4n$. This group is generated by the isometries in item (2). In particular, $\mathbb{A}_q(\eta)$ is not rotational.
        \item If $\eta = 0$, the annulus $\mathbb{A}_q(0)$ is an $m$-cover of a free boundary hyperbolic catenoid.
    \end{enumerate}
\end{theorem}
To the author's knowledge, the surfaces in Theorem \ref{thm:main} are the first known free boundary minimal surfaces in geodesic balls of $\mathbb{H}^3$ that are not rotational.

\medskip

\subsection{Strategy of the proof and difficulties}

The basic strategy for proving Theorem \ref{thm:main} comes from our previous works \cite{CFM1,CFM2}. However, the method followed in \cite{CFM1,CFM2} presents severe complications when dealing with minimal surfaces in $\mathbb{H}^3$, as we do here. We explain next which is the main source of complication in this hyperbolic context, as well as the alternative approach that we will use to circumvent it.

First, we note that a common step in the strategy followed in \cite{CFM1,CFM2,FHM} and the present paper is to construct a set of particular solutions $\omega(u,v)$ to the general overdetermined system:
\begin{equation}\label{eq:sistemaIntro}
\begin{cases}
\Delta \omega + (H^2 + \kappa) e^{2\omega} - B e^{-2\omega} = 0, \\
2\omega_u = \alpha(u)e^\omega + \beta(u) e^{-\omega},
\end{cases}
\end{equation}
for some constant $B > 0$ and functions $(\alpha(u),\beta(u))$. This system was considered by Wente \cite{W}, who observed that any solution $\omega(u,v)$ to \eqref{eq:sistemaIntro} induces a conformal immersion $\psi(u,v)$ with constant mean curvature $H$ in a space of constant sectional curvature $\kappa$, with metric $e^\omega |dz|^2$ and holomorphic Hopf differential $Qdz^2$, where $z = u + iv$ and $B = 4|Q|^2$. This holds since the first condition of \eqref{eq:sistemaIntro} can be seen as the Gauss equation of the surface. On the other hand, the second condition of \eqref{eq:sistemaIntro} has a geometrical interpretation: namely, that one of the families of lines of curvature of the surface is {\em spherical}, i.e., each line of the family lies in a totally umbilical surface of the ambient space. The idea in \cite{CFM1,CFM2,FHM}, and also here, is to find within this family a free boundary annulus with the desired properties.

\medskip

In \cite{FHM}, \eqref{eq:sistemaIntro} was studied for the particular case of minimal surfaces in Euclidean space, i.e. $H = \kappa = 0$, with $B = 1$. In this way, the Gauss equation for the surface reduces to the Liouville equation. In \cite{CFM1,CFM2}, several cases were considered, but all of them involved fixing the constants $H$ and $\kappa$ so that $H^2 + \kappa > 0$. This condition covers CMC $H \neq 0$ surfaces in $\mathbb{R}^3$ ($\kappa = 0$), minimal and CMC surfaces in $\mathbb{S}^3$ ($\kappa = 1$) and CMC surfaces with $H > 1$ in $\mathbb{H}^3$ ($\kappa = -1$). In this situation, a simple change of variables is sufficient to turn the Gauss equation into the sinh-Gordon equation; see \cite[Section 4]{CFM2}. However, in the present case of minimal surfaces in hyperbolic space (that is, $H = 0$, $\kappa = -1$), \eqref{eq:sistemaIntro} cannot be reduced to the Liouville or sinh-Gordon equation; instead, it turns into the cosh-Gordon equation, whose treatment is quite different, as we explain next.

\medskip

Wente \cite{W} observed that in order for \eqref{eq:sistemaIntro} to have solutions, the pair of functions $(\alpha(u),\beta(u))$ must satisfy a second order system of differential equations: 
\begin{equation}\label{eq:alphabetaIntro}
\begin{cases}
\alpha'' = \hat a\alpha - 2\alpha^2 \beta - 2(H^2 + \kappa)\beta, \\
\beta'' = \hat a\beta - 2\alpha\beta^2 - 2B\alpha,
\end{cases}
\end{equation}
see \cite[Equation (1.4)]{W}, where $\hat a \in \mathbb{R}$ is a constant. He also noticed that \eqref{eq:alphabetaIntro} could be reduced to a first order system of differential equations for a pair of functions $(s(\lambda),t(\lambda))$ via a Hamilton-Jacobi procedure. The qualitative behaviour of this pair strongly depends on the sign of $H^2 + \kappa$: indeed, if $H^2 + \kappa \geq 0$ (i.e. if \eqref{eq:sistemaIntro} is a Liouville or sinh-Gordon type system), then the pair $(s,t)$ is real valued; see for example \cite[Sections III, IV]{W}. However, if $H^2 + \kappa < 0$, the functions $(s,t)$ are expected to take complex values, and the Hamilton-Jacobi procedure does not yield significative information. This is the case for minimal surfaces in $\mathbb{H}^3$, since $H^2 + \kappa = -1 < 0$. We also refer the reader to \cite[Section V]{W}, where it is also explained how the geometry of minimal surfaces in $\mathbb{H}^3$ with spherical curvature lines is much more complicated than its Euclidean and spherical counterparts.

To overcome this difficulty, in the present paper we adopt a different approach compared to \cite{CFM1, CFM2,FHM}. Instead of constructing solutions $\omega(u,v)$ to \eqref{eq:sistemaIntro} for fixed values of $H$ and $\kappa$, we fix $H = 0$ and treat $\kappa \in \mathbb{R}$ as a parameter. This allows us to deduce relevant properties of the solutions $\omega(u,v)$ for $\kappa < 0$ by studying the cases $\kappa = 0$ and $\kappa > 0$, modelled by Liouville and sinh-Gordon type equations. In this way, we avoid the complex-valued phase space analysis for $(s,t)$ that we would need in the cosh-Gordon case. One of the challenges of this strategy is showing that not only $\omega(u,v)$, but also the induced immersions $\psi(u,v)$ and every geometrical tool developed throughout this paper depends analytically on the parameter $\kappa$, specially as $\kappa$ crosses the value zero.

\subsection{Structure of the paper}
We start by following a similar structure to the one presented in \cite{CFM2}. In Section \ref{sec:preliminares}, we construct a family of analytic solutions $\omega(u,v)$ to the overdetermined system \eqref{eq:sistemaIntro} for $H = 0$ and $B = \frac{1}{4}$, which depend on three parameters $(a,b,\kappa)$. This set of solutions induces a family of minimal surfaces $\psi(u,v) = \psi(u,v;a,b,\kappa)$, each of them immersed in a space form $\mathbb{M}^3(\kappa)$ of constant curvature $\kappa$. These immersions are conformally parametrized by curvature lines. Moreover, the $v$-lines $v\mapsto \psi(u_0,v)$ of the immersion are {\em spherical}, i.e., they are contained in totally umbilical surfaces $\mathcal{Q}(u_0)$ of $\mathbb{M}^3(\kappa)$. By Joachimsthal's theorem, the intersection angle of $\psi(u,v)$ with $\mathcal{Q}(u_0)$ is constant along the corresponding $v$-line. 

\medskip

As explained before, a key difference between \cite{CFM2} and the current situation is that we now use $\kappa$, the sectional curvature of the ambient space, as a parameter in the construction. To do so, we use a special model for the spaces $\mathbb{M}^3(\kappa)$ which we detail in Definition \ref{def:M3k}. Specifically, $\mathbb{M}^3(\kappa)$ is defined as a hypersurface of $\mathbb{R}^4$ equipped with a certain metric dependent on $\kappa$. Intuitively, $\mathbb{M}^3(\kappa)$ is defined so that it depends {\em smoothly} on the parameter $\kappa$ around the Euclidean case $\kappa = 0$ connecting the spherical and hyperbolic spaces. This makes the immersions $\psi(u,v;a,b,\kappa)$ analytic as mappings from an open region of $\mathbb{R}^5$ into $\mathbb{R}^4$.

\medskip

The goal of Section \ref{sec:geometria} is to determine the symmetries of the surfaces $\psi(u,v)$. We will define a {\em period map} $\Theta = \Theta(a,b,\kappa)$ with the property that the $v$-lines of $\psi(u,v)$ are closed whenever $\Theta$ is a rational number. In such a case, the restriction $\Sigma_0 = \Sigma_0(u_0;a,b,\kappa)$ of $\psi(u,v)$ to a strip $[-u_0,u_0]\times \mathbb{R}$ covers a minimal annulus. By our study in Section \ref{sec:preliminares}, the boundary components of $\Sigma_0$ are contained in the totally umbilical surfaces $\mathcal{Q}(-u_0)$, $\mathcal{Q}(u_0)$, and the intersection angle is constant. In particular, if we find special parameters $(a,b,\kappa)$ and $u_0> 0$ such that:
\begin{enumerate}
    \item[(i)] the period $\Theta(a,b,\kappa)$ is rational,
    \item[(ii)] the surfaces $\mathcal{Q}(u_0)$ and $\mathcal{Q}(-u_0)$ are compact and coincide,
    \item[(iii)] the intersection angle with the immersion $\psi(u,v)$ is $\pi/2$,
\end{enumerate}
then $\Sigma_0$ is a compact annulus whose boundary components meet orthogonally the geodesic sphere $S = \mathcal{Q}(u_0) = \mathcal{Q}(-u_0)$. If $\kappa \leq 0$, one can show that $\Sigma_0$ is contained in the geodesic ball $B$ bounded by $S$, i.e., $\Sigma_0$ is a free boundary annulus in $B$. So, the rest of the paper is devoted to finding values $(a,b,\kappa)$, $u_0 > 0$ satisfying the properties listed above.

\medskip

At the end of Section \ref{sec:geometria}, we show that the immersions $\psi(u,v) = \psi(u,v;a,b,\kappa)$ associated to the parameter $a = 1$ are rotational. More specifically, they are {\em spherical, Euclidean or hyperbolic} catenoids, depending on the sign of the parameter $\kappa$. This motivates Section \ref{sec:rotacionales}, in which we study the geometry of these rotational examples. Our goal there is to show that each of these surfaces has a compact piece which is free boundary in some geodesic ball of $\mathbb{M}^3(\kappa)$. The proof of this fact is an adaptation of the arguments shown in \cite[Appendix]{CFM2}. We emphasize that the analysis of these examples is fundamental, as our strategy to prove Theorem \ref{thm:main} is to detect free boundary rotational examples within the family of minimal surfaces foliated by spherical curvature lines, and find a non-rotational bifurcation from them.

\medskip

In Section \ref{sec:tau}, we find a map $\tau(a,b,\kappa): \mathcal{W} \to (0,\infty)$ defined on an open region $\mathcal{W}$ of the parameter space $(a,b,\kappa)$ such that the immersion $\psi(u,v)$ meets orthogonally the totally umbilical surfaces $\mathcal{Q}(\tau)$, $\mathcal{Q}(-\tau)$ along the $v$-lines $u = \pm \tau$.

\medskip

In Section \ref{sec:periodo1} we compute explicitly the period map $\Theta(a,b,\kappa)$ for values with $a = 1$. This allows us to determine the structure of the level sets of this map for values $(a,b,\kappa)$ near $a = 1$. This is relevant since we aim to find immersions $\psi(u,v) = \psi(u,v;a,b,\kappa)$ for which the $v$-lines are closed (that is, the immersion covers an annulus), a condition that is satisfied whenever the period map $\Theta(a,b,\kappa)$ is a rational number.

\medskip

We now provide a geometric motivation for Section \ref{sec:EuclidFB}. We already know from Section \ref{sec:rotacionales} that if $a = 1$ and $\kappa = 0$, then $\psi(u,v)$ is a Euclidean catenoid which admits a free boundary piece in some geodesic ball. Our goal now is to detect this piece by using the map $\tau$ in Section \ref{sec:tau}. More specifically, for every $b$ we know that the restriction $\Sigma_0(\tau;1,b,0)$ covers a compact piece of a catenoid which intersects the spheres $\mathcal{Q}(\tau)$, $\mathcal{Q}(-\tau)$ orthogonally, where $\tau = \tau(1,b,0)$. We will find a special $b_0$ such that the spheres $\mathcal{Q}(\tau)$, $\mathcal{Q}(-\tau)$ coincide.

\medskip

The relevance of Section \ref{sec:EuclidFB} is that it enables us to understand the condition $\mathcal{Q}(\tau) = \mathcal{Q}(-\tau)$ for values $(a,b,\kappa)$ near $(1,b_0,0)$. A simultaneous control of this property and the rationality of the period map (already studied in Section \ref{sec:periodo1}) is carried out in Section \ref{sec:main}. There, we find a countable number of real analytic curves $\hat{\mu}_q$, $q \in \mathbb{Q}$, on the parameter space $(a,b,\kappa)$ lying inside the region $\{a > 1, \kappa < 0\}$ and satisfying both properties, i.e., the period map $\Theta$ is constantly $q$ along each curve $\hat{\mu}_q$, and the spheres $\mathcal{Q}(\tau)$, $\mathcal{Q}(-\tau)$ coincide. The annuli $\Sigma_0(\tau;a,b,\kappa)$ associated to the values $(a,b,\kappa)$ in these curves are shown to be free boundary, proving Theorem \ref{thm:main}.

\subsection{Open problems}
As claimed before, the surfaces constructed in Theorem \ref{thm:main} are the first known free boundary minimal surfaces in geodesic balls of $\mathbb{H}^3$ which are not rotational. However, it is natural to expect that some of the methods used to construct free boundary minimal surfaces in the Euclidean unit ball, such as min-max techniques (see for example \cite{CFS,CSW,FrKS,FrS,Ke}), could be applied in $\mathbb{H}^3$.

\medskip

In \cite[Theorem 1.2]{FHM}, it was proved that any embedded, free boundary minimal annulus in the Euclidean unit ball foliated by spherical curvature lines must be the critical catenoid. This result relies on a theorem by Seo \cite{Seo}. The study carried out in this paper leads us to believe that there are no embedded, non-rotational free boundary minimal annuli in geodesic balls of $\mathbb{H}^3$ foliated by spherical curvature lines. This supports the validity of the $\mathbb{H}^3$ version of the critical catenoid conjecture.

\medskip

We finally emphasize that the annuli constructed in Theorem \ref{thm:main} are free boundary in {\em some} geodesic balls of $\mathbb{H}^3$. In other words,  we do not prescribe the radius of the ball in our construction. We do not know whether non-rotational free boundary minimal annuli exist on any geodesic ball of $\mathbb{H}^3$.

\subsection{Acknowledgements} This work is part of the author's PhD thesis, supervised by Isabel Fernández and Pablo Mira. The author is deeply grateful to Fernández and Mira for their guidance, suggestions and for their careful review of earlier drafts of this paper.

\section{Preliminaries}\label{sec:preliminares}
Let $\R^4_{\kappa} = (\R^4,\langle \cdot, \cdot \rangle_\kappa)$, $\kappa \neq 0$, be the pseudo-Riemannian manifold which consists of $\mathbb{R}^4$ equipped with the metric
\begin{equation}\label{eq:metricaespacial}
\langle \cdot, \cdot \rangle_\kappa:= dx_1^2 + dx_2^2 + dx_3^2 + \frac{1}{\kappa} dx_4^2.
\end{equation}
Up to a homothety in the $x_4$ coordinate, we can identify $\R^4_{\kappa}$ with either the Euclidean space (if $\kappa > 0$) or the Lorentz space $\mathbb{L}^4$ (if $\kappa < 0$). Similarly, for $\kappa = 0$, we define $\mathbb{R}^4_0 := (\mathbb{R}^4,\langle \cdot , \cdot \rangle_0)$, where $\langle \cdot , \cdot \rangle_0$ is the usual canonical metric in $\mathbb{R}^4$.

\begin{definition}\label{def:M3k}
    For any $\kappa \neq 0$, we define the 3-manifold $\mathbb{M}^3(\kappa) \subset \R^4_{\kappa}$ as
\begin{equation}
    \mathbb{M}^3(\kappa) := \{(x_1,x_2,x_3,x_4) \; : \; \kappa(x_1^2 + x_2^2 + x_3^2) + x_4^2 = 1, \;  x_4 > 0 \text{ if } \kappa < 0\}.
\end{equation}
Similarly, for $\kappa = 0$, we define $\mathbb{M}^3(0)$ as the hyperplane $\{x_4 = 1\} \subset \mathbb{R}^4_0$.
\end{definition}

\begin{remark}\label{rem:M3k} 
Observe that for every $\kappa \in \mathbb{R}$, $\mathbb{M}^3(\kappa) \subset \mathbb{R}^4_\kappa$ is a Riemannian space form of constant sectional curvature $\kappa$. This is clear for $\kappa = 0$, and if $\kappa \neq 0$ it suffices to observe that $\mathbb{M}^3(\kappa)$ is just a standard 3-sphere (resp. three-dimensional hyperboloid) in $\mathbb{R}^4$ (resp. $\mathbb{L}^4$) after applying the change $x_4 \mapsto \sqrt{|\kappa|}x_4$.
\end{remark}

\medskip

Now, let $\omega(u,v):\mathcal{D} \to \R$ be a solution of the overdetermined system (see \cite{W}):
\begin{subequations} \label{eq:overdeterminedsystem}
\begin{align}
    \Delta \omega& + \kappa e^{2\omega} - \frac{1}{4} e^{-2\omega} = 0, \label{eq:gauss} \\[8pt]
    2\omega_u &= \alpha(u)e^\omega + \beta(u) e^{-\omega}, \label{eq:lineasesfericas}
\end{align}
\end{subequations}
defined on a simply connected domain $\mathcal{D} \subseteq \R^2$, for some $\kappa \in \R$ and a pair of functions $\alpha(u),\beta(u)$. Wente \cite{W} observed that the solutions of this system induce a family of minimal immersions in $\mathbb{M}^3(\kappa)$. More specifically, by the fundamental theorem of surfaces, we can interpret \eqref{eq:gauss} as the Gauss equation of a conformal minimal immersion $\psi(u,v): \mathcal{D} \to \M^3(\kappa)$ whose fundamental forms are given by
\begin{equation}\label{eq:III}
    I = e^{2\omega}(du^2 + dv^2), \; \; \; \; II = \frac{1}{2}(du^2 - dv^2).
\end{equation}

This immersion is unique up to orientation-preserving ambient isometries. The Hopf differential $Qdz^2$, where $Q = \langle \psi_{zz},N\rangle_\kappa$, $z := u +i v$ and $N$ is the Gauss map of the immersion, is given by $Q \equiv \frac{1}{4}$, and the principal curvatures $\kappa_1,\kappa_2$ are
\begin{equation}\label{eq:kappa}
    \kappa_1 = - \kappa_2 = \frac{1}{2}e^{-2\omega}.
\end{equation}
In particular, $\psi(u,v)$ has no umbilical points. As a consequence of \eqref{eq:III}, the $u$-lines and $v$-lines are lines of curvature of $\psi(u,v)$. On the other hand, equation \eqref{eq:lineasesfericas} implies that the $v$-lines $v \mapsto \psi(u,v)$ are {\em spherical}; see Definition \ref{def:esfericas} and Lemma \ref{lem:lineasesfericas}.

\subsection{A set of solutions of the system \texorpdfstring{\eqref{eq:overdeterminedsystem}}{}}

We will now construct a 3-parameter family of solutions of the overdetermined system \eqref{eq:overdeterminedsystem}, which by our previous discussion will result in a set of minimal immersions in $\mathbb{M}^3(\kappa)$ with fundamental forms given by \eqref{eq:III}. Let $\cO$ be the parameter space
\begin{equation}\label{eq:mO}
\cO:= \left\{(a,b,\kappa) \; : \; a,b \geq 1,  4|\kappa| < 1, -4\kappa a < b\right\}.
\end{equation}
In particular, observe that $\kappa\in \left(-\frac{1}{4},\frac{1}{4}\right)$ from now on. For each $(a,b,\kappa) \in \cO$, we define the constants
\begin{equation}\label{eq:mamb}
    \mA := a + \frac{1}{a}, \; \; \; \; \; \mB := b + \frac{4\kappa}{b},
\end{equation}
as well as the polynomial
\begin{equation}\label{eq:px}
p(x):= -(x-a)\left(x - \frac{1}{a}\right)\left(4\kappa x + b\right)\left(x + \frac{1}{b}\right).
\end{equation}
Now, let $(\alpha(u),\beta(u))$ be the unique solution to the differential system (see \cite[Equation (1.4)]{W})
\begin{equation}\label{eq:alphabeta}
\left\{\def\arraystretch{1.3} \begin{array}{lll} \alpha'' & = & \hat a \alpha - 2 \alpha^2 \beta - 2\kappa \beta, \\ \beta'' & = &\hat a \beta - 2\alpha \beta^2 -\frac{1}{2}\alpha,\end{array} \right.
\end{equation}
with initial conditions and constant $\hat a$ given by
\begin{equation}\label{eq:abha}
    \begin{aligned}
        \alpha(0) &= \beta(0) = 0, \\
        4\alpha'(0) &= b + \frac{4\kappa}{b} - 4\kappa a - \frac{4\kappa}{a} = \mathcal{B} - 4\kappa \mathcal{A},\\
        4 \beta'(0) &= a + \frac{1}{a} - b - \frac{4\kappa}{b} = \mathcal{A} -  \mathcal{B},\\
        4 \hat a &= -\frac{4\kappa}{ab} - \frac{4\kappa a}{b} - \frac{b}{a} - ab + 4\kappa + 1 = -\mathcal{A}\mathcal{B} + 4\kappa + 1.
    \end{aligned}
\end{equation}

The system \eqref{eq:alphabeta} admits a first integral, namely
\begin{equation}\label{eq:primeraintegral1}
        \mathcal{C}_1 = \alpha'\beta' - \hat a\alpha \beta  + \alpha^2\beta^2 + \kappa \beta^2 + \frac{1}{4}\alpha^2,
\end{equation}
for some appropriate constant $\mathcal{C}_1 \in \mathbb{R}$ determined by \eqref{eq:abha}. If, additionally, $\kappa \geq 0$, it also holds
\begin{equation}\label{eq:primeraintegral2}
        \mathcal{C}_2 = (\alpha \beta' - \alpha' \beta)^2 + 4\left(\sqrt{\kappa}\beta' - \frac{\alpha'}{2}\right)^2 + 4(\alpha \beta - \hat a - \sqrt{\kappa})\left(\frac{\alpha}{2} - \sqrt{\kappa}\beta\right)^2
\end{equation}
for another constant $\mathcal{C}_2 \in \mathbb{R}$. Let us now define $x(v) = x(v;a,b,\kappa)$ as the unique solution of the second order differential equation
\begin{equation}\label{eq:xvv}
    x''(v) = \frac{1}{8}p'(x(v)),
\end{equation}
with initial conditions $x(0) = \frac{1}{a}$, $x'(0) = 0$, where $p(x)$ is the polynomial \eqref{eq:px}. Notice that if $a = 1$, then $x(v) \equiv 1$. For $a > 1$, it is possible to obtain the first integral
\begin{equation}\label{eq:xvintegral}
    4x'(v)^2 = p(x(v)).
\end{equation}
This implies that $x(v)$ is a periodic function taking values on the interval $\left[\frac{1}{a},a\right]$; see \eqref{eq:px}. Now, for any $v_0 \in \mathbb{R}$, we define $u \mapsto \omega(u,v_0)$ as the solution to the ODE
\begin{equation}\label{eq:omegau}
    2\omega_u(u,v_0) = \alpha(u) e^\omega + \beta(u) e^{-\omega}
\end{equation}
with initial condition $\omega(0,v_0) = \log(x(v_0))$. We emphasize that the function $\omega(u,v) = \omega(u,v;a,b,\kappa)$ might not be defined on the whole plane $(u,v) \in \mathbb{R}^2$, as \eqref{eq:omegau} may diverge in finite time. However, since $x(v)$ is periodic, $\omega(u,v)$ can always be defined on an open strip $\mathcal{S} = (-u_\mathcal{S}, u_\mathcal{S})\times \mathbb{R} \subset \mathbb{R}^2$.

\begin{remark}\label{rem:omeganalitico}
    The outlined construction shows that $\alpha(u) = \alpha(u;a,b,\kappa)$, $\beta(u) =\beta(u;a,b,\kappa)$ and $\omega(u,v) = \omega(u,v;a,b,\kappa)$ depend analytically on the parameters $(a,b,\kappa) \in \mathcal{O}$.
\end{remark}

\begin{lemma}
    The function $\omega(u,v)$ satisfies the overdetermined system \eqref{eq:overdeterminedsystem}.
\end{lemma}

\begin{proof} By construction, $\omega(u,v)$ satisfies \eqref{eq:omegau}, so we only need to check that \eqref{eq:gauss} holds. Assume first that $a > 1$. Let
\begin{equation}\label{eq:phioriginal}
    \phi(u,v) := -(4\kappa + \alpha^2)e^{2 \omega} - (1 + \beta^2)e^{-2\omega} - 4 \alpha'e^\omega + 4 \beta'e^{-\omega} + 6\alpha \beta - 4 \hat a.
\end{equation}
We will now prove that 
\begin{equation}\label{eq:omegav}
    \omega_v^2(u,v) = \frac{1}{4}\phi(u,v).
\end{equation}
By \eqref{eq:omegau}, it follows that $2\omega_{vu} = (\alpha e^\omega - \beta e^{-\omega}) \omega_v$. On the other hand, a direct computation using \eqref{eq:alphabeta} and \eqref{eq:omegau} shows that
\begin{equation}\label{eq:phiprimerau}
    \phi_u = (\alpha e^\omega - \beta e^{-\omega})\phi.
\end{equation}
We can then integrate the expressions for $\omega_{vu}$ and $\phi_u$ in $u$ to obtain $\omega_v^2 = C(v)\phi$ for some function $C(v)$. Let us check that actually $C(v) \equiv \frac{1}{4}$. At $u = 0$, it holds $\omega_v^2(0,v) = \frac{x'(v)^2}{x(v)^2}$ since $x(v) = e^{\omega(0,v)}$, and
$$x'(v)^2 = \frac{p(x(v))}{4} = -\kappa x(v)^4 - \alpha'(0) x(v)^3 - \hat ax(v)^2 +  \beta'(0) x(v) - \frac{1}{4} = \frac{x(v)^2\phi(0,v)}{4},$$
where we have used \eqref{eq:xvintegral}. In particular, $C(v) \equiv \frac{1}{4}$, as we wanted to prove. A direct computation using \eqref{eq:omegau} and \eqref{eq:omegav} shows that
$$\Delta \omega = \omega_{uu} + \omega_{vv} = \frac{1}{4}e^{-2\omega} - \kappa e^{2\omega},$$
proving \eqref{eq:gauss}.

\medskip

Now, let $a = 1$. In such a case, $x(v) \equiv 1$, and consequently $\omega = \omega(u)$ does not depend on $v$. Consider the function $\phi = \phi(u)$ in \eqref{eq:phioriginal}. It is immediate by \eqref{eq:abha} that $\phi(0) = 0$, so by \eqref{eq:phiprimerau} it holds $\phi(u) \equiv 0$. We define next
\begin{equation}\label{eq:phitilde}
    \tilde \phi(u) := -2(4\kappa + \alpha^2)e^{2\omega} +2(1 + \beta^2)e^{-2\omega}  - 4\alpha'e^{\omega} - 4\beta'e^{-\omega}.
\end{equation}
It follows from \eqref{eq:abha} that $\tilde{\phi}(0) = 0$. A direct computation using \eqref{eq:alphabeta} and \eqref{eq:omegau} shows that
$$\tilde{\phi}_u(u) = (\alpha e^\omega + \beta e^{-\omega})\phi(u)  + \frac{1}{2} (\alpha e^\omega - \beta e^{-\omega})\tilde{\phi}(u) = \frac{1}{2} (\alpha e^\omega - \beta e^{-\omega})\tilde{\phi}(u),$$
so $\tilde{\phi}(u) \equiv 0$. We finally observe that
$$\Delta \omega - \frac{1}{4}e^{-2\omega} + \kappa e^{2\omega} = \omega_{uu} - \frac{1}{4}e^{-2\omega} + \kappa e^{2\omega} = -\frac{1}{8}\tilde{\phi}(u) = 0,$$
proving \eqref{eq:gauss}.
\end{proof}

\begin{remark}\label{rem:omegaa1}
    As we saw in the proof of the previous Lemma, if $a = 1$ then $\omega = \omega(u)$ does not depend on $v$. In particular, \eqref{eq:gauss} reduces to a second order differential equation for $\omega(u)$ which yields a unique solution since $\omega(0)= \omega_u(0) = 0$; see \eqref{eq:abha}, \eqref{eq:omegau}. This equation only depends on $\kappa$, so $\omega$ is also independent of the parameter $b$ in this particular case. \eqref{eq:gauss} admits a first integral, namely
    \begin{equation}\label{eq:omegau2}
        \omega_u^2 + \kappa \left(e^{2\omega} - 1\right) + \frac{1}{4} \left(e^{-2\omega} - 1\right) = 0.
    \end{equation}
    Since $\omega(0) = 0$, $\omega_{uu}(0) >0$, it can be deduced from \eqref{eq:omegau2} that $\omega = \omega(u)$ attains its global minimum at $u = 0$.
\end{remark}

Observe that the functions $\alpha(u),\beta(u)$ are anti-symmetric by \eqref{eq:alphabeta} and \eqref{eq:abha}. This implies, along with \eqref{eq:omegau}, that
\begin{equation}\label{eq:simetriau0}
    \omega(u,v) = \omega(-u,v).
\end{equation} 
We have an additional symmetry property:

\begin{lemma}\label{lem:sigma}
    There exists some $\sigma = \sigma(a,b,\kappa) > 0$ such that $x(v)$ in \eqref{eq:xvv} is increasing in $v \in [0,\sigma]$ and
    \begin{equation}\label{eq:periodicidadsigma}
        \omega(u,j\sigma + v) = \omega(u,j \sigma - v)
    \end{equation}
    for all $j \in \mathbb{Z}$. Moreover, $\sigma = \sigma(a,b,\kappa): \mathcal{O}\to\mathbb{R}^+$ is analytic, and
    \begin{equation}\label{eq:sigmaa1}
        \sigma(1,b,\kappa) = \frac{2\pi}{\sqrt{4\kappa + \mathcal{B} + 1}}.
    \end{equation}
\end{lemma}
\begin{proof}
    Assume first that $a > 1$ and recall that $\omega(0,v) = \log(x(v))$. By \eqref{eq:xvintegral}, it follows that $x(v)$ is a periodic function which oscillates between $\frac{1}{a}$ and $a$. By definition, we have set $x(0) = \frac{1}{a}$, so the first positive value at which $x(v)$ attains its maximum is given by $v = \sigma$, where
    \begin{equation}\label{eq:sigma}
        \sigma(a,b,\kappa) := \int_{\frac{1}{a}}^a \frac{2}{\sqrt{p(x)}}dx.
    \end{equation}
    It is also immediate by \eqref{eq:xvintegral} that $x(v)$ is $2\sigma$-periodic and  $x(v + j\sigma) = x(j\sigma - v)$ for any $j \in \mathbb{Z}$. Combining this identity and \eqref{eq:omegau} we obtain \eqref{eq:periodicidadsigma}. It also follows from \eqref{eq:sigma} that $\sigma(a,b,\kappa)$ is analytic for all $\mathcal{O} \cap \{a > 1\}$.

    \medskip

    Let us now rewrite \eqref{eq:sigma} by using the change of variables $x(t) = h_a(t) := \left(a - \frac{1}{a}\right)t + \frac{1}{a}$, which yields
    $$\sigma(a,b,\kappa) = \int_{0}^1 \frac{2\left(a - \frac{1}{a}\right)}{\sqrt{p(h_a(t))}}dt.$$
    This alternative expression is defined and analytic for $a = 1$, yielding \eqref{eq:sigmaa1}. Observe that \eqref{eq:periodicidadsigma} trivially holds in this case since $\omega = \omega(u)$ for $a = 1$; see Remark \ref{rem:omegaa1}.
\end{proof}

\subsection{A family of minimal immersions in \texorpdfstring{$\mathbb{M}^3(\kappa)$}{}}\label{sec:familia}
The set of solutions $\omega(u,v)$ constructed in the previous section induces a family of minimal immersions $\psi(u,v)$ of $\mathbb{M}^3(\kappa)$ with first and second fundamental forms given by \eqref{eq:III}. This is a consequence of the fundamental theorem of surfaces and \eqref{eq:gauss}, which corresponds with the Gauss equation of the surface. The rest of the section will be devoted to develop the geometric meaning behind \eqref{eq:lineasesfericas}, the other condition satisfied by $\omega(u,v)$. We will see that this condition is equivalent to the fact that every $v$-line $v \mapsto \psi(u,v)$ of the immersion is spherical. We define this concept next.

\begin{definition}\label{def:esfericas}
    Let $u \in \mathbb{R}$. We will say that the $v$-line $v \mapsto \psi(u,v)$ is spherical if it is contained in some totally umbilical surface of $\mathbb{M}^3(\kappa)$.
\end{definition}
We recall that if $\kappa = 0$, totally umbilical surfaces are spheres or planes. If $\kappa \neq 0$, totally umbilical surfaces are the result of intersecting $\mathbb{M}^3(\kappa)$ with a hyperplane of $\mathbb{R}^4_\kappa$, that is, they correspond with sets
\begin{equation}\label{eq:Smd}
    S[m,d] := \left\{ x \in \mathbb{M}^3(\kappa) \; : \; \langle x, m \rangle_\kappa = d\right\}
\end{equation}
for some $m \in \mathbb{R}^4\setminus \{0\}$, $d \in \mathbb{R}$. The values $m,d$ are determined up to a common multiplicative constant. A necessary condition for $S[m,d]$ to be a (non-empty) surface is that
$$\langle m, m \rangle_\kappa - \kappa d^2 > 0.$$
This condition is also sufficient except when $\kappa < 0$ and $\langle m, m \rangle_\kappa \leq 0$, in which case it is additionally required that $d$ and the fourth coordinate of $m$ have opposite signs.

\medskip

If $\kappa > 0$, any surface $S[m,d]$ must be a round sphere. For $\kappa < 0$, $S[m,d]$ will be a sphere (resp. horosphere, pseudosphere) if and only if $\langle m , m \rangle_\kappa$ is negative (resp. zero, positive). Furthermore, $S[m,d]$ is totally geodesic whenever $d = 0$. Let us also remark that $S[m,d]$ can be defined alternatively as the boundary of the set
\begin{equation}\label{Smd}
    B[m,d]:= \left\{ x \in \mathbb{M}^3(\kappa) \; : \; \langle x, m \rangle_\kappa \geq d\right\}.
\end{equation}
In particular, $B[m,d]$ is a geodesic ball whenever $S[m,d]$ is a 2-sphere.

\medskip

We will now prove that condition \eqref{eq:lineasesfericas} is equivalent to the fact that the $v$-lines $v \mapsto \psi(u,v)$ are spherical. First, we recall that the immersions induced by the functions $\omega(u,v)$ are unique up to an orientation-preserving isometry. Hence, they will be determined once we fix the following values of $(\psi, \psi_u, \psi_v,N)$ at $(u,v) = (0,0)$:
\begin{equation}\label{eq:initialdata}
    \psi(0,0) = {\bf e}_4, \; \; \; \psi_u(0,0) = e^{\omega(0,0)}{\bf e}_3, \; \; \;
    \psi_v(0,0) = -e^{\omega(0,0)}{\bf e}_2, \; \; \;
    N(0,0) = {\bf e}_1,
\end{equation}
where $({\bf e}_1, {\bf e}_2, {\bf e}_3, {\bf e}_4)$ is the canonical basis of $\mathbb{R}^4$. We remark that $\psi(0,0) = {\bf e}_4$ belongs to $\mathbb{M}^3(\kappa)$ for every $\kappa \in \left(-\frac{1}{4},\frac{1}{4}\right)$, and that $\psi_u(0,0), \psi_v(0,0), N(0,0) \in T_{{\bf e}_4}\mathbb{M}^3(\kappa)$. On the other hand, the Gauss-Weingarten equations for the surface are given by

\begin{equation}\label{eq:gw}
\left\{
\renewcommand{\arraystretch}{1.1} 
\begin{array}{l}
\psi_{uu}= \omega_u \psi_u - \omega_v \psi_v + \frac{1}{2} N - \kappa e^{2\omega} \psi, \\[4pt]
\psi_{uv}= \omega_v \psi_u + \omega_u \psi_v, \\[4pt]
\psi_{vv}= -\omega_u \psi_u + \omega_v \psi_v - \frac{1}{2} N - \kappa e^{2\omega} \psi, \\[4pt]
N_{u}= - \frac{1}{2}e^{-2\omega} \psi_u, \\[4pt]
N_{v}= \frac{1}{2}e^{-2\omega} \psi_v.
\end{array}
\right.
\end{equation}

\medskip

\begin{remark}\label{rem:psianalitico}
    By \eqref{eq:initialdata} and \eqref{eq:gw}, it is clear that the immersions $\psi(u,v) = \psi(u,v;a,b,\kappa)$ depend analytically on the parameters $(a,b,\kappa) \in \mathcal{O}$ as $\mathbb{R}^4$-valued maps; see also Remark \ref{rem:omeganalitico}.
\end{remark}

\begin{definition}\label{def:Sigma}
    We denote by $\Sigma = \Sigma(a,b,\kappa)$ the surface associated to the minimal immersion $\psi(u,v)$ with initial values \eqref{eq:initialdata}.
\end{definition}
\begin{lemma}\label{lem:lineasesfericas}
   With the above notations, equation \eqref{eq:lineasesfericas} holds along a $v$-line $v \mapsto \psi(u,v)$ if and only if $\psi$ intersects a totally umbilical surface $\mathcal{Q}(u) \subset \mathbb{M}^3(\kappa)$ with constant angle $\theta = \theta(u)$ along the curvature line $v\mapsto\psi(u,v)$. More specifically, if $\kappa \neq 0$, then $\mathcal{Q}(u) = S[m(u),d(u)]$ with 
  \begin{equation}\label{eq:abmd}
      \alpha = -2\frac{ \kappa d}{\sin\theta \sqrt{\langle m, m \rangle_\kappa - \kappa d^2}}, \; \; \; \; \beta = - \frac{\cos \theta}{\sin \theta}.
  \end{equation}
    If $\kappa = 0$ and $\alpha(u) \neq 0$, then
    \begin{equation}\label{eq:abeuclid}
        |\alpha| = \frac{2\sqrt{ 1 + \beta^2}}{R}, \; \; \; \; \beta = - \frac{\cos \theta}{\sin \theta},
    \end{equation}
    where $R(u) \in (0,\infty)$ is the radius of the sphere $\mathcal{Q}(u)$. If $\alpha(u) = 0$, then $\mathcal{Q}(u)$ is a plane.
\end{lemma}
\begin{proof}
The proof for the case $\kappa \neq 0$ was already covered in \cite{CFM2}. Indeed, \eqref{eq:abmd} can be obtained from \cite[Lemma 2.1]{CFM2} by substituing $H = 0$, $c_0 = \kappa$, $Q = 1/4$, $|\hat N| = \sqrt{ \langle m, m \rangle_\kappa - \kappa d^2}$. It was also deduced there that $m(u)$ is, up to a multiplicative constant, given by
\begin{equation}\label{eq:mnoeuclideo}
    m(u) = e^{-\omega}\psi_u - \beta N -\frac{\alpha}{2} \psi.
\end{equation}
In particular $S[m,d]$ will be a sphere if and only if \begin{equation}\label{eq:condesfera}
    4\kappa(1 + \beta^2) + \alpha^2 > 0.
\end{equation}
Let us now assume that $\kappa = 0$ and \eqref{eq:lineasesfericas} holds. Suppose first that $\alpha(u) = 0$. In such a case, from \eqref{eq:lineasesfericas}, \eqref{eq:gw} we deduce that ${\bf t} = e^{-\omega}\psi_u - \beta N$ is constant along the $v$-line. In particular, $d = \langle {\bf t}, \psi \rangle_0$ satisfies $d_v = 0$, so $\psi(u,v)$ lies in a plane. Notice also that the angle between ${\bf t}$ and $\psi_u$ is constant. Conversely, assume that $\psi(u,v)$ lies in a plane, that is, there is some unit vector ${\bf t}$ such that $\langle {\bf t}, \psi \rangle_0 = d$ for some constant $d$. In particular, $\langle {\bf t}, \psi_v \rangle_0 = 0$. If we assume further that the intersection angle between $\psi$ and the plane is constant, then ${\bf t} = A e^{-\omega}\psi_u + B N$ for some constants $A,B$. Differentiating this with respect to $v$ we get $2\omega_u(u,v) = \beta e^{-\omega}$ for some $\beta$ independent of $v$, obtaining \eqref{eq:lineasesfericas}.

\medskip

Assume now that \eqref{eq:lineasesfericas} holds with $\alpha(u) \neq 0$, and define
\begin{equation}\label{eq:meuclideo}
    \hat c := -\frac{2}{\alpha}e^{-\omega}\psi_u + \frac{2\beta}{\alpha}N + \psi.
\end{equation}
Notice that $\hat c \in \mathbb{M}^3(0)$. It is straightforward to check that $\hat c_v \equiv 0$. In particular, $\|\hat c - \psi\|^2_0$ is constant along the $v$-line, so we deduce that $\psi(u,v)$ lies in a sphere whose center is $\hat c$. It can be checked that the radius and intersection angle of $\psi$ with the sphere satisfy \eqref{eq:abeuclid}.

Conversely, assume that $\psi(u,v)$ lies in a sphere of radius $R$ and intersects it with constant angle $\theta$. Let us denote by $\hat c$ the center of such sphere. Since $R^2 = \|\hat c - \psi\|^2_0$, if we differentiate with respect to $v$ we obtain $\langle \hat c - \psi, \psi_v\rangle_0 = 0$. This implies that 
$$\hat c - \psi = A e^{-\omega} \psi_u + B N$$
for some $A, B$ such that $A^2 + B^2 = R^2$. The fact that the intersection angle is constant implies that $A,B$ do not depend on $v$. Furthermore, since $\hat c_v \equiv 0$, it holds $2\omega_u = \alpha e^{\omega} + \beta e^{-\omega}$ for some pair $\alpha, \beta$ independent of $v$ and satisfying \eqref{eq:abeuclid}.
\end{proof}

\begin{remark}\label{rem:meuclideo}
    We note that the expression for $m(u)$ in \eqref{eq:mnoeuclideo} makes sense for $\kappa = 0$, and in fact it is just a rescaling of the center map $\hat c(u)$ of the (Euclidean) spheres, since $m(u) = -\frac{\alpha}{2}\hat c(u)$.
\end{remark}
We now highlight the following fact about the map $m(u)$:

\begin{proposition}\label{pro:centroplano}
    For any $(a,b,\kappa) \in \mathcal{O}$, the map $m(u)$ in \eqref{eq:mnoeuclideo} belongs to the 2-dimensional subspace $\mathcal{P} = \mathcal{P}(a,b,\kappa) \subset \mathbb{R}^4_\kappa$ given by $\mathcal{P} = \text{\em span}\{m(0),m'(0)\}$.
\end{proposition}
\begin{proof}
The result is a consequence of the fact that $m''$ is proportional to $m$. Indeed, a long but straightforward computation using \eqref{eq:alphabeta}, \eqref{eq:omegau}, \eqref{eq:omegav} and \eqref{eq:gw} shows that
\begin{equation}\label{eq:muu}
    m'' = (\hat a - 2 \alpha \beta)m.
\end{equation}

Notice also that the vectors $m(0)$ and $m'(0)$ are linearly independent: indeed,
\begin{equation}\label{eq:m0mu0}
\begin{aligned}
    m(0) &= e^{-\omega(0,0)} \psi_u(0,0) = {\bf e}_3, \\
    m'(0) &= Ae^{-\omega(0,0)}\psi_u(0,0) + BN(0,0) + C\psi(0,0) = B{\bf e}_1 + C{\bf e}_4,
\end{aligned}
\end{equation}
where $A = 0$ and $B, C$ are given by
\begin{equation}\label{eq:BC}
\begin{aligned}
    B &= \frac{1}{4}\left(a - \frac{1}{a} + \mathcal{B}\right), \\
    C &= \frac{1}{8}\left(4\kappa \left(a - \frac{1}{a}\right)  - \mathcal{B}\right).
    \end{aligned}
\end{equation}
Observe that $B > 0$ since $a,b \geq 1$ and $4|\kappa| < 1$, so the vectors $m'(0)$, $m(0)$ are not collinear. Let $\mathcal{P}$ be then the plane spanned by $\{m(0), m'(0)\}$. Now, we define $f(u)$ as the orthogonal projection of $m(u)$ in $\mathcal{P}$ with respect to the canonical Euclidean metric in $\mathbb{R}^4$. By \eqref{eq:muu}, it holds $f'' = (\hat a - 2 \alpha \beta) f$, and moreover $f(0) = m(0)$, $f'(0) = m'(0)$, since $m(0), m'(0) \in \mathcal{P}$ by construction. By uniqueness of solutions to ordinary differential equations, we deduce $m(u) \equiv f(u)$ for all $u$, so necessarily $m(u) \in \mathcal{P}$.
\end{proof}

\begin{remark}\label{rem:esferas}
We will be specially interested in the situation in which the totally umbilical surfaces $\mathcal{Q}(u)$ are 2-spheres. This occurs if and only if \eqref{eq:condesfera} holds, or equivalently, if $m(u)$ admits a rescaling $c(u)$ that lies in $\mathbb{M}^3(\kappa)$. In such a case, the rescaling is given by
\begin{equation}\label{eq:definicioncentro}
    c(u):= \varepsilon\frac{-2e^{-\omega}\psi_u + 2\beta N + \alpha \psi}{\sqrt{4\kappa(1 + \beta^2) + \alpha^2}},
\end{equation}
where $\varepsilon = \pm 1$. In order to ensure that $c(u)$ belongs to $\mathbb{M}^3(\kappa)$ for $\kappa \leq 0$, we further need to impose $\varepsilon = \text{sgn}(\alpha)$. It is possible to check that $c(u)$ is precisely the geodesic center of $\mathcal{Q}(u)$. In particular, in the Euclidean case $c(u)$ reduces to the center $\hat c(u)$ obtained in \eqref{eq:meuclideo}.
\end{remark}

\begin{remark}\label{rem:kappaneg}
Since $\alpha, \beta, \omega, \psi, \psi_u, N$, depend analitically on $(u;a,b,\kappa)$ (see Remarks \ref{rem:omeganalitico} and \ref{rem:psianalitico}), the center $c(u;a,b,\kappa)$ defined in \eqref{eq:definicioncentro} is also analytic as a $\mathbb{R}^4$-valued map provided that $\alpha(u) \neq 0$ and \eqref{eq:condesfera} holds. Observe further that if $\kappa \leq 0$, \eqref{eq:condesfera} automatically implies $\alpha(u) \neq 0$.
\end{remark}

We note that if $\mathcal{Q}(u)$ is a 2-sphere, then its geodesic center $c(u)$ must belong to both $\mathbb{M}^3(\kappa)$ and $\mathcal{P}$, so in particular the set $\mathbb{M}^3(\kappa) \cap \mathcal{P}$ is not empty. In the next Lemma we explore the properties of this intersection.

\begin{lemma}\label{lem:geodesicaPi}
    Let $\mathcal{P}$ be the plane in Proposition \ref{pro:centroplano} and ${\bf S}_\kappa$ be the totally geodesic surface
\begin{equation}\label{eq:mathcalS}
        {{\bf S}_\kappa}:= \mathbb{M}^3(\kappa) \cap \{x_3 = 0\}.
    \end{equation}
    Then, ${\mathcal{L}_\kappa}:= \mathbb{M}^3(\kappa) \cap \mathcal{P}$ is a geodesic of $\mathbb{M}^3(\kappa)$ if and only if $\kappa \geq 0$ or $\kappa < 0$ and $\mathcal{P}$ is timelike. In such a case, if $\kappa \leq 0$, then ${\mathcal{L}_\kappa} \cap {{\bf S}_\kappa}$ is a single point $p$, while for $\kappa > 0$ this intersection consists of two antipodal points $\{p,-p\}$. Moreover, ${\mathcal{L}_\kappa}$ is orthogonal to ${{\bf S}_\kappa}$ at $p$, and
    \begin{equation}\label{eq:interseccionL}
        p = -\frac{B {\bf e}_1 + C{\bf e}_4}{C^2 + \kappa B^2},
    \end{equation}
    where $B,C$ are the constants in \eqref{eq:BC}. 
\end{lemma}
\begin{proof}
    By \eqref{eq:m0mu0}, we can parametrize the plane $\mathcal{P}$ as
\begin{equation}\label{eq:Pi}
    \mathcal{P} = \{v_1 {\bf e}_3 + v_2\left(B{\bf e}_1 +C{\bf e}_4\right) \; : \; v_1,v_2 \in \mathbb{R}\},
\end{equation}
with $B,C$ as in \eqref{eq:BC}. The assumption that $\mathcal{P}$ is timelike for $\kappa < 0$ is equivalent to
\begin{equation}\label{eq:CkB}
    C^2 + \kappa B^2 > 0.
\end{equation} 
Notice further that if $\kappa = 0$, then $C = -b/8 \neq 0$. We deduce that if either $\kappa \geq 0$ or $\kappa < 0$ and \eqref{eq:CkB} holds, the point $p$ in \eqref{eq:interseccionL} is well defined and belongs to $\mathbb{M}^3(\kappa) \cap \mathcal{P}$. In such a case, the (non-empty) set ${\mathcal{L}_\kappa}= \mathcal{P} \cap \mathbb{M}^3(\kappa)$ must be a geodesic of $\mathbb{M}^3(\kappa)$, since $\M^3(\kappa)$ will be either a 3-hyperboloid, a 3-sphere or an affine hyperplane and $\mathcal{P}$ is a 2-dimensional linear subspace. Reciprocally, assume that $\mathcal{L}_\kappa = \mathcal{P} \cap \mathbb{M}^3(\kappa)$ is a geodesic. If $\kappa \geq 0$, we have nothing to prove, while if $\kappa < 0$, we have to show that $\mathcal{P}$ is timelike. This is immediate: indeed, any $p_0 \in \mathcal{L}_\kappa \subset \mathcal{P}$ must be timelike, since $p_0 \in \mathbb{M}^3(\kappa)$. Hence, $\mathcal{P}$ is timelike.

\medskip

Finally, notice that ${\mathcal{L}_\kappa}\cap {{\bf S}_\kappa} = \mathcal{P} \cap \{x_3 = 0 \} \cap \mathbb{M}^3(\kappa)$. By \eqref{eq:Pi}, it is straightforward that ${\mathcal{L}_\kappa}\cap {{\bf S}_\kappa}$ is either the single point $p$ in \eqref{eq:interseccionL} for $\kappa \leq 0$ or the antipodal pair $\{p,-p\}$ when $\kappa > 0$. Furthermore, observe that the vector ${\bf e}_3$ is tangent to both $\mathcal{P}$ and $\mathbb{M}^3(\kappa)$ at the point $p$, so in particular ${\bf e}_3$ is tangent to ${\mathcal{L}_\kappa}$ at $p$. However, ${\bf e}_3$ is orthogonal to ${{\bf S}_\kappa}$, so ${\mathcal{L}_\kappa}$ is orthogonal to ${{\bf S}_\kappa}$ at $p$.
\end{proof}

\section{Geometric properties of the family of immersions}\label{sec:geometria}
Given $(a,b,\kappa) \in \mathcal{O}$, where the set $\mathcal{O}$ was defined in \eqref{eq:mO}, let $\Sigma= \Sigma(a,b,\kappa)$ be the surface in Definition \ref{def:Sigma}. In this section, we aim to determine the conditions under which $\Sigma$ contains a compact annulus $\Sigma_0$ whose boundary components intersect a 2-sphere orthogonally. This objective is attained in Theorem \ref{thm:simetrias}. To achieve this purpose, we will follow a similar strategy to that of \cite[Sections 4.3-4.5]{CFM2}.

\subsection{Description of the symmetries of \texorpdfstring{$\Sigma$}{}} 
Let $\psi(u,v)$ be the immersion associated to $\Sigma(a,b,\kappa)$. The first and second fundamental forms of $\psi(u,v)$ are given by \eqref{eq:III}, and they allow us to detect the symmetries of $\psi(u,v)$ via the study of the function $\omega(u,v)$.

\begin{lemma}\label{lem:simetriaS}
    Let $\Psi_{{\bf S}_\kappa}: \mathbb{M}^3(\kappa) \to \mathbb{M}^3(\kappa)$ denote the reflection with respect to the totally geodesic surface ${{\bf S}_\kappa}$ in \eqref{eq:mathcalS}. Then $\Sigma$ is invariant under $\Psi_{{\bf S}_\kappa}$; specifically,
    $$\psi(u,v) = \Psi_{{\bf S}_\kappa}(\psi(-u,v)).$$
    In particular, $\Sigma$ meets ${{\bf S}_\kappa}$ orthogonally along the $v$-line $v \mapsto \psi(0,v)$, so that this curve is a geodesic of $\Sigma$.
\end{lemma}
\begin{proof}
    It is an immediate consequence of \eqref{eq:III}, \eqref{eq:simetriau0}, \eqref{eq:initialdata} and the fundamental theorem of surfaces.
\end{proof}

In a similar way, from Lemma \ref{lem:sigma} we deduce the existence of a broader set of symmetries.

\begin{proposition}\label{pro:simetrias}
    For each $j \in \mathbb{Z}$, let $\nu_j = \psi_v(0,j\sigma)$, where $\sigma$ is defined in Lemma \ref{lem:sigma}, and define $\Omega_j \subset \mathbb{M}^3(\kappa)$ as the unique totally geodesic surface orthogonal to $\nu_j$ at the point $\psi(0,j\sigma)$. Then,
    \begin{enumerate}
        \item Each $\Omega_j$ is orthogonal to the totally geodesic surface ${{\bf S}_\kappa}$ in \eqref{eq:mathcalS}.
        \item Let $\Psi_j$ be the reflection with respect to $\Omega_j$. Then,
        $$\psi(u, v + j \sigma) = \Psi_j(\psi(u,j \sigma - v)).$$
        In particular, the $u$-curve $u \mapsto \psi(u,j \sigma)$ lies in $\Omega_j$.
        \item The angle between two consecutive vectors $\nu_j$, $\nu_{j + 1}$ does not depend on $j$.
        \item Assume that the vectors $\nu_0$, $\nu_1$ satisfy
        \begin{equation}\label{eq:angulonu}
        \langle\nu_0, \nu_0 \rangle_\kappa\langle\nu_1, \nu_1 \rangle_\kappa > \langle\nu_0, \nu_1 \rangle_\kappa^2.
    \end{equation}
        Then, $\Lambda := \text{span}\{\nu_j\; : \; j \in \mathbb{Z} \} \subset \mathbb{R}_\kappa^4$ is a 2-dimensional subspace, which is spacelike if $\kappa < 0$.
        \item If \eqref{eq:angulonu} holds and $\mathcal{P}$ is the plane in Proposition \ref{pro:centroplano}, then ${\mathcal{L}_\kappa} = \mathcal{P} \cap \mathbb{M}^3(\kappa)$ is a geodesic of $\mathbb{M}^3(\kappa)$, and
    \begin{equation}\label{eq:LOmega}
        {\mathcal{L}_\kappa} = \bigcap_{j \in \mathbb{Z}} \Omega_j.
    \end{equation}
    \end{enumerate}
\end{proposition}
\begin{proof}
    To prove (1) recall that the $v$-curve $v \mapsto \psi(0,v)$ is contained in ${{\bf S}_\kappa}$. In particular, the vectors $\nu_j = \psi_v(0,j\sigma)$ are tangent to ${\bf S}_\kappa$, so the surfaces $\Omega_j$ are orthogonal to ${{\bf S}_\kappa}$. Item (2) is a direct consequence of the fundamental theorem of surfaces, the expressions of the fundamental forms \eqref{eq:III} and the periodicity of $\omega(u,v)$ in \eqref{eq:periodicidadsigma}.

    \medskip

    Let us now prove item (3). By item (2), notice that $\nu_j = - \Psi_j(\nu_j)$ and $\nu_{j + 1} = - \Psi_j(\nu_{j - 1})$. Since $\Psi_j$ is an isometry,
    $$\langle \nu_j, \nu_{j + 1}\rangle_\kappa = \langle \Psi_j(\nu_j), \Psi_j(\nu_{j + 1})\rangle_\kappa = \langle \nu_j, \nu_{j-1} \rangle_\kappa.$$
    By induction, it is immediate that $\langle \nu_j, \nu_{j + 1}\rangle_\kappa = \langle \nu_0, \nu_{1}\rangle_\kappa$ for all $j$, and so the angle between two consecutive vectors $\nu_j$, $\nu_{j+1}$ is constant.

    \medskip

Regarding item (4), let us define $\Lambda_0:= \text{span}\{\nu_0,\nu_1\} \subset \Lambda$. By \eqref{eq:angulonu}, $\Lambda_0$ is a 2-dimensional subspace which must be spacelike if $\kappa < 0$. Now, the fact that $\nu_{j + 1} = - \Psi_j(\nu_{j - 1})$ implies that $\nu_{j+1}$ can be expressed as a linear combination of $\nu_{j-1}$ and $\nu_j$. By induction, this means that actually every $\nu_j$ is spanned by $\{\nu_0,\nu_1\}$, and hence $\Lambda = \Lambda_0$.

\medskip

Let us prove item (5). We will split the proof in two cases, depending on whether $\kappa \neq 0$ or $\kappa = 0$. Assume first that $\kappa \neq 0$, and for each $j$ define $P_j \subset \mathbb{R}^4_\kappa$ as the 3-dimensional linear subspace orthogonal to $\nu_j$. We note that $\Omega_j = P_j \cap \mathbb{M}^3(\kappa)$: indeed, totally geodesic surfaces in $\mathbb{M}^3(\kappa)$ appear as intersections of $\mathbb{M}^3(\kappa)$ with linear subspaces of dimension 3, and $P_j$ is the unique such subspace orthogonal to $\nu_j$. Now, a direct computation using \eqref{eq:mnoeuclideo} shows that, for each $j$, the vector $\nu_j = \psi_v(0,j\sigma)$ is orthogonal to both $m(0)$ and $m'(0)$. In particular,
$$\mathcal{P} \subseteq \Lambda^\perp = \bigcap_{j \in \mathbb{Z}} P_j.$$
Note that actually $\mathcal{P} = \Lambda^\perp = \bigcap_{j \in \mathbb{Z}} P_j$, since both $\mathcal{P}$ and $\Lambda^\perp$ are 2-dimensional subspaces. Hence, it follows that $\mathcal{P}$ is timelike in the case $\kappa < 0$. We deduce from Lemma \ref{lem:geodesicaPi} that $\mathcal{L}_\kappa \subset \mathbb{M}^3(\kappa)$ is a geodesic for every $\kappa \in \left(-\frac{1}{4},\frac{1}{4}\right)$. Now, intersecting the sets $\mathcal{P}$ and $\bigcap_{j \in \mathbb{Z}} P_j$ with $\mathbb{M}^3(\kappa)$, we see that
$${\mathcal{L}_\kappa} = \mathcal{P} \cap \mathbb{M}^3(\kappa) =\bigcap_{j \in \mathbb{Z}} \left(P_j \cap \mathbb{M}^3(\kappa) \right)  =\bigcap_{j \in \mathbb{Z}} \Omega_j,$$
as we wanted to prove.

\medskip

Finally, assume that $\kappa = 0$. By Lemma \ref{lem:geodesicaPi}, we know that the line ${\mathcal{L}_\kappa}$ is orthogonal to the plane ${{\bf S}_\kappa}$. In particular, the planes $\Omega_j$ are parallel to ${\mathcal{L}_\kappa}$. Actually, by \eqref{eq:meuclideo} it follows that the center map $c(u) \in {\mathcal{L}_\kappa}$ lies in the planes $\Omega_j$ for every $j$, so
$${\mathcal{L}_\kappa} \subseteq \bigcap_{j \in \mathbb{Z}}\Omega_j. $$
By \eqref{eq:angulonu}, the intersection $\bigcap_{j \in \mathbb{Z}}\Omega_j$ is an affine space of dimension at most one, so we deduce \eqref{eq:LOmega}.
\end{proof}

For geometrical purposes, we will be interested in those surfaces $\Sigma(a,b,\kappa)$ for which \eqref{eq:angulonu} holds, as this allows us to identify ${\mathcal{L}_\kappa}$ with the intersection of the surfaces $\Omega_j$. This motivates the following definition:

\begin{definition}\label{def:omenos}
    We define $\mathcal{O}^-\subset \mathcal{O}$ as the open subset for which the condition \eqref{eq:angulonu} holds. 
\end{definition}

\begin{remark}\label{rem:panalitico}
    It will later be shown that $\mathcal{O}^-$ is not empty, as it contains every $(a,b,\kappa) \in \mathcal{O}$ with $a = 1$; see Proposition \ref{pro:Onovacio}. In this set, the map $p = p(a,b,\kappa): \mathcal{O}^- \to \mathbb{R}^4$ in \eqref{eq:interseccionL} is well defined and analytic.
\end{remark}

\subsection{The period map}\label{sec:periodmap}
Let $(a,b,\kappa) \in \mathcal{O}^-$, and consider the immersion $\psi(u,v)$ associated to $\Sigma(a,b,\kappa)$. We want to determine the conditions for which the planar curve $v \mapsto \Gamma(v) := \psi(0,v)$ is closed, or equivalently, periodic.

By Lemma \ref{lem:geodesicaPi} and the fact that $(a,b,\kappa) \in \mathcal{O}^-$, we know that ${\mathcal{L}_\kappa} = \mathcal{P} \cap \mathbb{M}^3(\kappa)$ is a geodesic of $\mathbb{M}^3(\kappa)$ which meets ${{\bf S}_\kappa}$ in \eqref{eq:mathcalS} at the point $p$ in \eqref{eq:interseccionL}. Observe that there is a unique isometry in $\mathbb{M}^3(\kappa)$ which fixes the coordinates $x_2, x_3$ and sends $p$ to ${\bf e}_4 \in \mathbb{M}^3(\kappa)$. This will modify the initial data for $\psi(0,0)$ and $N(0,0)$ in \eqref{eq:initialdata}, but it will do so in an analytic way, since $p$ is analytic; see Remark \ref{rem:panalitico}. In any case, this isometry sends ${\mathcal{L}_\kappa}$ to the geodesic $\{x_1 = x_2 = 0\} \cap \mathbb{M}^3(\kappa) \subset \mathbb{R}^4_\kappa$.

\begin{remark}\label{rem:omenoscond}
    From now on, for every $(a,b,\kappa) \in \mathcal{O}^-$ we will make use of these new initial data for the surface $\Sigma(a,b,\kappa)$ instead of \eqref{eq:initialdata}.
\end{remark}

\medskip

We will now project the spaces $\mathbb{M}^3(\kappa)$ into $\mathbb{R}^3$ via the stereographic map $\varphi_\kappa$ from the point $-{\bf e}_4$. More precisely, for each $\kappa$, $\varphi_\kappa$ is the restriction to $\mathbb{M}^3(\kappa)$ of the map $\varphi: \mathbb{R}^4 \cap \{x_4 > -1\} \to \mathbb{R}^3$ given by
\begin{equation}\label{eq:proyestereografica}
    \varphi(x_1,x_2,x_3,x_4) := \left(\frac{2 x_1}{x_4 + 1}, \frac{2 x_2}{x_4 + 1},\frac{2 x_3}{x_4 + 1}\right).
\end{equation}
For $\kappa > 0$, $\varphi_\kappa$ sends $\mathbb{M}^3(\kappa) \setminus \{-{\bf e}_4\}$ to $\mathbb{R}^3$. For $\kappa < 0$, $\varphi_\kappa$ sends $\mathbb{M}^3(\kappa)$ to the Euclidean 3-ball $B_\kappa$ of radius $\frac{2}{\sqrt{-\kappa}}$. Finally, for $\kappa = 0$, $\varphi_0$ reduces to the projection $\varphi(x_1,x_2,x_3,1) = (x_1,x_2,x_3)$. In any case, $\varphi_\kappa$ sends the surface ${{\bf S}_\kappa}$ to the plane $\{z = 0\}$ and the geodesic ${\mathcal{L}_\kappa}$ to the line $\{x = y = 0\}$, where $(x,y,z)$ are the usual Euclidean coordinates on $\mathbb{R}^3$.

\medskip
\begin{remark}\label{rem:varphiconforme}
    It follows from Definition \ref{def:M3k} and \eqref{eq:proyestereografica} that the projections $\varphi_\kappa$ are conformal for all $\kappa \in \left(-\frac{1}{4},\frac{1}{4}\right)$.
\end{remark}

\medskip

As mentioned before, we are interested in studying the geometry and periodicity properties of the curve $\Gamma(v)$. We will do so by using its stereographic projection into $\mathbb{R}^3$, that is, $\gamma := \varphi_\kappa \circ \Gamma$. In particular, $\gamma$ is contained in the plane $\{ z= 0\}\subset \mathbb{R}^3$. This motivates the introduction of the following {\em period map}:

\begin{definition}\label{def:periodo}
    For any $(a,b,\kappa) \in \mathcal{O}^-$, we define the \textbf{period map} as
    \begin{equation}\label{eq:defperiodo}
\Theta(a,b,\kappa) := \frac{1}{\pi}\int_0^\sigma \kappa_\gamma \|\gamma'\|dv,
\end{equation}
where $\kappa_\gamma$ denotes the Euclidean curvature of the planar curve $\gamma = \varphi_\kappa \circ \Gamma$, $\|\cdot\|$ is the Euclidean norm, and $\sigma$ is the value defined in \eqref{eq:sigma}.
\end{definition}

\begin{remark}
     Geometrically, $\pi \Theta$ measures the variation of the angle of the tangent vector $\gamma'(v)$ from $v = 0$ to $v = \sigma$. Notice also that the map $\Theta(a,b,\kappa)$ is analytic, since $\Gamma = \Gamma(v;a,b,\kappa)$ and $\sigma = \sigma(a,b,\kappa)$ depend analytically on these parameters.
\end{remark}

The following Proposition shows that the curve $\Gamma(v)$ (and hence, $\gamma(v)$) is closed precisely when the period $\Theta$ is rational:

\begin{proposition}\label{pro:simetriasGamma}
  Let $(a,b,\kappa) \in \mathcal{O}^-$ such that $\Theta(a,b,\kappa) = m/n \in \mathbb{Q}$, where $m$ and $n$ are coprime integers. Then, the following conditions hold:
  \begin{enumerate}
      \item The curve $\Gamma(v) = \psi(0,v)$ is a closed curve satisfying $\Gamma(v + 2n \sigma) = \Gamma(v)$.
      \item The rotation index of $\Gamma:[0,2n\sigma] \mapsto \mathbb{R}^2$ is $m$. 
      \item If $a > 1$, $\Gamma$ has a dihedral symmetry group $D_n$ of order $2n$. 
  \end{enumerate}
\end{proposition}
\begin{proof}
The proof is analogous to that of \cite[Proposition 4.10]{CFM2}. More precisely, the spaces $\mathbb{M}^3(\varepsilon)$ in \cite{CFM2} are substituted by $\mathbb{M}^3(\kappa)$, and we use Proposition \ref{pro:simetrias} instead of \cite[Proposition 4.5]{CFM2}. Also, the curvature of $\Gamma(v)$ in ${\bf S}_\kappa$ is in this case
$$\kappa_\Gamma(v) = \kappa_2(0,v) = -\frac{1}{2}e^{-2\omega(0,v)}.$$
In this way, if $a > 1$, then $\kappa_\Gamma$ has critical points only at $v = j\sigma$, $j \in \mathbb{Z}$, since these are the critical points of $x(v) = e^{\omega(0,v)}$; see the proof of Lemma \ref{lem:sigma}.
\end{proof}

\begin{remark}\label{rem:lineascerradas}
If $\Theta(a,b,\kappa) = m/n$ is a rational number, then not only $\Gamma(v) = \psi(0,v)$, but also any $v$-line $v \mapsto \psi(u_0,v)$ is $2n\sigma$-periodic.
\end{remark}

\subsection{Constructing minimal annuli}

We will now use the results of Propositions \ref{pro:simetrias} and \ref{pro:simetriasGamma} to construct a set of compact minimal annuli in $\mathbb{M}^3(\kappa)$ and determine their symmetries. We introduce first the following definition:

\begin{definition}\label{def:Sigma_0}
    For any $(a,b,\kappa) \in \mathcal{O}$ and $u_0 > 0$, we define $\Sigma_0(u_0;a,b,\kappa)$ as the restriction of the immersion $\psi(u,v)$ to the strip $[-u_0,u_0] \times \mathbb{R}$.
\end{definition}

\begin{remark}\label{rem:anillos}
    By Remark \ref{rem:lineascerradas}, if $\Theta(a,b,\kappa) = m/n \in \mathbb{Q}$, then $\Sigma_0$ is a compact minimal annulus under the identification $(u,v) \sim (u,v + 2n\sigma)$.
\end{remark}

We now present the main result of this section. 

\begin{theorem}\label{thm:simetrias}
    Let $(a,b,\kappa) \in \mathcal{O}^-$ so that $\Theta(a,b,\kappa) = m/n \in \mathbb{Q}$, where $m, n$ are coprime integers. Then, for any $u_0 > 0$, the minimal annulus $\Sigma_0(u_0;a,b,\kappa)$ in Definition \ref{def:Sigma_0} satisfies:
    \begin{enumerate}
        \item $\Sigma_0$ is invariant under the reflection with respect to the totally geodesic surface ${\bf S}_\kappa \subset \mathbb{M}^3(\kappa)$ in \eqref{eq:mathcalS} and also with respect to $n$ totally geodesic surfaces $\Omega_j \subset \mathbb{M}^3(\kappa)$ which meet equiangularly along the geodesic ${\mathcal{L}_\kappa}$ of Lemma \ref{lem:geodesicaPi}.
        \item The boundary components of $\Sigma_0$, given by the $v$-lines $\psi(u_0,v)$ and $\psi(-u_0,v)$, intersect the totally umbilical surfaces $\mathcal{Q}(u_0)$, $\mathcal{Q}(-u_0) \subset \mathbb{M}^3(\kappa)$ at a constant angle $\theta(u_0)$. Moreover, $\mathcal{Q}(u_0) = \Psi_{{\bf S}_\kappa}\left(\mathcal{Q}(-u_0)\right)$, where $\Psi_{{\bf S}_\kappa}$ denotes the reflection with respect to ${{\bf S}_\kappa}$.
        \item Assume that $u_0 > 0$ satisfies $\beta(u_0) = 0$. Then, $\Sigma_0$ meets the surfaces $\mathcal{Q}(u_0), \mathcal{Q}(-u_0)$ orthogonally.
        \item Assume that $m_3(u_0) = 0$, where $m_3(u)$ denotes the third coordinate of the map $m(u)$ in \eqref{eq:mnoeuclideo}. Then, $\mathcal{Q}(u_0) = \mathcal{Q}(-u_0)$, and in fact these surfaces are 2-spheres centered at ${\bf e}_4 \in \mathbb{M}^3(\kappa)$.
        \item If $a > 1$, $\Sigma_0$ has a prismatic symmetry group of order $4n$, that is, $D_n \times \mathbb{Z}_2$. This group is generated by the reflections described in item (1). In particular, $\Sigma_0$ is not rotational.
    \end{enumerate}
\end{theorem}
\begin{proof}
    Item (1) follows from Lemma \ref{lem:simetriaS} and Proposition \ref{pro:simetrias}. Similarly, items (2) and (3) are immediate by Lemma \ref{lem:lineasesfericas}. Let us now prove item (4). The fact that $m_3(u_0) = 0$ implies that $$m(u_0) \in \mathcal{P} \cap \{x_3 = 0\} = \{x_1 = x_2 = x_3 = 0\}.$$
    In particular, we can rescale $m(u_0)$ so that it becomes ${\bf e}_4 \in \mathbb{M}^3(\kappa)$. By Remark \ref{rem:esferas}, the surface $\mathcal{Q}(u_0)$ is a 2-sphere whose geodesic center is ${\bf e}_4$. Hence, $\mathcal{Q}(u_0)$ is invariant under the reflection $\Psi_{{\bf S}_\kappa}$ by the totally geodesic surface ${{\bf S}_\kappa}$. By item (2), $\mathcal{Q}(-u_0) = \Psi_{{\bf S}_\kappa}\left(\mathcal{Q}(u_0)\right) = \mathcal{Q}(u_0)$, as claimed.

    \medskip

    The proof of item (5) is analogous to that of \cite[Item (5), Theorem 4.12]{CFM2}, so we merely sketch the idea. Assume that $\Phi'$ is an isometry of $\mathbb{M}^3(\kappa)$ that leaves $\Sigma_0$ invariant. This isometry should also leave invariant the curve $\Gamma$, and furthermore it will either fix or switch the boundary components of $\Sigma_0$. These facts and Proposition \ref{pro:simetriasGamma} allow us to conclude that $\Phi'$ must be a composition of the isometries mentioned in item (1).
\end{proof}


\begin{remark}\label{rem:FBnegativo}
    Let $(a,b,\kappa) \in \mathcal{O}^-$, $u_0 > 0$ such that $\Theta(a,b,\kappa) \in \mathbb{Q}$ and $m_3(u_0) = \beta(u_0) = 0$. If we further assume that $\kappa \leq 0$, then it is an immediate consequence of the maximum principle that the compact annulus $\Sigma_0 = \Sigma_0(u_0;a,b,\kappa)$ lies in the ball $B$ bounded by the sphere $\mathcal{Q}(u_0)$, so it is actually free boundary in $B$.
\end{remark}

\subsection{The geometry of the immersions \texorpdfstring{$\Sigma(1,b,\kappa)$}{}}\label{sec:1bk}

We conclude this section by studying the surfaces associated to the parameters $(a,b,\kappa) \in \mathcal{O}$ in the special case $a = 1$. Recall that in this situation the function $x(v)$ in \eqref{eq:xvintegral} is constant, $x(v) \equiv 1$, and hence the metric $\omega$ only depends on $u$; see Remark \ref{rem:omegaa1}. Geometrically, this implies that $\Sigma$ is invariant under a 1-parameter group $\mathcal{G}$ of orientation-preserving ambient isometries of $\mathbb{M}^3(\kappa)$. In particular, the $v$-line $\Gamma(v)= \psi(0,v)$ corresponds with the orbit of $\psi(0,0)$ under the action of $\mathcal{G}$. This curve is contained in the totally geodesic surface ${{\bf S}_\kappa}$ in \eqref{eq:mathcalS}. The geodesic curvature of $\Gamma$ in ${{\bf S}_\kappa}$ is constant and coincides with $\kappa_2 \equiv -\frac{1}{2}$, as $\Sigma(1,b,\kappa)$ meets $\mathbf{S}_\kappa$ orthogonally; see \eqref{eq:kappa} and Lemma \ref{lem:simetriaS}.
Since ${{\bf S}_\kappa}$ is a surface of constant curvature $\kappa$, the curve $\Gamma$ will be a circle (and hence, bounded) provided that $\kappa > -\frac{1}{4}$. This property holds by definition of the set $\mathcal{O}$; see \eqref{eq:mO}. These facts imply that $\mathcal{G}$ must be a rotation group in $\mathbb{M}^3(\kappa)$, and so the immersions $\Sigma(1,b,\kappa)$ are minimal surfaces of revolution. The profile curve of these examples is given by $u \mapsto \psi(u,0)$, while $v$ corresponds with the rotation parameter. The {\em rotation axis}, which we denote by $\Upsilon \subset \mathbb{M}^3(\kappa)$, is a geodesic pointwise fixed by any isometry of $\mathcal{G}$. This geodesic coincides with the set ${\mathcal{L}_\kappa}$ introduced in Lemma \ref{lem:geodesicaPi}:

\begin{proposition}\label{pro:upsilonL}
    For any $(1,b,\kappa) \in \mathcal{O}$, the set ${\mathcal{L}_\kappa}$ in Lemma \ref{lem:geodesicaPi} is a geodesic. Moreover, ${\mathcal{L}_\kappa}$ coincides with the rotation axis $\Upsilon$ of the surface $\Sigma(1,b,\kappa)$.
\end{proposition}
\begin{proof}
    By Lemma \ref{lem:geodesicaPi}, we know that ${\mathcal{L}_\kappa}$ is a geodesic if $\kappa \geq 0$, while for $\kappa < 0$ we need to check that the plane $\mathcal{P}$ is timelike. This is immediate by \eqref{eq:Pi} since the vector $B{\bf e}_1 + C{\bf e}_4 = \frac{\mathcal{B}}{4}\left({\bf e}_1 - \frac{1}{2}{\bf e}_4\right) \in \mathcal{P}$ is timelike; see also \eqref{eq:metricaespacial}.

    \medskip

    Let us now prove that ${\mathcal{L}_\kappa} = \Upsilon$. First, notice that the $v$-curve $\Gamma(v) = \psi(0,v)$ parametrizes a circle in ${{\bf S}_\kappa} \subset \mathbb{M}^3(\kappa)$, so the map $v \mapsto \psi_v(0,v)$ is not constant. In particular, we can choose $\hat{\sigma} > 0$ such that the vectors $\hat{\nu}_0 := \psi_v(0,0)$ and $\hat{\nu}_1 := \psi_v(0,\hat{\sigma})$ are linearly independent. More generally, we can define the vectors $\hat{\nu}_j := \psi_v(0,j \hat{\sigma})$, $j \in \mathbb{Z}$ and consider the totally geodesic surfaces $\hat{\Omega}_j \subset \mathbb{M}^3(\kappa)$ orthogonal to $\hat{\nu}_j$ at $\psi(0,j \hat{\sigma})$. Since the metric $\omega = \omega(u)$ does not depend on $v$, we can apply the arguments of the proof of Proposition \ref{pro:simetrias} using $\hat{\sigma},\hat{\nu}_j,\hat{\Omega}_j$ instead of $\sigma,\nu_j,\Omega_j$ to prove that:
    \begin{enumerate}
        \item If $\hat{\Psi}_j$ is the reflection with respect to $\hat{\Omega}_j$, then
        $$\psi(u, v + j \hat{\sigma}) = \hat{\Psi}_j(\psi(u,j \hat{\sigma} - v)).$$
        \item The geodesic ${\mathcal{L}_\kappa}$ coincides with the intersection of the surfaces $\hat{\Omega}_j$, that is, 
    \begin{equation}\label{eq:LOmegahat}
        {\mathcal{L}_\kappa} = \bigcap_{j \in \mathbb{Z}} \hat{\Omega}_j.
    \end{equation}
    \end{enumerate}
    In particular, ${\mathcal{L}_\kappa}\subset \mathbb{M}^3(\kappa)$ is the unique geodesic fixed pointwise by the isometries $\hat{\Psi}_j$. These isometries must also fix the rotation axis $\Upsilon$ of the surface, so necessarily $\Upsilon = {\mathcal{L}_\kappa}$. 
\end{proof}


\section{Catenoids in \texorpdfstring{$\mathbb{M}^3(\kappa)$}{}}\label{sec:rotacionales}

We will devote this section to studying the set of rotational minimal surfaces associated to the parameters $(1,b,\kappa) \in \mathcal{O}$. Our main goal is to prove that every surface $\Sigma(1,b,\kappa)$ admits a certain $\tilde u > 0$ such that the piece $\Sigma_0(\tilde{u};1,b,\kappa)$ in Definition \ref{def:Sigma_0} covers an embedded, compact minimal annulus which is free boundary in some geodesic ball of $\mathbb{M}^3(\kappa)$. Moreover, we will prove that $\tilde{u} = \tilde{u}(\kappa)$ is an analytic function independent of $b$. This objective is attained in Proposition \ref{pro:catenoidesFB}.

\medskip

Do Carmo and Dajczer \cite{CD} studied rotational minimal and CMC surfaces in space forms. More recently, Fernández, Mira and the author \cite{CFM2} proved that a wide family of minimal and CMC rotational surfaces in $\mathbb{S}^3 = \mathbb{M}^3(1)$ and $\mathbb{H}^3 = \mathbb{M}^3(-1)$ contain a compact, minimal annular piece which is free boundary in some geodesic ball. Similar results were obtained by several authors \cite{BPS,LX,RdO} by using different methods. A straightforward adaptation of the proof developed in \cite[Section 5]{CFM2} would show that for any $(1,b,\kappa) \subset \mathcal{O} \cap \{a = 1\}$ with $\kappa \neq 0$, there exists a value $\tilde u$ such that the annulus $\Sigma_0(\tilde u; 1,b,\kappa)$ is free boundary in a geodesic ball $B$ of $\mathbb{M}^3(\kappa)$. The Euclidean case $\kappa = 0$, although not covered in \cite{CFM2}, is immediate: indeed, the surface $\Sigma(1,b,0)$ is just a Euclidean catenoid in $\mathbb{R}^3$, and it is well known that these surfaces admit a free boundary piece. However, in the current situation we will need to further prove that $\tilde u$ depends analytically on the curvature parameter $\kappa$. This requires a different approach to the one presented in \cite[Section 5]{CFM2}, and will be achieved in Proposition \ref{pro:catenoidesFB}.

\medskip

We will first determine the geometry of the rotational surfaces $\Sigma(1,b,\kappa)$ given by the immersions $\psi(u,v) = \psi(u,v;1,b,\kappa)$. We know that these surfaces have no umbilical points (see \eqref{eq:kappa}) and that they have a geodesic $v \mapsto \psi(0,v)$ with normal curvature $\kappa_2(0,v) \equiv - \frac{1}{2}$ lying in the totally geodesic surface $\mathbf{S}_\kappa$; see Lemma \ref{lem:simetriaS} and Subsection \ref{sec:1bk}. These two properties are enough to determine $\Sigma(1,b,\kappa)$ uniquely; see \cite[Proposition 3.2]{CD}. If $\kappa > 0$, the surfaces $\Sigma(1,b,\kappa)$ admit the following parametrization in our model $\mathbb{M}^3(\kappa)$:
\begin{equation}\label{eq:paramesfera}
    \varphi(\mathfrak{s},\theta) = (x(\mathfrak{s})\cos(\theta), -x(\mathfrak{s})\sin(\theta),\sqrt{1/\kappa - x(\mathfrak{s})^2}\sin(\phi(\mathfrak{s})),\sqrt{1 - \kappa x(\mathfrak{s})^2}\cos(\phi(\mathfrak{s}))),
\end{equation}
where $x(\mathfrak{s})$ is the unique analytic, non-constant solution of
\begin{equation}\label{eq:x}
    x'^2 = \frac{h(x)}{x^2}, \; \; \; h(x):= x^2 - \kappa x^4 - \delta^2, \; \; \; \delta = \frac{2}{4\kappa + 1}, \; \; \; x(0) = \frac{2}{\sqrt{4\kappa + 1}},
\end{equation}
and 
\begin{equation}\label{eq:phiesferico}
    \phi(\mathfrak{s}) = \int_0^\mathfrak{s} \frac{\delta}{\sqrt{\kappa} x(\mathfrak{s}')(1/\kappa - x(\mathfrak{s}')^2)}d\mathfrak{s}'.
\end{equation}
Similarly, for $\kappa < 0$, $\Sigma(1,b,\kappa)$ can be parametrized as
\begin{equation}\label{eq:paramhiperb}
    \varphi(\mathfrak{s},\theta) = (x(\mathfrak{s})\cos(\theta), -x(\mathfrak{s})\sin(\theta),\sqrt{x(\mathfrak{s})^2 - 1/\kappa}\sinh(\phi(\mathfrak{s})),\sqrt{1 - \kappa x(\mathfrak{s})^2}\cosh(\phi(\mathfrak{s}))),
\end{equation}
where $x(\mathfrak{s})$ is the unique non-constant analytic solution of \eqref{eq:x} and
\begin{equation}\label{eq:phihip}
    \phi(\mathfrak{s}) = \int_0^\mathfrak{s} \frac{\delta}{\sqrt{-\kappa} x(\mathfrak{s}')(x(\mathfrak{s}')^2-1/\kappa)}d\mathfrak{s}'.
\end{equation}

Finally, for the Euclidean case $\kappa = 0$, $\Sigma(1,b,0)$ is a standard catenoid whose neck has curvature $\frac{1}{2}$. An explicit parametrization of this catenoid in $\mathbb{M}^3(0)$ is
\begin{equation}\label{eq:parameuclideo}
    \varphi(\mathfrak{s},\theta) = \left(\sqrt{\mathfrak{s}^2 + 4}\cos(\theta),-\sqrt{\mathfrak{s}^2 + 4}\sin(\theta), 2 \; \text{arcsinh}\left(\frac{\mathfrak{s}}{2}\right),1\right).
\end{equation}

\begin{definition}
    We will refer to the immersions \eqref{eq:paramesfera} (resp. \eqref{eq:paramhiperb}) as spherical (resp. hyperbolic) catenoids.
\end{definition}

\begin{remark}\label{rem:parametrosstheta}
For any $\kappa \in (-1/4,1/4)$, the parameter $\mathfrak{s}$ represents the arc-length of the profile curve of the rotation surface $\varphi(\mathfrak{s},\theta)$, while $\theta$ corresponds with the rotation parameter. Hence, this parametrization is by curvature lines (although not conformal), and the principal curvatures are
\begin{equation}\label{eq:curvaturasrotacionales}
    \kappa_1 = \frac{\delta}{x^2}, \; \; \; \; \; \; \; \kappa_2 = -\frac{\delta}{x^2},
\end{equation}
where $\delta$ is given in \eqref{eq:x}. The rotation axis of the surfaces is the geodesic $\Upsilon = \mathbb{M}^3(\kappa) \cap \{x_1 = x_2 = 0\}$. Moreover, $\varphi(\mathfrak{s},\theta)$ is symmetric under the reflection with respect to the totally geodesic surface ${{\bf S}_\kappa}$ in \eqref{eq:mathcalS}. This surface contains the $\theta$-curve $\varphi(0,\theta)$, which is a geodesic with constant normal curvature $\kappa_2(0,\theta) \equiv -\frac{1}{2}$.
\end{remark}

\begin{remark}\label{rem:independenciab}
    Observe that the parametrizations $\varphi(\mathfrak{s},\theta)$ in \eqref{eq:paramesfera}, \eqref{eq:paramhiperb}, \eqref{eq:parameuclideo} of $\Sigma(1,b,\kappa)$ are independent of the parameter $b$. This also happens for the conformal parametrizations $\psi(u,v)$ of these examples, as the metrics $\omega(u)$ only depend on $\kappa$; see Remark \ref{rem:omegaa1}.
\end{remark}

So far, we have given two different parametrizations of each surface $\Sigma(1,b,\kappa)$, namely $\varphi(\mathfrak{s},\theta)$ and $\psi(u,v)$. We determine next the relation between their parameters:

\begin{lemma}\label{lem:stuv}
    The parameters $(\mathfrak{s},\theta)$ and $(u,v)$ are related by a change of variables $\mathfrak{s} = \mathfrak{s}(u)$, $\theta = \theta(v)$. Moreover, $\mathfrak{s}'(u) = e^{\omega(u)}$ and $\mathfrak{s}(0) = 0$.
\end{lemma}
\begin{proof}
    By the discussion in Section \ref{sec:1bk} and Remark \ref{rem:parametrosstheta}, we know that the $\mathfrak{s}$-curves in $\varphi(\mathfrak{s},\theta)$ and the $u$-curves in $\psi(u,v)$ correspond with the profile curves of the rotational surfaces, so $\mathfrak{s} = \mathfrak{s}(u)$. Moreover, by \eqref{eq:III} and since $\mathfrak{s}$ is the arc-length parameter, $\mathfrak{s}'(u) = e^{\omega(u)}$. Similarly, $\theta$ and $v$ both represent the rotation parameter, so $\theta = \theta(v)$. Finally, the equality $\mathfrak{s}(0) = 0$ follows from the fact that the curvatures of the $\mathfrak{s}$-curve $\varphi(\mathfrak{s},0)$ and the $u$-curve $\psi(u,0)$ both attain their maximum value at the point $u = \mathfrak{s} = 0$; this follows from \eqref{eq:x}, \eqref{eq:curvaturasrotacionales}, \eqref{eq:kappa} and Remark \ref{rem:omegaa1}.
\end{proof}

We now present the main result of this section. We first introduce the following notation.

\begin{definition}\label{def:zetau}
    Let $\psi(u,v)$ be the parametrization associated to $(1,b,\kappa) \in \mathcal{O}$. For every $u \in \mathbb{R}$, we denote by $\zeta_u$ the geodesic of $\mathbb{M}^3(\kappa)$ passing through $\psi(u,0)$ with tangent vector $\psi_u(u,0)$.
\end{definition}

\begin{proposition}\label{pro:catenoidesFB}
    For any $(1,b,\kappa) \in \mathcal{O}$, the following assertions hold:
    \begin{enumerate}
        \item There exists $\tilde u = \tilde u(\kappa)$, independent of $b$, such that $\Sigma_0(\tilde u; 1,b,\kappa)$ in Definition \ref{def:Sigma_0} is an embedded compact minimal annulus which is free boundary in some geodesic ball $B = B(\kappa) \subset \mathbb{M}^3(\kappa)$ centered at ${\bf e}_4 \in \mathbb{M}^3(\kappa)$. Moreover, $B(\kappa)$ is contained in the half-space $\{x_4 > 0\}$.
        \item There exists a neighbourhood $\mathcal{I} \subset \mathbb{R}$ of $\tilde u$ such that, for every $u \in \mathcal{I}$, the axis $\Upsilon$ and $\zeta_u$ meet at a unique point $\hat p = \hat p(u)$ in the half-space $\{x_4 > 0\}$. The map $\hat p: \mathcal{I} \to \Upsilon$ is analytic, and moreover ${\hat p}_3(\tilde u) = 0$, ${\hat p}_3'(\tilde u) > 0 $, where ${\hat p}_3$ denotes the third coordinate of $\hat p$.
        \item $\tilde u(\kappa): (-1/4,1/4) \mapsto \mathbb{R}^+$ is a real analytic function.
    \end{enumerate}
\end{proposition}
\begin{proof}
 The proof can be found in Appendix \ref{sec:appendix}.
\end{proof}

\section{The orthogonality condition}\label{sec:tau}

Let $(a,b,\kappa) \in \mathcal{O}$. Our goal in this section is to determine whether for some $u_0> 0$ the piece $\Sigma_0(u_0;a,b,\kappa)$ in Definition \ref{def:Sigma_0} meets the totally umbilical surfaces $\mathcal{Q}(-u_0), \mathcal{Q}(u_0)$ orthogonally along its boundary components. According to Theorem \ref{thm:simetrias}, this will happen if the function $\beta(u) = \beta(u;a,b,\kappa)$ in \eqref{eq:alphabeta} vanishes at $u = u_0$.

\medskip

The main result of this section is Proposition \ref{pro:tauW1}, which states that there exists a subset $\mathcal{W}_1 \subset \mathcal{O}$ and an analytic map $\tau(a,b,\kappa): \mathcal{W}_1 \to \mathbb{R}^+$ such that the first positive root of $\beta(u) = \beta(u;a,b,\kappa)$ is attained at $u = \tau$.

\medskip

In order to prove the existence of the map $\tau(a,b,\kappa)$, we will study the differential system \eqref{eq:alphabeta}. This system admits a first integral that can be obtained by a classical Hamilton-Jacobi procedure; see \cite[pp. 9-14]{W}. More specifically, if $\kappa \geq 0$, consider the change of variables
\begin{equation}\label{eq:cambiost}
    \begin{aligned}
        \alpha \beta &= s + t - 2\sqrt{\kappa}, \\
        -st &= \left(\alpha/2 - \sqrt{\kappa}\beta\right)^2,
    \end{aligned}
\end{equation}
where we set $s \geq t$. We can express the pair $(s,t)$ explicitly as
\begin{equation}\label{eq:cambiostexplicito} 
    \begin{aligned}
        s  = \sqrt{\kappa} + \frac{\alpha\beta}{2} + \frac{1}{2}\sqrt{(2\sqrt{\kappa} + \alpha \beta)^2 + \left(\alpha - 2\sqrt{\kappa}\beta\right)^2}, \\
        t  = \sqrt{\kappa} + \frac{\alpha\beta}{2} - \frac{1}{2}\sqrt{(2\sqrt{\kappa} + \alpha \beta)^2 + \left(\alpha - 2\sqrt{\kappa}\beta\right)^2}.
    \end{aligned}
\end{equation}
If $(\alpha(u),\beta(u))$ is a solution of \eqref{eq:alphabeta} with initial conditions \eqref{eq:abha}, then $(s,t) = (s(\lambda),t(\lambda))$ satisfies the following first order system of differential equations:
\begin{equation}\label{eq:st}
\left\{\def\arraystretch{1.3} \begin{array}{lll} s'^2(\lambda) &=& s(s - 2\sqrt{\kappa})g(s), \\
        t'^2(\lambda) &=& t(t - 2\sqrt{\kappa})g(t),\end{array} \right.
\end{equation}
where $\lambda = \lambda(u)$ satisfies
\begin{equation}\label{eq:lambdau}
    2u'(\lambda) =s(\lambda) - t(\lambda) \geq 0
\end{equation}
and $g(x)$ is the third-order polynomial given by
\begin{equation}\label{eq:gx}
    g(x) = -x^3 + (\hat{a} + 3\sqrt{\kappa})x^2 + (\mathcal{C}_1 - 2\hat{a}\sqrt{\kappa} -2\kappa)x + \frac{1}{4}\mathcal{C}_2,
\end{equation}
where $\mathcal{C}_1,\mathcal{C}_2$ are the constants in \eqref{eq:primeraintegral1}, \eqref{eq:primeraintegral2}; see \cite[(2.25)]{W}. For every $(a,b,\kappa) \in \mathcal{O}$, the polynomial $g(x)$ has three real roots $r_1 \leq r_2 \leq 0 < r_3$ satisfying
\begin{equation}\label{eq:raicesg}
    \begin{aligned}
        r_1 & = \frac{-4\kappa + 4a\sqrt{\kappa}b - a^2b^2}{4ab} = -\frac{\left(ab -2\sqrt{\kappa}\right)^2}{4ab},\\
        r_2 & = \frac{-4a^2\kappa + 4a\sqrt{\kappa}b - b^2}{4ab} = -\frac{\left(2a\sqrt{\kappa} - b\right)^2}{4ab},\\
        r_3 & = \frac{4\kappa + 4\sqrt{\kappa}+ 1}{4} = \frac{\left(2\sqrt{\kappa} + 1\right)^2}{4}. 
    \end{aligned}
\end{equation}
An alternative version of this change of variables can also be obtained for $\kappa < 0$, although in this case the corresponding functions $(s,t)$ are expected to take complex values, so the Hamilton-Jacobi procedure does not yield significative information. For this reason, we will derive the information we need for the case $\kappa < 0$ by means of a careful analysis of the situation for $\kappa \geq 0$.

\subsection{Case \texorpdfstring{$\kappa > 0$}{}}
We set $\lambda = \lambda(u)$ in \eqref{eq:lambdau} so that $\lambda(0) = 0$. According to \eqref{eq:cambiostexplicito}, the initial conditions $\alpha(0) =\beta(0) =0$ imply that $s(0) = 2\sqrt{\kappa}$, $t(0) = 0$. Hence, the function $s(\lambda)$ will oscillate in the interval $[2 \sqrt{\kappa}, r_3]$ with a period $\mathcal{M} > 0$. The behaviour of the function $t(\lambda)$ depends on the values of the roots \eqref{eq:raicesg}:
\begin{enumerate}
    \item If $r_1 < r_2 < 0$, $t(\lambda)$ will be oscillating in $[r_2, 0]$ with a certain period $\mathcal{N} > 0$.
    \item If $r_1 = r_2 < 0$, $t(\lambda)$ decreases for all $\lambda > 0$, converging to the limit value $r_2$ as $\lambda \to \infty$.
    \item Finally, if $r_2 = 0$, then $t(\lambda) \equiv 0$.
\end{enumerate}

We emphasize that in any case, it holds $2u'(\lambda) = s(\lambda) - t(\lambda) \geq 2\sqrt{\kappa} > 0$ for every $\lambda\in \mathbb{R}$. By \eqref{eq:st}, the periods $\mathcal{M}$, $\mathcal{N}$ can be computed in terms of the following integrals, provided that $r_2 < 0$:
\begin{equation}\label{eq:MN}
\begin{aligned}
    \mathcal{M} &= 2\int_{2\sqrt{\kappa}}^{r_3}\frac{dz}{\sqrt{-z(z - 2\sqrt{\kappa})(z-r_1)(z-r_2)(z-r_3)}}, \\
    \mathcal{N} &= 2\int_{r_2}^{0}\frac{dz}{\sqrt{-z(z - 2\sqrt{\kappa})(z-r_1)(z-r_2)(z-r_3)}}.
\end{aligned}
\end{equation}
Observe that the integral for $\mathcal{M}$ always converges, while $\mathcal{N} < \infty$ if and only if $r_1 < r_2$.

\begin{lemma}\label{lem:MN}
    If $r_1 \leq  r_2 < 0$, then $\mathcal{M} < \mathcal{N}$.
\end{lemma}
\begin{proof}
If $r_1 = r_2$ then $\mathcal{N}$ diverges, and the lemma holds. Thus, we may assume that $r_1 < r_2$. Consider the function
$$f(z) = \frac{i}{\sqrt{z}\sqrt{z - 2\sqrt{\kappa}}\sqrt{z - r_1}\sqrt{z-r_2}\sqrt{z-r_3}},$$
where the square root $w \mapsto \sqrt{w}$ is defined and holomorphic on $\mathbb{C}$ minus the real interval $(-\infty, 0]$. In this way, $f(z)$ is holomorphic on the whole complex plane minus the intervals $I_1 = (-\infty,r_1]$, $I_2 = [r_2, 0]$, $I_3 = [2\sqrt{\kappa},r_3]$. Now, let $\Xi_n$ be the integration path of Figure \ref{fig:integracion1}. $\Xi_n$ can be decomposed into four pieces, namely: the path $\Xi_n^{(1)}$ around the interval $I_1$, the cycles $\Xi_n^{(2)}$ and $\Xi_n^{(3)}$ enclosing the intervals $I_2$ and $I_3$ and the circular arc $\Xi_n^{(4)}$ of radius $n$. As $n$ grows, we make the pieces $\Xi_n^{(i)}$, $i = 1,2,3$ converge to their corresponding intervals. By Cauchy's theorem,
$$\int_{\Xi_n}f(z)dz = 0.$$

\begin{figure}
\centering
\includegraphics[width=0.45\textwidth]{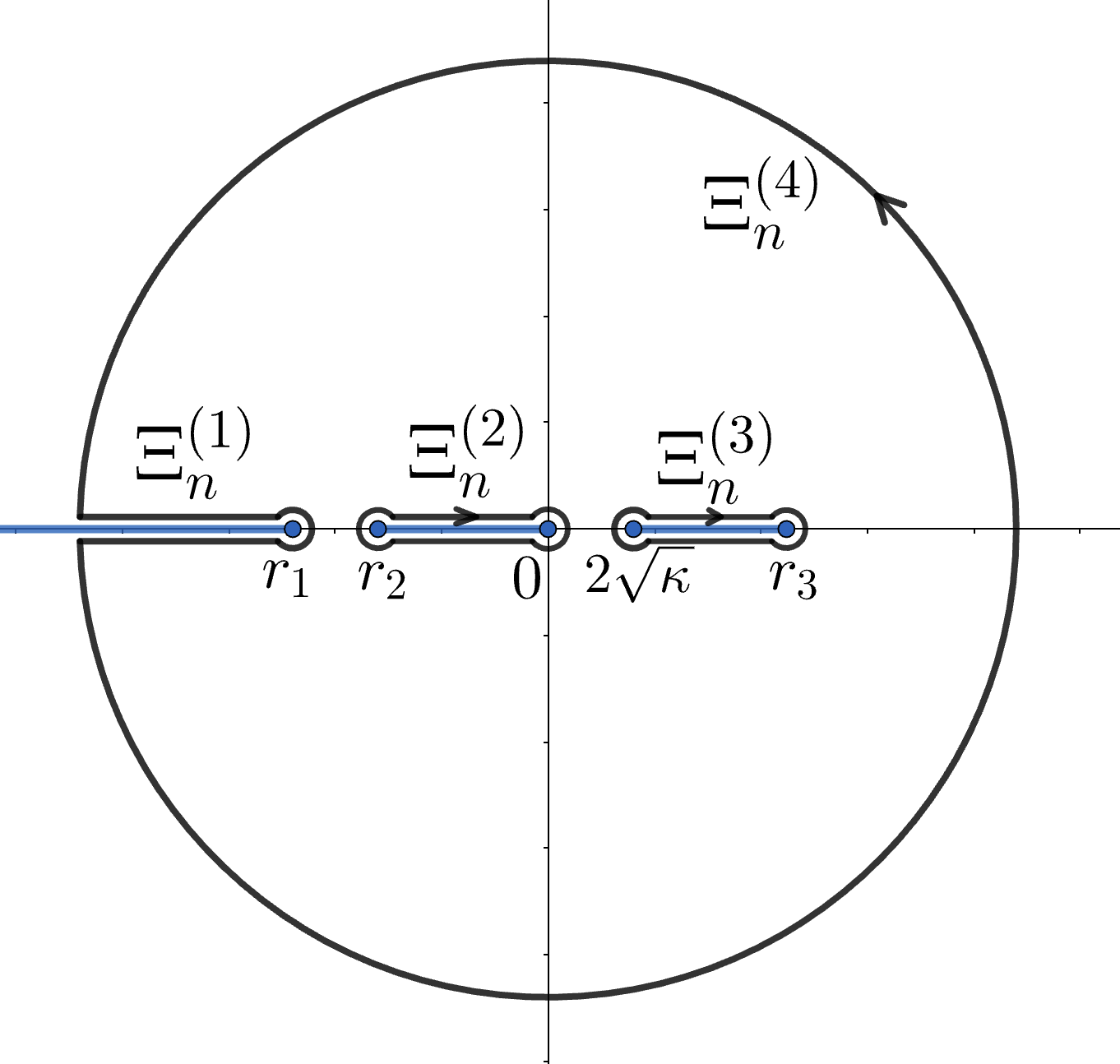}
\caption{Integration path $\Xi_n$.}\label{fig:integracion1}
\end{figure}

On the other hand, the decomposition of $\Xi_n$ in its four pieces yields
\begin{align*}
   \lim_{n \to \infty}& \int_{\Xi^{(1)}_n}f(z)dz =2\int_{-\infty}^{r_1} \frac{1}{\sqrt{-z}\sqrt{2\sqrt{\kappa} - z}\sqrt{r_1-z}\sqrt{r_2 - z}\sqrt{r_3-z}}  = \mathcal{C} > 0, \\
   \lim_{n \to \infty}& \int_{\Xi^{(2)}_n}f(z)dz =-2\int_{r_2}^{0} \frac{1}{\sqrt{-z}\sqrt{2\sqrt{\kappa} - z}\sqrt{z - r_1}\sqrt{z-r_2}\sqrt{r_3-z}}  = -\mathcal{N}, \\
   \lim_{n \to \infty}& \int_{\Xi^{(3)}_n}f(z)dz =2\int_{2\sqrt{\kappa}}^{r_3} \frac{1}{\sqrt{z}\sqrt{z - 2\sqrt{\kappa}}\sqrt{z - r_1}\sqrt{z-r_2}\sqrt{r_3-z}}  = \mathcal{M}, \\
   \lim_{n \to \infty}& \int_{\Xi^{(4)}_n}f(z)dz  = 0,
\end{align*}
for some $\mathcal{C} > 0$. In particular,
$$0 = \lim_{n \to \infty} \int_{\Xi_n} f(z)dz = \mathcal{C} + \mathcal{M} - \mathcal{N},$$
and so $\mathcal{N} = \mathcal{M} + \mathcal{C} > \mathcal{M}$, as we wanted to show.
\end{proof} 
The inequality in Lemma \ref{lem:MN} allows us to prove the following:
\begin{proposition}\label{pro:u1}
Let $(a,b,\kappa) \in \mathcal{O} \cap \{\kappa > 0\}$. Then, there exists a value $u_1 > 0$ such that the functions $y(u) := \alpha(u)/2 + \sqrt{\kappa} \beta(u)$ and $z(u) := \alpha(u)/2 - \sqrt{\kappa} \beta(u)$ satisfy:
\begin{enumerate}
    \item $y(u) > 0$ for all $u \in (0, u_1)$, and $y(u_1) = 0$.
    \item If $\mathcal{B} >  2\sqrt{\kappa}\mathcal{A}$, then $z(u)$ is positive on $(0, u_1]$, where $\mathcal{A}$ and $\mathcal{B}$ are defined in \eqref{eq:mamb}. In the limit case $\mathcal{B} =  2\sqrt{\kappa}\mathcal{A}$, $z(u) \equiv 0$.
\end{enumerate}
Additionally, the map $u_1(a, b, \kappa):\mathcal{O} \cap \{\kappa > 0\}\to \mathbb{R}^+$ is analytic.
\end{proposition}
\begin{proof}
We first rewrite system \eqref{eq:alphabeta} for $(\alpha,\beta)$ in terms of $(y,z)$ as above, as follows:
\begin{equation}\label{eq:yz}
\left\{\def\arraystretch{1.3} \begin{array}{lll} y'' & = & \hat{a} y - y\left(\sqrt{\kappa}  + \frac{1}{\sqrt{\kappa}}\left(y^2 - z^2\right)\right), \\ 
z'' & = &\hat{a} z + z\left(\sqrt{\kappa} - \frac{1}{\sqrt{\kappa}}\left(y^2 - z^2\right) \right).\end{array} \right.
\end{equation}
The initial conditions \eqref{eq:abha} are given in this situation by $y(0) = z(0) = 0$ and
\begin{equation}\label{eq:condiniyz}
    y'(0) = \left(\frac{1}{2} - \sqrt{\kappa}\right)\left(\frac{\mathcal{B} + 2\sqrt{\kappa}\mathcal{A} }{4}\right), \; \; \; \; \; z'(0) = \left(\frac{1}{2} + \sqrt{\kappa}\right)\left(\frac{ \mathcal{B} - 2\sqrt{\kappa}\mathcal{A}}{4}\right).
\end{equation}

It is easy to check that $y(u)$ (resp. $z(u)$) vanishes if and only if $s(\lambda) = 2\sqrt{\kappa}$ (resp. $t(\lambda) = 0)$. Since $y(0) = 0$, $y'(0) > 0$, there exists a first positive root $u_1$ of $y(u)$, corresponding to the point at which $\lambda(u_1) = \mathcal{M}$; see \eqref{eq:MN}.

To prove (2), observe that for $\mathcal{B} = 2\sqrt{\kappa}\mathcal{A}$ it holds $z(0) = z'(0) = 0$, and so $z \equiv 0$. On the other hand, if $\mathcal{B} >2\sqrt{\kappa}\mathcal{A}$, then $z(0) = 0$, $z'(0) > 0$ and therefore $z(u)$ is positive in $(0,u_2)$, where $u_2$ satisfies $\lambda(u_2) = \mathcal{N}$; see \eqref{eq:MN}. Since in this case $r_1 \leq r_2 < 0$, Lemma \ref{lem:MN} gives that $u_1 < u_2$ (note that $\lambda(u)$ is strictly increasing due to \eqref{eq:lambdau} and the fact that for $\kappa > 0$ it holds $s(\lambda) - t(\lambda) \geq 2\sqrt{\kappa} > 0$), which proves (2).

\medskip

Let us finally show that $u_1 = u_1(a,b,\kappa)$ is an analytic function. Since $\alpha(u) = \alpha(u;a,b,\kappa)$ and $\beta(u) = \beta(u;a,b,\kappa)$ are analytic and $u_1 > 0$ is by definition the first root of $y(u)$, it suffices to show that $y'(u_1) \neq 0$ and apply the implicit function theorem. This is immediate: assume by contradiction that $y'(u_1) = 0$. From the fact that $y(u_1) = 0$ and \eqref{eq:yz}, it would follow that $y(u) \equiv 0$, which is impossible.
\end{proof}

\medskip

The function $u_1$ will be key to finding the first positive root $\tau$ of the function $\beta(u)$. Let
\begin{equation}\label{eq:mathcalW0}
    \mathcal{W}_0 := \{(a,b,\kappa) \in \mathcal{O} \; : \; \mathcal{A} > \mathcal{B} \geq 2\sqrt{\kappa}\mathcal{A}, \;  0 <\kappa < 1/4\}.
\end{equation}
\begin{remark}\label{rem:motivacionW0}
    The restrictions on the definition of $\mathcal{W}_0$ are motivated by the following fact: if $(a,b,\kappa) \in \mathcal{W}_0$ then by \eqref{eq:abha} and \eqref{eq:condiniyz} we have that $\alpha'(0), \beta'(0)> 0$ and $z'(0) \geq 0$.
\end{remark}

\begin{proposition}\label{pro:tauW0}
    For any $(a,b, \kappa) \in \mathcal{W}_0$, the function $\beta(u)$ has a first positive root $\tau \in (0,u_1]$, where $\tau = u_1$ if and only if $\mathcal{B} = 2\sqrt{\kappa}\mathcal{A}$. Additionally, the map $\tau(a,b,\kappa):\mathcal{W}_0 \to \mathbb{R}^+$ is analytic.
\end{proposition}
\begin{proof}
   Consider the function
    \begin{equation}\label{eq:flambda}
        f(\lambda) := s(\lambda) + t(\lambda) - 2\sqrt{\kappa}.
    \end{equation}
    By \eqref{eq:cambiost}, $f(\lambda)$ coincides with the product $\alpha(u)\beta(u)$. It is then straightforward that $f(0) = 0$ and $f(\lambda) > 0$ for $\lambda > 0$ small enough, since $\alpha(u),\beta(u)$ are positive for small $u$; see Remark \ref{rem:motivacionW0}.

    \medskip

    We will now prove that $f(\lambda)$ has a first positive root $\lambda^* \in (0, \mathcal{M}]$, and denote by $\tau$ the value in the variable $u$ corresponding to $\lambda^*$. By \eqref{eq:cambiost}, we would have that either $\alpha(\tau) = 0$ or $\beta(\tau) = 0$. We will show that the latter case always holds, and so $\tau$ is the first root of $\beta(u)$. We will split our analysis into two cases, depending on whether the inequality $\mathcal{B} \geq 2\sqrt{\kappa}A$ is strict or not.

    \medskip

Assume first that $\mathcal{B} = 2\sqrt{\kappa}\mathcal{A}$. Since $a, b \geq 1$, this implies that $b = 2a\sqrt{k}$, and by \eqref{eq:raicesg}, $r_2 = 0$. The function $t(\lambda)$ is constantly zero, and so $f(\lambda)$ reduces to $s(\lambda) - 2 \sqrt{\kappa}$. Since $\mathcal{M}$ is the first positive value for which $s(\lambda) = 2\sqrt{\kappa}$, we deduce that $\lambda^* = \mathcal{M}$, so $\tau$ exists and coincides with $u_1 > 0$; see the proof of Proposition \ref{pro:u1}. Now, by \eqref{eq:cambiost}, $\alpha(\tau) \beta(\tau) = 0$ and $\alpha(\tau)/2 = \sqrt{\kappa}\beta(\tau)$, so $\alpha(\tau) = \beta(\tau) = 0$ and $\tau$ is the first positive root of both functions.

    \medskip
    
    Let us now show that $\tau$ is analytic in this case. It suffices to check that $\beta'(\tau) \neq 0$. This is immediate by the first integral \eqref{eq:primeraintegral1} and the fact that $\alpha(\tau) = \beta(\tau) = 0$, as $0 \neq \alpha'(0)\beta'(0) = \alpha'(\tau)\beta'(\tau)$, so necessarily $\beta'(\tau) \neq 0$.

    \medskip

    We now assume that $\mathcal{B} > 2\sqrt{\kappa}\mathcal{A}$. Hence, $b > 2 \sqrt{\kappa}a$, and so $r_2 < 0$. In that case, $t(\lambda)$ will take values in the interval $[r_2,0]$. If $r_1 = r_2$, $t(\lambda)$ never vanishes for $\lambda > 0$, while if $r_1 < r_2$, $t(\lambda)$ will have a first positive root at $\lambda = \mathcal{N} >  \mathcal{M}$. In any case, we see that
    $f(\mathcal{M}) = t(\mathcal{M}) < 0$, so $f(\lambda)$ must have a first positive root $\lambda^* \in (0, \mathcal{M})$. As before, this implies the existence of a first value $\tau > 0$ for which either $\alpha(\tau) = 0$ or $\beta(\tau) = 0$. Observe also that by Proposition \ref{pro:u1}, $z(u)$ is positive for all $u \in (0,u_1]$, so $0 < z(\tau) = \alpha(\tau)/2 - \sqrt{\kappa}\beta(\tau)$. Now, $\alpha(\tau), \beta(\tau) \geq 0$, so necessarily $\beta(\tau) =0$ and $\alpha(\tau) > 0$. 

    \medskip

We will finally show that $\tau$ is analytic in this case. Assume by contradiction that $\beta'(\tau) = 0$. Notice that $\beta''(\tau) = - \frac{\alpha(\tau)}{2} < 0$. This indicates that $\beta(u)$ has a local maximum at $u = \tau$, but this is impossible since $\beta(u) > \beta(\tau) = 0$ for $u \in (0,\tau)$. Hence, $\beta'(\tau) \neq 0$, and so $\tau$ is analytic in $\mathcal{W}_0$ by the implicit function theorem.
\end{proof}
\subsection{Case \texorpdfstring{$\kappa = 0$}{}}
    So far, the map $\tau(a,b,\kappa)$ is just defined on the set $\mathcal{W}_0$ in \eqref{eq:mathcalW0}. We now aim to extend this map for points $(a,b,\kappa)$ with $\kappa = 0$. To do so, we will make use again of the change of variables \eqref{eq:cambiost}. The system \eqref{eq:st} for $\kappa = 0$ was previously studied by Wente \cite[Section IV]{W} and Fernández, Hauswirth, Mira \cite[Section 3]{FHM}, and it exhibits a different behaviour with respect to the case $\kappa > 0$. We summarize it in the next Lemma:
    \begin{lemma}{\bf (\cite{FHM})}\label{lem:solucionesdegeneradas}
       Let $\kappa = 0$ and $(s(\lambda),t(\lambda))$ be a solution of \eqref{eq:st}. Assume that for some $\lambda_0 \in \mathbb{R}$ it holds $s(\lambda_0) \in (0,r_3]$, $ t(\lambda_0)\in [r_2,0)$ and $t(\lambda)$ is decreasing at $\lambda = \lambda_0$. Then,
        \begin{enumerate}
            \item The function $s(\lambda)$ takes values in $(0,r_3]$ for all $\lambda$. More specifically, $s(\lambda)$ is increasing until reaching $r_3$ and then decreases, satisfying $\lim_{\lambda \to \pm \infty} s(\lambda) = 0$.
            \item The function $t(\lambda)$ takes values in $[r_2,0)$, and $\lim_{\lambda \to -\infty} t(\lambda) = 0$.
            \item If $r_1 = r_2$, $t(\lambda)$ is decreasing for all $\lambda$, and $\lim_{\lambda \to \infty}t(\lambda) = r_2$.
            \item Otherwise, if $r_1 < r_2$, then $t(\lambda)$ reaches the minimum $r_2$ and then it is increasing, with $\lim_{\lambda \to \infty}t(\lambda) = 0$.
        \end{enumerate}
    \end{lemma}
    \begin{proof}
    Notice that the roots $r_1, r_2, r_3$ of $g(x)$ in \eqref{eq:raicesg} satisfy $r_1 \leq r_2 < 0 < r_3 = \frac{1}{4}$, where the equality $r_1 = r_2$ holds if and only if $a = 1$. The polynomial $g(x)$ in \eqref{eq:st} is nonnegative on $(-\infty,r_1] \cup [r_2,0) \cup (0,r_3]$. In particular, under the hypotheses of the Lemma, $s(\lambda)$ must take values in $(0,r_3]$ for all $\lambda$ while $t(\lambda)$ does so in $[r_2,0)$. The monotonicity and limit properties of $s(\lambda), t(\lambda)$ are straightforward from \eqref{eq:st}.
    \end{proof}

    \medskip
    
    \begin{remark}\label{rem:stabk0}
    Any solution $(s(\lambda),t(\lambda))$ of \eqref{eq:st} in the conditions of Lemma \ref{lem:solucionesdegeneradas} induces by \eqref{eq:cambiost} a solution $(\varepsilon \alpha(u),\varepsilon\beta(u))$, of \eqref{eq:alphabeta} (defined up to a sign $\varepsilon = \pm 1$). Notice, however, that this solution is not defined for all $u \in \mathbb{R}$: if $r_1 < r_2$, $u'(\lambda) = \frac{s - t}{2}$ converges to zero exponentially as $\lambda \to \pm \infty$, so the domain of the induced solution is some interval $I:= (u_m,u_M)$. By \eqref{eq:cambiost}, we deduce that $\alpha(u) \neq 0$ for $u \in I$, but $\alpha(u_m) = \alpha(u_M) = 0$. On the other hand, if $r_1 = r_2$, then $s-t$ converges to $0$ when $\lambda \to -\infty$ and to $-r_2 > 0$ as $\lambda \to \infty$. In this case, the induced solution is defined on an interval $u \in (u_m,\infty)$.

    \medskip
    
    Reciprocally, let $(\alpha(u),\beta(u)): I \to \mathbb{R}^2$ be any solution of \eqref{eq:alphabeta} with initial conditions \eqref{eq:abha}, where $I:= (0,u_M)$ is the maximal interval in which $\alpha(u)$ does not vanish. For $u_0$ small enough, $\alpha(u_0)$ and $\beta(u_0)$ are near zero, so the corresponding values $(s,t)$ obtained by the change \eqref{eq:cambiostexplicito} will be in the conditions of Lemma \ref{lem:solucionesdegeneradas}. The limit $\lambda \to -\infty$, in which $s,t$ converge to zero, corresponds with the limit value $u = 0$ in the system $(\alpha(u),\beta(u))$. If $r_1 = r_2$, then $I = (0,\infty)$, while for $r_1 < r_2$ the value $u_M < \infty$ corresponds with the first positive root of $\alpha(u)$.
    \end{remark}

Our main objective is to extend analytically the map $\tau$ defined in $\mathcal{W}_0$ \eqref{eq:mathcalW0} for values with $\kappa = 0$. As an intermediate step, we must also extend the map $u_1$ of Proposition \ref{pro:u1}.

\begin{proposition}\label{pro:u1extension}
The function $u_1$ defined in Proposition \ref{pro:u1} can be extended analytically to the set $\mathcal{O} \cap \{a > 1, \kappa = 0\}$. In particular, if $\kappa = 0$, $u_1 = u_1(a, b, 0)$ corresponds to the first root of $\alpha(u)$.
\end{proposition}
\begin{proof}
   Let $(a,b,\kappa) \in \mathcal{O} \cap \{a > 1, \kappa = 0\}$. Then, the function $y(u)$ coincides with $\alpha(u)/2$, so we just need to study the first root of $\alpha(u)$, in case it exists. Since $a > 1$, we know that $r_1 < r_2$ (see \eqref{eq:raicesg}), and by Remark \ref{rem:stabk0}, $\alpha(u)$ has a first root $u_M = u_M(a,b)$. Hence, $u_1(a,b,0)\equiv u_M(a,b)$. Let us now show that this extension is analytic. We claim that $\alpha'(u_1) \neq 0$: suppose otherwise that $\alpha'(u_1) = 0$. Since $\alpha(u_1) = 0$, using \eqref{eq:alphabeta} we deduce that $\alpha(u) \equiv 0$, which is impossible as $\alpha'(0) > 0$; see \eqref{eq:abha}. By the implicit function theorem, the map $u_1(a,b,\kappa)$ extends analytically to $\mathcal{O}\cap \{a > 1, \kappa = 0\}$, as we wanted to prove.
\end{proof}


We will now extend the map $\tau(a,b,\kappa)$ of Proposition \ref{pro:tauW0}:

\begin{proposition}\label{pro:tauW1}
    The map $\tau = \tau(a,b,\kappa) > 0$ defined in Proposition \ref{pro:tauW0} can be extended analytically to the set
    \begin{equation}\label{eq:mathcalW1}
        \mathcal{W}_1:= \left\{ (a,b,\kappa) \in \mathcal{O}\; : \; 
\mathcal{A} > \mathcal{B} \geq 2\sqrt{\kappa}\mathcal{A}, \; 0 \leq \kappa < 1/4\right\}.
    \end{equation}
    In particular, $\tau(a,b,\kappa) > 0$ and $\beta(\tau(a,b,\kappa)) = 0$ for all $(a,b,\kappa) \in \mathcal{W}_1$.
\end{proposition}
\begin{proof}
 
Note that $\mathcal{W}_1$ is exactly $\mathcal{W}_0$ in \eqref{eq:mathcalW0} plus the boundary component of $\mathcal{W}_0$ in which $\kappa = 0$. Therefore, to prove this result, we only need to show that $\tau(a, b, \kappa)$ can be analytically extended for these boundary values. We will split the proof in two cases, depending on whether $a > 1$ or $a = 1$.

Let $(a_0,b_0,0) \in \mathcal{W}_1$ with $a_0>1$. Since $\mathcal{W}_1 \subset \overline{\mathcal{W}_0}$, we can define
$$\rho_0:= \liminf_{(a,b,\kappa) \to (a_0,b_0,0)} \tau(a,b,\kappa).$$

This limit exists and it is finite, since $0 < \tau(a,b,\kappa) \leq u_1(a,b,\kappa)$ for all $(a,b,\kappa) \in \mathcal{W}_0$ and $u_1(a_0,b_0,0) < \infty$; see Proposition \ref{pro:u1extension}. Notice that $\rho_0$ is a root of $\beta(u;a_0,b_0,0)$: indeed, consider a sequence $(a_n,b_n,\kappa_n)$ converging to $(a_0,b_0,0)$ and such that $\tau_n := \tau(a_n,b_n,\kappa_n)$ converges to $\rho_0$. By continuity, $\beta(\rho_0;a_0,b_0,0) = \lim_{n\to \infty}\beta(\tau_n;a_n,b_n,\kappa_n) = 0$. We also claim that $\rho_0 > 0$: otherwise, since
$$\beta(0;a_n,b_n,\kappa_n) = \beta(\tau_n;a_n,b_n,\kappa_n) = 0,$$
there would exist some $\xi_n \in (0,\tau_n)$ such that $\beta'(\xi_n;a_n,b_n,\kappa_n) = 0$. If $\rho_0 = \lim \tau_n = 0$, then the sequence $\{\xi_n\}_n$ would also converge to zero, and hence $\beta'(0;a_0,b_0,0) =0$, which is impossible by definition of $\mathcal{W}_1$ \eqref{eq:mathcalW1} and \eqref{eq:abha}. As a consequence, $\rho_0 >0$, so $\beta(u;a_0,b_0,0)$ necessarily admits a first positive root $\tau = \tau(a_0,b_0,0) \in (0,\rho_0]$. In particular, $\tau(a_0,b_0,0) \leq \rho_0 \leq u_1(a_0,b_0,0)$.

For the case $a = 1$, we consider the function $f(\lambda)$ defined in \eqref{eq:flambda}. We recall that $f$ coincides with the product $\alpha(u)\beta(u)$ by \eqref{eq:cambiost}. Since $\alpha(0)\beta(0) = 0$ but $\alpha(u)\beta(u) > 0$ for small positive values of $u$, we deduce that $f(\lambda)$ is positive for $\lambda$ negative and large enough. In this case, it holds $r_1 = r_2 < 0$, so $\lim_{\lambda \to \infty} f(\lambda) =\lim_{\lambda \to \infty} t(\lambda) = r_2 < 0$. In particular, $f(\lambda)$ changes sign, so it must have a first root $\lambda^* \in \mathbb{R}$. Let $\tau$ be the value in the variable $u$ corresponding to $\lambda^*$. By \eqref{eq:cambiost} and the fact that $s(\lambda^*), t(\lambda^*) \neq 0$, we conclude that $\alpha(\tau) \neq 0$, so $\beta(\tau) = 0$. Moreover, since $\lambda^*$ is the first root of $f(\lambda)$, $\tau$ must be the first (positive) root of $\beta(u)$.

\medskip

Let us now show that $\tau = \tau(a,b,\kappa)$ is analytic in $\mathcal{W}_1$. We already proved this in $\mathcal{W}_0$, so we only need to consider those points in $\mathcal{W}_1$ with $\kappa = 0$. As we have done before, it suffices to check that $\beta'(\tau) \neq 0$ and apply the implicit function theorem. Assume by contradiction that $\beta'(\tau) = 0$. If $a = 1$ or $a > 1$ and $\tau < u_1$, then, by \eqref{eq:alphabeta}, $\beta''(\tau) = -\frac{\alpha(\tau)}{2} < 0$. This implies that $\beta$ has a maximum at $\tau$, which is impossible since $\beta(u) > \beta(\tau) = 0$ for all $u \in (0,\tau)$. In particular, $\beta'(\tau) < 0$. Finally, if $a >1$ and $\tau = u_1$, then $\alpha(\tau) = \beta(\tau) = 0$. Using the first integral \eqref{eq:primeraintegral1} of the system $(\alpha, \beta)$, we deduce that $\beta'(\tau) = -\beta'(0) < 0$.
\end{proof}

\begin{remark}\label{rem:taulambdastar}
    In the previous Proposition we showed that if $a = 1$ and $\kappa = 0$, then $\tau$ corresponds with the first root $\lambda^*$ of the function $f(\lambda)$ in \eqref{eq:flambda}, and that $\alpha(\tau) > 0$, $\beta'(\tau) < 0$.
\end{remark}

\begin{remark}\label{rem:tauW}
    The analyticity of $\tau(a,b,\kappa)$ on $\mathcal{W}_1$ implies that $\tau(a,b,\kappa)$ is defined and analytic on a larger open set $\mathcal{W} \subset \mathcal{O}$ which contains $\mathcal{W}_1$. In particular $\mathcal{W}$ is a neighborhood of $\mathcal{W}_1$ which extends into the region $\{\kappa < 0\} \cap \mathcal{O}$; see Figure \ref{fig:mathcalW1}.
    
    Let us also remark that $\beta(\tau(a,b,\kappa)) \equiv 0$ for all $(a,b,\kappa) \in \mathcal{W}$: indeed, this equality holds on $\mathcal{W}_1$, so by analyticity it extends to $\mathcal{W}$. Hence, by \eqref{eq:abmd}, \eqref{eq:abeuclid} we deduce that the surfaces $\Sigma_0(\tau(a,b,\kappa);a,b,\kappa)$ intersect orthogonally the spheres $\mathcal{Q}(\pm\tau)$ for all $(a,b,\kappa) \in \mathcal{W}$; see Definition \ref{def:Sigma_0}.
\end{remark}

\begin{figure}
\centering
\includegraphics[width=0.45\textwidth]{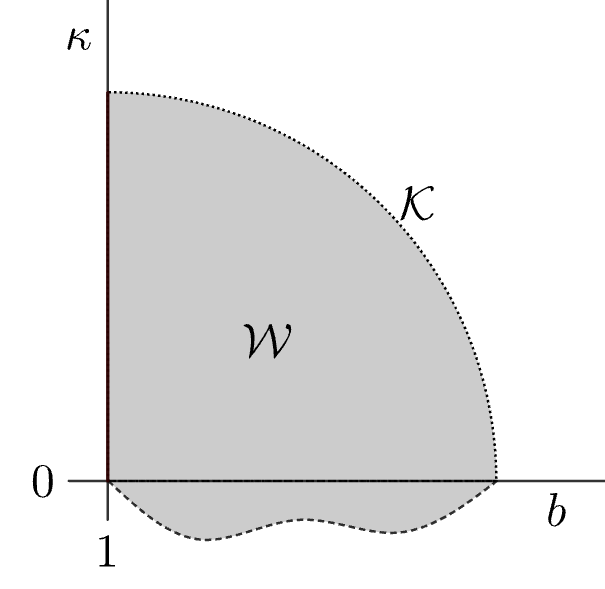}
\caption{The open set $\mathcal{W} \subset \mathcal{O}$ in Remark \ref{rem:tauW}. The boundary component of $\mathcal{W}$ in the quadrant $\{b > 1, \kappa \geq 0\}$ represents the set $\mathcal{K}$ in \eqref{eq:defK}.}\label{fig:mathcalW1}
\end{figure}

\begin{lemma}\label{lem:limitecontinuotau}
    Let $(a_0,b_0,\kappa_0) \in \mathcal{K}$, where
    \begin{equation}\label{eq:defK}
        \mathcal{K}:= \{(a,b,\kappa) \in \partial \mathcal{W}_1\; : \; \mathcal{A} = \mathcal{B}, \; 0\leq \kappa < 1/4\}.
    \end{equation}
    Then, $\lim_{n \to \infty}\tau(a_n,b_n,\kappa_n) = 0$ for any sequence $\{(a_n,b_n,\kappa_n)\}_n\subset \mathcal{W}_1$ converging to $(a_0,b_0,\kappa_0)$.
    \end{lemma}

\begin{proof}
    Let us denote by $(\alpha_0(u),\beta_0(u))$ the solution to \eqref{eq:alphabeta} associated to the parameters $(a_0,b_0,\kappa_0) \in \mathcal{K}$, and by $(\alpha_n(u),\beta_n(u))$, $n \in \mathbb{N}$, the corresponding solutions for $(a_n,b_n,\kappa_n) \in \mathcal{W}_1$. Moreover, let $\tau_n := \tau(a_n,b_n,\kappa_n) > 0$. To prove this Lemma, it suffices to show that the value $\tau_0:= \limsup_{n \to \infty} \tau_n$ is zero, so assume by contradiction that $\tau_0 > 0$. By definition, $\tau_n > 0$ is the first positive root of $\beta_n(u)$. In fact, $\beta_n(u) > 0$ for every $u \in (0,\tau_n)$, since $\beta_n'(0) > 0$. By continuity of the solutions to \eqref{eq:alphabeta} with respect to the parameters $(a,b,\kappa)$, it follows that $\beta_0(u) \geq 0$ for every $u \in [0,\tau_0]$. However, this is impossible: by \eqref{eq:alphabeta}, \eqref{eq:abha} and \eqref{eq:defK}, we have that $\beta_0(0) = \beta_0'(0) = \beta_0''(0) = 0$ but $\beta_0'''(0) = - \frac{\alpha_0'(0)}{2} < 0$, so $\beta_0(u)$ must be negative for $u > 0$ small enough.
\end{proof}

\section{The period map for the rotational immersions}\label{sec:periodo1}

In this section we will obtain an explicit expression for the the period map \eqref{eq:defperiodo} for every $(a,b,\kappa) \in \mathcal{O}$ with $a = 1$. In particular, we will show that every tuple $(1,b,\kappa)\in \mathcal{O}$ belongs to $\mathcal{O}^-$ so that $\Theta(1,b,\kappa)$ is indeed well defined; see Definition \ref{def:periodo}. To achieve this purpose, we will adapt the results in \cite[Section 7]{CFM2} to our current context.

\medskip

We first compute the value $\Theta(1,b,\kappa)$ for every $(1,b,\kappa) \in \mathcal{O}^- \cap \{a = 1\}$; see Definition \ref{def:omenos}. This will later allow us to conclude that, in fact, every tuple $(1,b,\kappa) \in \mathcal{O}$ belongs to $\mathcal{O}^-$.

\medskip

\begin{proposition}\label{pro:periodoa1}
    Let $(1,b,\kappa) \in \mathcal{O}^-$. Then,
    \begin{equation}\label{eq:Theta1bk}
    \Theta(1,b,\kappa) = -\frac{\sqrt{4\kappa + 1}}{\sqrt{4\kappa + \mathcal{B} + 1}}.
    \end{equation}
\end{proposition}
\begin{proof}
    The proof is a straightforward adaptation of \cite[Theorem 7.1]{CFM2}. Let $\psi(u,v)$ be the immersion associated to $\Sigma(1,b,\kappa)$ and $\varphi_\kappa$ be the stereographic projection from $\mathbb{M}^3(\kappa)\setminus \{-e_4\}$ to $\mathbb{R}^3$; see the discussion after Remark \ref{rem:omenoscond}. We can use the inverse map $\varphi_\kappa^{-1} = \varphi_\kappa^{-1}(x,y,z)$ to express the metric on ${{\bf S}_\kappa}\setminus \{-e_4\}$ (where ${\bf S}_\kappa$ is given by \eqref{eq:mathcalS}) as
    \begin{equation}
        \rho(dx^2 + dy^2), \; \; \; \; \rho:= \frac{1}{\left(1 + \frac{\kappa}{4}(x^2 + y^2)\right)^2},
    \end{equation}
    where $(x,y)$ are the usual Euclidean coordinates on $\mathbb{R}^2 \equiv \mathbb{R}^3\cap \{z = 0\} = \varphi_\kappa^{-1}\left({{\bf S}_\kappa}\setminus \{-e_4\}\right)$. In particular, for the curve $\gamma = \varphi_\kappa \circ \Gamma$, it holds
    \begin{equation}\label{eq:gammaprima}
        \|\gamma'\| = \frac{e^{\omega(0,v)}}{\sqrt{\rho}},
    \end{equation}
    see Definition \ref{def:periodo}. Since the projection $\varphi_\kappa$ is conformal (see Remark \ref{rem:varphiconforme}), we can relate the geodesic curvature $\kappa_\Gamma$ of $\Gamma(v)$ in ${{\bf S}_\kappa}$ with the curvature $\kappa_\gamma$ of the planar curve $\gamma(v)$ in $\mathbb{R}^2$ by
    \begin{equation}\label{eq:relacioncurvaturas}
        \kappa_\gamma = \frac{\langle\nabla \sqrt{\rho},{\bf n}\rangle}{\sqrt{\rho}} + \sqrt{\rho}\kappa_\Gamma,
    \end{equation}
    where $\nabla$, $\langle\cdot,\cdot\rangle$ denote the usual Euclidean gradient and inner product respectively, while ${\bf n}= {\bf n}(v)$ is the unit normal of $\gamma(v)$. Since $\Gamma(v) \subset {{\bf S}_\kappa}$ has constant negative geodesic curvature $\kappa_\Gamma \equiv -\frac{1}{2}$ (see Subsection \ref{sec:1bk}), we deduce that its projection $\gamma(v)$ is a negatively oriented circle of some radius $r > 0$ in $\mathbb{R}^2$. This circle must be centered at $(0,0)$: indeed, $(0,0) = \varphi_{\kappa}({\bf e}_4)$, and by Proposition \ref{pro:upsilonL}, the rotation axis $\Upsilon$ of the immersion $\psi(u,v)$ coincides with the geodesic ${\mathcal{L}_\kappa}$, which passes through ${\bf e}_4$. In particular, along the curve $\gamma(v)$ we have that ${\bf n}= \frac{1}{r}\gamma(v),$ $\kappa_\gamma = -\frac{1}{r}$ and 
    \begin{equation}
        \; \; \; \sqrt{\rho} = \frac{1}{1 + \frac{\kappa}{4} r^2}, \; \; \;  \nabla \sqrt{\rho} = -\frac{\kappa}{2(1 + \frac{\kappa}{4} r^2)^2}\gamma(v).
    \end{equation}
    It follows then from \eqref{eq:relacioncurvaturas} that $r$ satisfies
    \begin{equation}\label{eq:r}
        r = \frac{4}{1 + \sqrt{1 + 4\kappa}}.
    \end{equation}
    Hence, combining \eqref{eq:sigmaa1}, \eqref{eq:defperiodo}, \eqref{eq:gammaprima}, and having in mind that $e^{\omega(0,v)} \equiv 1$ if $a = 1$, we deduce \eqref{eq:Theta1bk}.
\end{proof}
\begin{proposition}\label{pro:Onovacio}
    Every $(a,b,\kappa) \in \mathcal{O}$ with $a = 1$ belongs to $\mathcal{O}^-$. Moreover, $\Theta(1,b,\kappa) \in \left(-\frac{\sqrt{2}}{2},0\right)$.
\end{proposition}
\begin{proof}
    Let $(1,b,\kappa) \in \mathcal{O}$. We know that $\Sigma(1,b,\kappa)$ is a rotational minimal surface whose rotation axis $\Upsilon$ coincides with the geodesic ${\mathcal{L}_\kappa}$ of Lemma \ref{lem:geodesicaPi}; see Proposition \ref{pro:upsilonL}. Following the discussion at the beginning of Section \ref{sec:periodmap}, there exists a unique isometry in $\mathbb{M}^3(\kappa)$ which fixes the coordinates $x_2,x_3$ and sends the point $p$ in \eqref{eq:interseccionL} to ${\bf e}_4$. We apply this isometry and consider the stereographic projection $\varphi$ defined in \eqref{eq:proyestereografica}. Setting $\Gamma(v) = \psi(0,v)$ and $\gamma = \varphi_\kappa \circ \Gamma$, we can define the value
    \begin{equation}\label{eq:hattheta}
        \hat{\theta} := \frac{1}{\pi}\int_{0}^\sigma\kappa_\gamma \|\gamma'\|dv,
    \end{equation}
   If $(1,b,\kappa)$ belonged to $\mathcal{O}^-$, then $\hat \theta$ would coincide with $\Theta(1,b,\kappa)$ in \eqref{eq:defperiodo}. In any case, the integral in \eqref{eq:hattheta} is well defined, and $\pi \hat{\theta}$ measures the variation of the angle of the unit normal vector of $\gamma(v)$ from $v = 0$ to $v = \sigma$. In fact, applying the same computations as in Proposition \ref{pro:periodoa1} we obtain that the value $\hat{\theta}$ also satisfies \eqref{eq:Theta1bk}. In particular, it holds $\hat \theta \in \left(-\frac{\sqrt{2}}{2},0\right)$.

\medskip

   Let us now prove that the vectors $\nu_0 = \psi_v(0,0)$, $\nu_1 = \psi_v(0,\sigma)$ satisfy \eqref{eq:angulonu}. First, we claim that $\nu_0,\nu_1$ are not collinear: otherwise, the unit normal vectors of $\gamma(v)$ at $v = 0$ and $v = \sigma$ would be collinear as well, and so $\hat{\theta}$ should be an integer. This shows \eqref{eq:angulonu} provided that $\kappa \geq 0$. If $\kappa < 0$, we also need to check that the plane $\Lambda_0:= \text{span}\{\nu_0,\nu_1\}$ is spacelike. Observe that both $\nu_0,\nu_1$ are orthogonal to the vectors $m(0),m'(0)$ which span the plane $\mathcal{P}$ in Proposition \ref{pro:centroplano}, so $\Lambda_0^\perp = \mathcal{P}$. Now, since ${\mathcal{L}_\kappa} = \Upsilon$ is a geodesic, we deduce by Lemma \ref{lem:geodesicaPi} that $\mathcal{P}$ is timelike, so $\Lambda_0$ must be spacelike.
\end{proof}

\begin{remark}\label{rem:indicecatenoides}
    Let $(1,b,\kappa) \in \mathcal{O}$ and $\Theta_0 := \Theta(1,b,\kappa)$. Assume further that $\Theta_0 \in \mathbb{Q}$ and express $\Theta_0$ as an irreducible fraction $\Theta_0 = -m/n$, $m,n \in \mathbb{N}$. Then, it follows from Remark \ref{rem:lineascerradas} that the annulus $\Sigma_0(u_0;1,b,\kappa)$ is an $m$-cover of a piece of a catenoid under the identification $(u,v) \sim (u,v + 2n\sigma)$; see Definition \ref{def:Sigma_0}.
\end{remark}

\begin{remark}\label{rem:btheta}
    Let $(1,b_0,\kappa_0) \in \mathcal{O}$ and $\Theta_0:= \Theta(1,b_0,\kappa_0)$. By \eqref{eq:Theta1bk} it follows that the partial derivative $\frac{\partial \Theta}{\partial b}(1,b_0,\kappa_0)$ is strictly positive. Since the map $\Theta(a,b,\kappa)$ is analytic, we deduce by the implicit function theorem that the level set $\Theta(a,b,\kappa) = \Theta_0$ can be expressed locally around $(1,b_0,\kappa_0)$ as a graph $b_{\Theta_0} = b_{\Theta_0}(a,\kappa)$. If $a = 1$, we can use \eqref{eq:Theta1bk} and \eqref{eq:mamb} to compute $b_{\Theta_0}$ explicitly:
   \begin{equation}\label{eq:btheta}
       b_{\Theta_0}(1,\kappa) = \frac{1}{2\Theta_0^2}\left((1 + 4\kappa)(1 - \Theta_0^2) + \sqrt{(1+4\kappa)^2(1 - \Theta_0^2)^2 - 16\Theta_0^4\kappa}\right).
   \end{equation}
\end{remark}

\section{Detecting the Euclidean free boundary catenoid}\label{sec:EuclidFB}

Let $a = 1$, $\kappa = 0$. The goal of this section is to detect some $b_0 \geq 1$ such that the surface $\Sigma_0(\tau;1,b_0,0)$ covers a Euclidean catenoid with free boundary in a ball, where $\tau := \tau(1,b_0,0)$; see Definition \ref{def:Sigma_0} and Proposition \ref{pro:tauW1}. By Proposition \ref{pro:catenoidesFB}, this happens if $\tau(1,b_0,0)$ satisfies the condition
\begin{equation}\label{eq:tautilde}
    \tau(1,b_0,0) = \tilde u(0).
\end{equation}
We recall that when $a = 1$, the surfaces $\Sigma(1,b,\kappa)$ are independent of $b$, i.e., they all represent the same catenoid in $\mathbb{M}^3(\kappa)$; see Remark \ref{rem:independenciab}. In particular, for $\kappa = 0$, $\Sigma(1,b,0)$ is a Euclidean catenoid of neck curvature $\frac{1}{2}$. This explains why $\tilde u = \tilde u (\kappa)$ in Proposition \ref{pro:catenoidesFB} does not depend on this parameter. We will devote the rest of this section to prove that
\begin{equation}\label{eq:desigualdadbstar}
    \tau(1,1,0) > \tilde u(0) > 0 = \lim_{b \to 2^-} \tau(1,b,0).
\end{equation}
In particular, this implies the existence of $b_0 \in (1,2)$ such that \eqref{eq:tautilde} holds. The limit in the right-hand side of \eqref{eq:desigualdadbstar} is immediate: indeed, $(1,2,0) \in \mathcal{O}$ belongs to the set $\mathcal{K}$ in Lemma \ref{lem:limitecontinuotau}, so $\lim_{b \to 2^-} \tau(1,b,0) = 0$. Hence, it remains to show the first inequality of \eqref{eq:desigualdadbstar}. This will be achieved in Theorem \ref{thm:radiogeometrico}, and the rest of the Section is dedicated to prove that result. It will be necessary to study the geometry of the catenoid $\Sigma(1,b,0)$ as well as the system $(\alpha(u),\beta(u))$ in \eqref{eq:alphabeta} for the parameters $(a,b,\kappa) = (1,1,0) \in \mathcal{O}$. We do this next.

\subsection{The case \texorpdfstring{$(1,1,0) \in \mathcal{O}$}{}}
 We fix the parameters $(a,b,\kappa) =(1,1,0)$, and consider the solution $(s(\lambda),t(\lambda))$ to \eqref{eq:st}. The roots of the polynomial $g(x)$ satisfy $-r_1 = -r_2 = r_3 = 1/4$, that is, 
 \begin{equation}\label{eq:geuclideo}
     g(x) = -(x + 1/4)^2(x-1/4), 
 \end{equation}
 see \eqref{eq:raicesg}. If we define the variables $S = 4s$, $T = 4t$, $\eta = 2^{-3}\lambda$, \eqref{eq:st} reduces to
 \begin{equation}\label{eq:ST}
\left\{\def\arraystretch{1.3} \begin{array}{lll} S'^2(\eta) &=& S^2(1-S)(S+1)^2, \\
T'^2(\eta) &=& T^2(1-T)(T+1)^2.\end{array} \right.
\end{equation}

This system was previously studied in \cite[Section 4]{FHM}. The following result is a straightforward consequence of \cite[Theorem 4.2]{FHM} and its proof:

\begin{theorem}\label{thm:fhm}{\bf (\cite[Theorem 4.2]{FHM}).} The function $f(\lambda) = s(\lambda) + t(\lambda)$ in \eqref{eq:flambda} has a first root $\lambda^*$ satisfying
\begin{equation}
    \mathcal{F}\left(4s(\lambda^*)\right) = -1
\end{equation}
and $s'(\lambda^*) < 0$, where $\mathcal{F}(x):= H(x)H(-x)$, being
$$H(x) := \frac{1 + \sqrt{1 - x}}{1 - \sqrt{1 - x}} 
\left( 
\frac{\sqrt{2} - \sqrt{1 - x}}{\sqrt{2} + \sqrt{1 - x}} 
\right)^{\frac{1}{\sqrt{2}}}.$$

\end{theorem}

\begin{lemma}\label{lem:alphaderivada}
    Let $(\alpha(u),\beta(u))$ be the solution to \eqref{eq:alphabeta} for $(a,b,\kappa) = (1,1,0)$, and $\tau_1 := \tau(1,1,0)$. Then, $\alpha'(\tau_1) < 0$. 
\end{lemma}
\begin{proof}
    Let $(s(\lambda),t(\lambda))$ be the solution to \eqref{eq:st} for $(1,1,0) \in \mathcal{O}$. By Remark \ref{rem:taulambdastar}, $\tau_1$ corresponds with the first root $\lambda^*$ of the function $f(\lambda) = s(\lambda) + t(\lambda)$, and $\alpha(\tau_1) > 0$. Now, by \eqref{eq:cambiost}, $\alpha^2 = -4st$, so differentiating with respect to $u$ at $\tau_1$,
    $$\alpha(\tau_1)\alpha'(\tau_1) = -2\lambda'(\tau_1)\left(s'(\lambda^*)t(\lambda^*) + s(\lambda^*)t'(\lambda^*)\right) = 2\lambda'(\tau_1)s(\lambda^*)\left(s'(\lambda^*) - t'(\lambda^*)\right),$$
    where we use $s(\lambda^*) = -t(\lambda^*)$. Since $\alpha(\tau_1),\lambda'(\tau_1), s(\lambda^*) > 0$, it suffices to prove that 
    \begin{equation}\label{eq:desigualdadstprima}
        s'(\lambda^*) - t'(\lambda^*) < 0.
    \end{equation}
    Let $s^*:= s(\lambda^*) = -t(\lambda^*)$. We know that $t(\lambda)$ is decreasing so $t'(\lambda^*) < 0$, and that $s^* \in (0,1/4]$; see Remark \ref{rem:stabk0}. By Theorem \ref{thm:fhm}, it also holds that $s'(\lambda^*) < 0$, which by \eqref{eq:st} and \eqref{eq:geuclideo} implies that $s^* \neq 1/4$. Hence, \eqref{eq:desigualdadstprima} is equivalent to
    $$1 > \frac{t'(\lambda^*)}{s'(\lambda^*)} = \frac{\sqrt{g(-s^*)}}{\sqrt{g(s^*)}}.$$
    It is immediate to check that $g(x) > g(-x)$ for all $x \in (0,1/4)$, so we deduce \eqref{eq:desigualdadstprima}.
\end{proof}

\subsection{The orthogonal radius}
As we mentioned at the beginning of this Section, for every $b \geq 1$ the surface $\Sigma = \Sigma(1,b,0)$ represents the same catenoid. In fact, the immersion $\psi(u,v;1,b,0)$ does not depend on $b$ either; see Remark \ref{rem:independenciab}. In this particular situation, we can compute $\psi(u,v)$ explicitly. First, we identify $\mathbb{M}^3(0) = \mathbb{R}^4 \cap \{x_4 = 1\}$ with $\mathbb{R}^3$. Under this identification, the rotation axis $\Upsilon$ is just the $x_3$-axis. The initial conditions in Remark \ref{rem:omenoscond} for the immersion are
\begin{equation}\label{eq:inicondicat}
    \psi(0,0) = 2{\bf e}_1, \; \; \; \; \psi_u(0,0) = {\bf e}_3,\; \; \; \; \psi_v(0,0) = -{\bf e}_2, \; \; \; \; N(0,0) = {\bf e}_1,
\end{equation}
where $\{{\bf e}_1,{\bf e}_2,{\bf e}_3\}$ denotes the canonical basis of $\mathbb{R}^3$. It is immediate by \eqref{eq:omegau2} that $e^{\omega(u)} = \cosh\left(\frac{u}{2}\right)$. Since the $u$-curve $\psi(u,0)$ is the profile curve of the catenoid and $v$ is the rotation parameter, we deduce that
\begin{equation}\label{eq:catenoide}
\begin{aligned}
\psi(u,v) &= \left(2\cosh\left(\frac{u}{2}\right)\cos\left(\frac{v}{2}\right),-2\cosh\left(\frac{u}{2}\right)\sin\left(\frac{v}{2}\right),u\right), \\
N(u,v) &= \left(\operatorname{sech}\left(\frac{u}{2}\right)\cos\left(\frac{v}{2}\right),-\operatorname{sech}\left(\frac{u}{2}\right)\sin\left(\frac{v}{2}\right),-\tanh\left(\frac{u}{2}\right)\right).
\end{aligned}
\end{equation}
We now introduce the following concept.
\begin{definition}
   For every $u_0 > 0$, we define the \textbf{orthogonal radius} $R_\perp = R_\perp(u_0) > 0$ as the radius of the unique ball $B_\perp = B_\perp(u_0) \subset \mathbb{R}^3$ which intersects the Euclidean catenoid $\psi(u,v)$ in \eqref{eq:catenoide} orthogonally along the $v$-line $\psi(u_0,v)$; see Figure \ref{fig:radiogeometrico}.
\end{definition}
\begin{remark}\label{rem:cperp}
    It follows from the definition of orthogonal radius that the centers $c_\perp(u)$ of the balls $B_\perp(u)$ must lie in the rotation axis $\Upsilon$ of the catenoid; see Figure \ref{fig:radiogeometrico}. These centers are given by
    \begin{equation}\label{eq:cperp}
        c_\perp(u) = \psi(u,0) - R_\perp(u)e^{-\omega(u)}\psi_u(u,0).
    \end{equation}
\end{remark} 
The study of the orthogonal radius will be key to prove \eqref{eq:desigualdadbstar}, as we will see in Theorem \ref{thm:radiogeometrico}.

\begin{lemma}
    The orthogonal radius satisfies the following properties:
    \begin{enumerate}
        \item $\lim_{u \to 0^+} R_{\perp}(u) = \lim_{u \to \infty} R_{\perp}(u) = \infty$.
        \item $R_\perp(u)$ is strictly decreasing in $(0,u^*]$ and increasing in $[u^*,\infty)$, where $u^* = 2\arcsinh\left(1\right)$.
        \item $u^* < \tilde u(0) < \hat u$, where $\tilde u(0)$ is given in Proposition \ref{pro:catenoidesFB}, and $\hat u =2\arcsinh\left(2\right)$.
    \end{enumerate}
\end{lemma}
\begin{figure}
\centering
\includegraphics[width=0.45\textwidth]{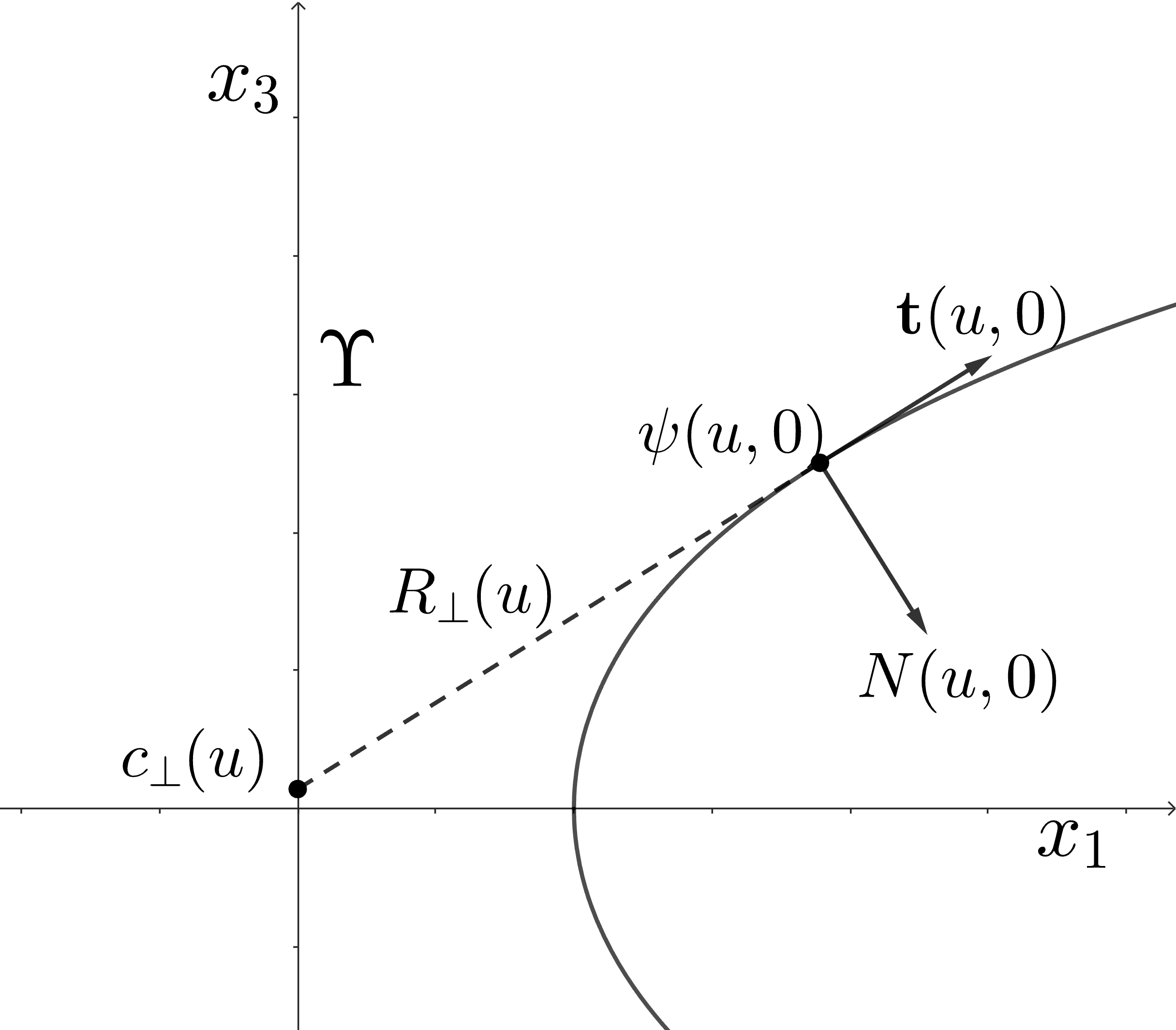}
\caption{Geometric interpretation of the orthogonal radius. Here, ${\bf t}(u,0)$ denotes the unit tangent vector $e^{-\omega(u)}\psi_u(u,0)$.}\label{fig:radiogeometrico}
\end{figure}

\begin{proof}
We first obtain an explicit expression for the orthogonal radius. Indeed, for each $u> 0$, $R_\perp(u) > 0$ is the unique value for which the right-hand side of \eqref{eq:cperp} lies in the $x_3$-axis. A straightforward computation using \eqref{eq:catenoide} allows us to deduce that
\begin{equation}\label{eq:crperp}
\begin{aligned}
    c_\perp(u) &= \left(0,0,  u-2\cotanh\left(\frac{u}{2}\right) \right), \\
    R_\perp(u) &= 2\cotanh\left(\frac{u}{2}\right)\cosh\left(\frac{u}{2}\right).
\end{aligned}
\end{equation}
Items (1) and (2) of the Lemma are then immediate. Let us prove item (3). By definition, $\tilde u(0)$ is the value in which the piece of catenoid $\Sigma_0(\tilde u(0);1,b,0)$ is free boundary in a ball; see Proposition \ref{pro:catenoidesFB}. By definition, this ball must be exactly $B_\perp(\tilde u(0))$, so necessarily $c_\perp(\tilde u(0)) = (0,0,0)$. In particular, we deduce by \eqref{eq:crperp} that $\tilde u(0)$ is a root of
\begin{equation}\label{eq:G}
    G(u):= u-2\cotanh\left(\frac{u}{2}\right).
\end{equation}
Observe that $G(u)$ is strictly increasing, so $\tilde u(0)$ is actually the unique root of this function. Since $G(u^*) = 2\arcsinh(1) - 2\sqrt{2} < 0$ and $G(\hat u) = 2\arcsinh(2) - \sqrt{5} > 0$, we deduce that $u^* < \tilde u (0) < \hat u$, proving item (3).
\end{proof}

\begin{remark}
    Since the map $R_\perp(u)$ is increasing for $u \in [u^*,\infty)$, it holds 
    \begin{equation}\label{eq:Rperp5}
        R_\perp(\tilde u(0)) < R_\perp(\hat u) = 5.
    \end{equation}
\end{remark}

\begin{corollary}\label{cor:radiogeometrico}
    Let $u_0 > 0$ such that $R_\perp(u_0) > R_\perp(\tilde u(0))$ and $R_\perp'(u_0) > 0$. Then, $u_0 > \tilde u(0)$.
\end{corollary}
\begin{proof}
    Since $R_\perp'(u_0) > 0$, we deduce by the previous Lemma that $u_0$ belongs to the interval $(u^*,\infty)$, in which $R_\perp$ is strictly increasing. The value $\tilde u(0)$ also lies in this interval, and due to the fact that $R_\perp(u_0) > R_\perp(\tilde u(0))$, we deduce $u_0 > \tilde u(0)$. 
\end{proof}

Our next goal is to show that the value $u_0 = \tau(1,1,0)$ is in the conditions of Corollary \ref{cor:radiogeometrico}, from which we will prove \eqref{eq:desigualdadbstar}.

\begin{lemma}\label{lem:derivadaRperp}
    Let $\tau_1 = \tau(1,1,0)$. Then, $R'_\perp(\tau_1) > 0$.
\end{lemma}
\begin{proof}
    Let $R = R(u)$ denote the radius of the sphere $\mathcal{Q}(u)$ in Lemma \ref{lem:lineasesfericas}. At $u = \tau_1$, the sphere $\mathcal{Q}(\tau_1)$ intersects the $v$-line $\psi(\tau_1,v)$ orthogonally, so
    \begin{equation}\label{eq:RRperptau}
        R(\tau_1) = R_\perp(\tau_1)
    \end{equation}
     by definition of orthogonal radius. We will now prove that
\begin{equation}\label{eq:comparativaR}
    R'_\perp(\tau_1) \geq R'(\tau_1) > 0,
\end{equation}
    from which the Lemma follows. 
    First, by \eqref{eq:abeuclid},
    \begin{equation}\label{eq:derivadaR}
        R'(\tau_1) =\left(\frac{2\sqrt{1 + \beta^2}}{\alpha}\right)'\bigg|_{u = \tau_1} =\left(\frac{2}{\alpha}\right)'\bigg|_{u = \tau_1} = -\frac{2\alpha'(\tau_1)}{\alpha^2(\tau_1)},
    \end{equation}
where we have used that $\beta(\tau_1) = 0$ and $\alpha(\tau_1) > 0$; see Remark \ref{rem:taulambdastar}. By Lemma \ref{lem:alphaderivada} we conclude that $R'(\tau_1) > 0$. 

\medskip

Let us now prove that $R'_\perp(\tau_1) \geq R'(\tau_1)$. To do so, for any $u> 0$ we will compute the first component of $\psi(u,0)$, i.e. $\langle \psi(u,0),{\bf e}_1\rangle$ in two different ways. Here, $\langle \cdot, \cdot \rangle$ denotes the usual Euclidean product in $\mathbb{R}^3$ and ${\bf e}_1 = (1,0,0)$. First, we use the fact that $c_\perp(u)$ lies in the rotation axis $\Upsilon$, so $\langle c_\perp(u) ,{\bf e}_1 \rangle = 0$. In particular, by \eqref{eq:catenoide} and \eqref{eq:crperp},
\begin{equation}\label{eq:calculo1}
    \langle \psi(u,0),{\bf e}_1\rangle = \langle \psi(u,0) - c_\perp(u),{\bf e}_1\rangle = R_\perp(u)\tanh\left(\frac{u}{2}\right) > 0.
\end{equation}
We can apply a similar trick by using the center map $c(u)$ in \eqref{eq:definicioncentro} instead, yielding
\begin{equation}\label{eq:calculo2}
    \langle \psi(u,0),{\bf e}_1\rangle = \langle \psi(u,0) - c(u),{\bf e}_1\rangle = \frac{2}{\alpha(u)}\tanh\left(\frac{u}{2}\right) - \frac{2\beta(u)}{\alpha(u)}\operatorname{sech}\left(\frac{u}{2}\right).
\end{equation}
Now, the expressions in \eqref{eq:calculo1} and \eqref{eq:calculo2} must coincide, so
\begin{equation}\label{eq:Rperpalpha}
    R_\perp(u) = \frac{2}{\alpha(u)} - \frac{2\beta(u)}{\alpha(u)\sinh\left(\frac{u}{2}\right)}.
\end{equation}
Since $\beta(\tau_1) =0$ but $\beta'(\tau_1) < 0$ (see Remark \ref{rem:taulambdastar}), it follows that $\beta(u) < 0$ for every $u > \tau_1$ close enough to $\tau_1$. In particular, from \eqref{eq:Rperpalpha} it holds 
$R_\perp(\tau_1) = \frac{2}{\alpha(\tau_1)}$ while $R_\perp(u) > \frac{2}{\alpha(u)}$ for $u > \tau_1$. Along with \eqref{eq:derivadaR}, this implies that
$$R_\perp'(\tau_1) \geq \left(\frac{2}{\alpha}\right)'\bigg|_{u = \tau_1} = R'(\tau_1) > 0.$$
This proves \eqref{eq:comparativaR} and the Lemma.
\end{proof}
We finally prove \eqref{eq:desigualdadbstar}, the main result of this Section.

\begin{theorem}\label{thm:radiogeometrico}
    Let $\tau_1 = \tau(1,1,0)$. Then, $\tau_1 > \tilde u(0)$.
\end{theorem}
\begin{proof}
It is an immediate consequence of Corollary \ref{cor:radiogeometrico} and Lemma \ref{lem:derivadaRperp} once we prove that $R_\perp(\tau_1) > R_\perp(\tilde u(0))$. We first note that the function $\mathcal{F}(x)$ in Theorem \ref{thm:fhm} satisfies
\begin{equation}\label{eq:igualdades}
    -1 = \mathcal{F}(4s(\lambda^*)) = \mathcal{F}(2\alpha(\tau_1)) = \mathcal{F}\left(\frac{4}{R(\tau_1)}\right) = \mathcal{F}\left(\frac{4}{R_\perp(\tau_1)}\right).
\end{equation}
The first equality follows from the statement of Theorem \ref{thm:fhm}. For the second equality, we use \eqref{eq:cambiost} and the fact that $\lambda^*$ corresponds with the value $\tau_1$ in the variable $u$; see Remark \ref{rem:taulambdastar}. Finally, the third and fourth equalities are a consequence of \eqref{eq:abeuclid}, Remark \ref{rem:taulambdastar} and \eqref{eq:RRperptau}. Now, it is not hard to prove that the function $\mathcal{F}(x)$ is strictly increasing. This implies, by \eqref{eq:Rperp5}, that
$$\mathcal{F}\left(\frac{4}{R_\perp(\tilde u(0))}\right) > \mathcal{F}\left(\frac{4}{5}\right).$$
Observe that $\mathcal{F}\left(\frac{4}{5}\right) \approx -0.862875 > -1$. Hence, by \eqref{eq:igualdades}, it follows that
$$\mathcal{F}\left(\frac{4}{R_\perp(\tilde u(0))}\right) > -1 = \mathcal{F}\left(\frac{4}{R_\perp(\tau_1)}\right).$$
By the monotonicity of $\mathcal{F}(x)$,  we deduce that $R_\perp(\tau_1) > R_\perp(\tilde u(0))$, as we wanted to prove.
\end{proof}

\begin{remark}\label{rem:mu}
Let $\hat f:\mathcal{W} \cap \{a = 1\} \to \mathbb{R}$ be the analytic map
\begin{equation}\label{eq:hatf}
    \hat f(b,\kappa) := \tau(1,b,\kappa) - \tilde u(\kappa),
\end{equation}
where $\mathcal{W}$ is the open set in Remark \ref{rem:tauW}. From \eqref{eq:desigualdadbstar}, it follows that $\hat f(1,0) > 0$ and that $\hat f(b,0) < 0$ for every $b < 2$ close enough to 2. Hence, there exists some $b_0 \in (1,2)$ with $\hat f(b_0,0) = 0$ and such that $b \mapsto \hat f(b,0)$ changes sign at $b = b_0$. Specifically, $\hat f(b,0) > 0$ for $b < b_0$ and $b$ close enough to $b_0$, while $\hat f(b,0) < 0$ for $b > b_0$ in the same conditions. 

\medskip
    
By Lojasiewicz's Structure Theorem \cite[Theorem 5.2.3]{KP} it follows that there exists an analytic regular curve $\mu \subset \mathbb{R}^2$ (not necessarily unique) contained in the region $\{b > 1, \kappa \leq 0\} \subset \mathbb{R}^2$ such that $\hat f(b,\kappa)$ vanishes along $\mu$, and $(b_0,0) \in \mu$. Since $\hat f$ changes sign at $(b_0,0)$, the curve $\mu$ can be locally expressed as a graph over $\kappa$, i.e., there exists some $\kappa_0 <0$ and an analytic function $b_\mu(\kappa)$ defined on $(\kappa_0,0)$ satisfying
$$\mu \equiv \mu(\kappa) = (b_\mu(\kappa),\kappa).$$
Furthermore, the curve $\mu$ can be chosen so that for every $\kappa \in (k_0,0)$, the function $b \mapsto \hat f(b,\kappa)$ changes sign at $b = b_\mu(\kappa)$, meaning that $\hat f(b,\kappa) > 0$ (resp. $\hat f(b,\kappa) <0$) for $b < b_\mu(\kappa)$ (resp. $b > b_\mu(\kappa)$). In other words, the sign of $\hat f(b,\kappa)$ coincides with that of $b_\mu(\kappa) - b$ locally around $\mu$.
\end{remark}

\section{Proof of the main theorem}\label{sec:main}
In Remark \ref{rem:mu}, we found a curve $\mu(\kappa) = (b_\mu(\kappa),\kappa) \in \mathcal{W}\cap \{a = 1\}$, defined for $\kappa \in (\kappa_0,0)$, along which $\tau(1,b_\mu(\kappa),\kappa) = \tilde u(\kappa)$. By Proposition \ref{pro:catenoidesFB}, this means that for every $\kappa \in(\kappa_0,0)$, the surface $\Sigma_0(\tau;1,b_\mu(\kappa),\kappa)$ is a hyperbolic free boundary catenoid, where $\tau= \tau(1,b_\mu(\kappa),\kappa)$. We now aim to construct non-rotational free boundary annuli. To achieve this purpose, we first recall that the surfaces $\Sigma_0(u_0;a,b,\kappa)$ with $a > 1$ are (non-rotational) annuli provided that $\Theta(a,b,\kappa) \in \mathbb{Q}$; see Remark \ref{rem:anillos}. This property will be studied in Subsection \ref{subsec:81}. In Subsection \ref{subsec:h} we will introduce a {\em height map} $\mathfrak{h} := \mathfrak{h}(a,b,\kappa)$ whose zeros are related with the free boundary property. Finally, in Subsection \ref{subsec:83}, we will prove Theorem \ref{thm:main}.

\subsection{The curve \texorpdfstring{$\mu$}{} and the period map}\label{subsec:81}
We will now study some properties of the curve $\mu$ and its relation with the period map $\Theta(a,b,\kappa)$ for $a = 1$.

\begin{lemma}\label{lem:curvaperiodo}
    Let $a = 1$ and $\Theta_0 \in \left(-\frac{1}{\sqrt{2}},-\frac{1}{\sqrt{3}}\right)$. Then, there exists some $\kappa_1 \in \left(0,\frac{1}{4}\right)$ such that the point $\left(1,b_{\Theta_0}(1,\kappa),\kappa \right)$ belongs to $\mathcal{W}_1$ for every $\kappa \in [0,\kappa_1)$ and to $\mathcal{K}$ for $\kappa = \kappa_1$, where $\mathcal{W}_1$ and $\mathcal{K}$ are defined in \eqref{eq:mathcalW1} and \eqref{eq:defK}, and $b_{\Theta_0}(1,\kappa)$ is given by \eqref{eq:btheta}.
\end{lemma}
\begin{figure}
\centering
\includegraphics[width=0.6\textwidth]{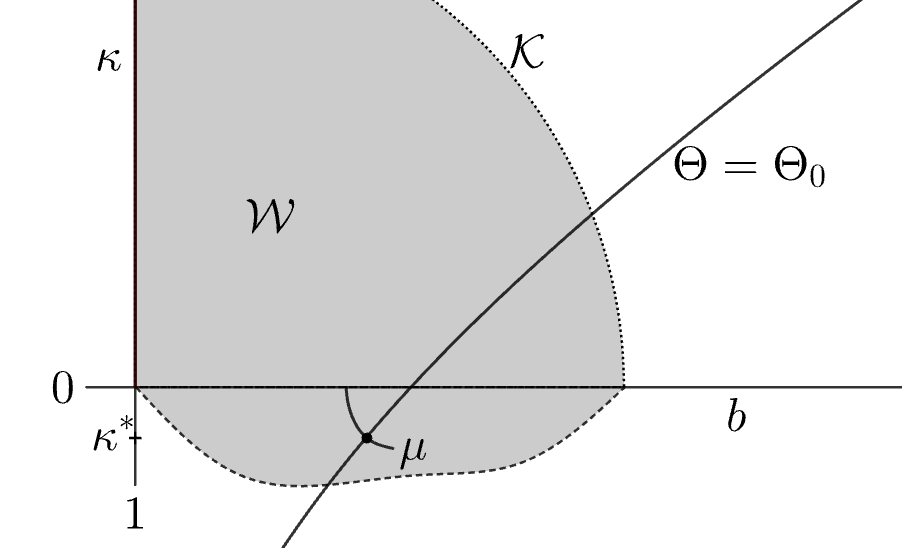}
\caption{Level curve $\Theta = \Theta_0$ of the period map on the plane $\{a = 1\}$ for $\Theta_0 \in \left(- \frac{1}{\sqrt{2}},-\frac{1}{\sqrt{3}}\right)$. This level set and the curve $\mu$ intersect at a point $\mu(\kappa^*) \in \mathcal{W}$ with $\kappa^* <0$.}\label{fig:muperiodo}
\end{figure}
\begin{proof}
    Let $a = 1$ and $\kappa \in [0,1/4)$. From \eqref{eq:mathcalW1} it follows that $(1,b,\kappa) \in \mathcal{O}$ belongs to $\mathcal{W}_1$ if and only if $b < \hat b(\kappa)$, and $(1,b,\kappa) \in \mathcal{K}$ if and only if $b = \hat b(\kappa)$, where $\hat b(\kappa) := 1 + \sqrt{1 - 4\kappa} \geq 1$. By \eqref{eq:btheta},
   \begin{align*}
       b_{\Theta_0}(1,0) &= \frac{1 - \Theta_0^2}{\Theta_0^2} < 2 = \hat b(0), \\
       b_{\Theta_0}\left(1,\frac{1}{4}\right) &= \frac{1 - \Theta_0^2}{\Theta_0^2} + \frac{\sqrt{1 - 2\Theta_0^2}}{\Theta_0^2} >1 = \hat b\left(\frac{1}{4}\right), \\
   \end{align*}
   which implies the existence of some $\kappa_1 \in (0,1/4)$ such that $(1,b_{\Theta_0}\left(1,\kappa\right),\kappa)$ belongs to $\mathcal{W}_1$ for $\kappa \in [0,\kappa_1)$ and to $\mathcal{K}$ for $\kappa = \kappa_1$.
\end{proof}

\begin{lemma}\label{lem:pernoconstante}
    Let $a = 1$ and consider the curve $\mu(\kappa) = (b_\mu(\kappa),\kappa)$, $\kappa \in (\kappa_0,0)$, introduced in Remark \ref{rem:mu}. Then, the period map is not constant along $\mu(\kappa)$. 
\end{lemma}
\begin{proof}
Assume by contradiction that the period map is constant along $\mu$, that is, $\Theta(1,b_\mu(\kappa),\kappa) \equiv \Theta_0$. This means that $\mu(\kappa)$ coincides with the level curve $\kappa \mapsto (b_{\Theta_0}(1,\kappa),\kappa)$ in \eqref{eq:btheta} for $\kappa \in (\kappa_0,0)$. Observe that, since $\mu(0) = (b_0,0)$ with $b_0 \in (1,2)$, \eqref{eq:Theta1bk} shows that $\Theta_0 \in \left(-\frac{1}{\sqrt{2}},-\frac{1}{\sqrt{3}}\right)$. Now, let
\begin{equation}\label{eq:g}
    g(\kappa):= \hat f(b_{\Theta_0}(1,\kappa),\kappa),
\end{equation}
where $\hat f(b,\kappa)$ is defined in \eqref{eq:hatf}. Note that $g(\kappa)$ is analytic for $\kappa \in (\kappa_0,\kappa_1)$, where $\kappa_1 > 0$ is defined in Lemma \ref{lem:curvaperiodo}. It is clear that for $\kappa \in (\kappa_0,0)$, $g(\kappa) =\hat f(\mu(\kappa)) = 0$, so by analyticity, $g(\kappa) \equiv 0$ for all $\kappa \in (\kappa_0,\kappa_1)$. However, this leads to a contradiction:
$$0 = \lim_{\kappa \to \kappa_1} g(\kappa) =\lim_{\kappa \to \kappa_1} \tau(1,b_{\Theta_0}(1,\kappa),\kappa) - \tilde{u}(\kappa) = - \tilde{u}(\kappa_1) < 0,$$
where in the last inequality we use the fact that $(1,b_{\Theta_0}(1,\kappa_1),\kappa_1) \in \mathcal{K}$ and Lemma \ref{lem:limitecontinuotau}. This contradiction proves Lemma \ref{lem:pernoconstante}.
\end{proof}

\begin{proposition}\label{pro:bthetabmu}
There exists an interval $\mathcal{J} \subset \left(-\frac{1}{\sqrt{2}},-\frac{1}{\sqrt{3}}\right)$ such that, for every $\Theta_0 \in \mathcal{J}$, it holds $b_{\Theta_0}(1,\kappa^*) = b_\mu(\kappa^*)$ for some $\kappa^* = \kappa^*(\Theta_0) < 0$. Moreover, the function $\kappa \mapsto  b_\mu(\kappa) - b_{\Theta_0}(1,\kappa)$ changes sign at $\kappa = \kappa^*$; see Figure \ref{fig:muperiodo}.
\end{proposition}
\begin{proof}
    The result is a consequence of Remark \ref{rem:btheta} and Lemma \ref{lem:pernoconstante}. Indeed, consider a closed interval $I:= [\kappa_a,\kappa_b] \subset (\kappa_0,0)$. Let $\Theta_m < \Theta_M$ be the minimum and maximum of the (non-constant) function $\kappa \mapsto \Theta(1,b_{\mu}(1,\kappa),\kappa)$ on $I$, and $\kappa_c, \kappa_d \in I$ be values in which the minimum and maximum are attained, respectively. Then, for every $\Theta_0\in\mathcal{J}:=  \left(\Theta_m,\Theta_M\right)$,
    $$0 < \Theta_0 - \Theta_m = \Theta(1,b_{\Theta_0}(1,\kappa_c),\kappa_c) - \Theta(1,b_\mu(\kappa_c),\kappa_c) = \frac{\partial \Theta}{\partial b}(1,\xi,\kappa_c)(b_{\Theta_0}(1,\kappa_c) - b_\mu(\kappa_c)),$$
    by the mean value theorem, where $\frac{\partial \Theta}{\partial b}$ denotes the partial derivative of $\Theta$ with respect to $b$. This derivative is positive (see Remark \ref{rem:btheta}), so $b_{\Theta_0}(1,\kappa_c) > b_\mu(\kappa_c)$. An analogue computation using $\Theta_0 - \Theta_M$ shows that $b_{\Theta_0}(1,\kappa_d) < b_\mu(\kappa_d)$. As a consequence, there exists some $\kappa^*$ between $\kappa_c$ and $\kappa_d$ satisfying the properties listed in the Proposition.
\end{proof}

\begin{corollary}\label{cor:gcambiasigno}
   With the notations of Proposition \ref{pro:bthetabmu}, for every $\Theta_0 \in \mathcal{J}$, the function $g(\kappa)$ in \eqref{eq:g} changes sign at $\kappa^* < 0$.
\end{corollary}
\begin{proof}
     By definition of $\kappa^*$, it holds $b_{\Theta_0}(1,\kappa^*) = b_\mu(\kappa^*)$, so $g(\kappa^*) = \hat f(b_\mu(\kappa^*),\kappa^*) = 0$. From Remark \ref{rem:mu} we also know that the function $b \mapsto \hat f(b,\kappa)$ changes sign at $b = b_\mu(\kappa)$, i.e., the sign of $\hat f(b,\kappa)$ coincides with that of $b_\mu(\kappa) - b$ locally around $\mu$. Now, since $b_\mu(\kappa)-b_{\Theta_0}(1,\kappa)$ changes sign at $\kappa = \kappa^*$, so does $g(\kappa)$.
\end{proof}

\subsection{The height map}\label{subsec:h}
We now define the following map, whose purpose is to detect free boundary examples.
\begin{definition}\label{def:h}
    Let $(a,b,\kappa) \in \mathcal{W}$ and suppose that for $\tau= \tau(a,b,\kappa)$, it holds \eqref{eq:condesfera} and $\alpha(\tau) \neq 0$, where $\tau$ and $\mathcal{W}$ are given by Proposition \ref{pro:tauW1} and Remark \ref{rem:tauW}. We then define the height map as the analytic function $\mathfrak{h}(a,b,\kappa):= c_3(\tau(a,b,\kappa))$, where $c_3$ is the third coordinate of the center map $c(u)$ in \eqref{eq:definicioncentro}. 
\end{definition}

    Let $(a,b,\kappa) \in \mathcal{W} \cap \mathcal{O}^-$. According to Remarks \ref{rem:anillos}, \ref{rem:FBnegativo}, if $\kappa < 0$, $\mathfrak{h}(a,b,\kappa) = 0$ and $\Theta(a,b,\kappa) \in \mathbb{Q}$, then $\Sigma_0(\tau;a,b,\kappa)$ is a free boundary minimal annulus in some geodesic ball of $\mathbb{M}^3(\kappa)$, where $\tau = \tau(a,b,\kappa)$. This motivates the study of the roots of the function $\mathfrak{h}(a,b,\kappa)$. The following Lemma relates $\mathfrak{h}$ with the map $\hat p$ in Proposition \ref{pro:catenoidesFB}:
\begin{lemma}\label{lem:hI}
    Let $(1,b,\kappa) \in \mathcal{W}$ with $\kappa < 0$ and $\hat p(u): \mathcal{I} \subset \mathbb{R} \to \Upsilon$ be the map in Proposition \ref{pro:catenoidesFB}. If $\tau(1,b,\kappa)$ lies in $\mathcal{I}$, then $\mathfrak{h}(1,b,\kappa)$ is well defined and satisfies
   \begin{equation}\label{eq:hp3}
       \mathfrak{h}(1,b,\kappa) = \hat p_3(\tau(1,b,\kappa)).
   \end{equation}
\end{lemma}
\begin{proof}
    Let $\tau = \tau(1,b,\kappa)$. We will prove that the center map $c(u) = c(u;1,b,\kappa)$ is well defined at $u = \tau$, i.e., the condition \eqref{eq:condesfera} is satisfied; see Remark \ref{rem:kappaneg}. In fact, we will show that $c(\tau) = \hat p(\tau)$, from which \eqref{eq:hp3} follows.

    \medskip

    Let $m(u)$ be the map in \eqref{eq:mnoeuclideo}. For every $u$, $m(u)$ lies in the timelike plane $\mathcal{P} = \{x_1 = x_2 = 0\}$; see Remark \ref{rem:omenoscond} and Proposition \ref{pro:Onovacio}. Recall that this plane intersects $\mathbb{M}^3(\kappa)$ along the geodesic ${\mathcal{L}_\kappa}$, which coincides with the rotation axis of the surface $\Sigma(1,b,\kappa)$; see Proposition \ref{pro:upsilonL}.
    
    \medskip
    
    Now, let $u = \tau$. Using \eqref{eq:mnoeuclideo} and by definition of the map $\tau$, it follows that $m(\tau)$ belongs to the timelike plane $P_0$ spanned by $\psi(\tau)$ and $\psi_u(\tau)$. We will now prove that $P_0 \cap \mathcal{P}$ is a timelike line in $\mathbb{R}^4_\kappa$, which implies by Remark \ref{rem:esferas} that $c(u)$ is well defined, i.e., $m(u)$ admits a rescaling which belongs to $\mathbb{M}^3(\kappa)$.

    \medskip
    
    We note that $P_0 \cap \mathbb{M}^3(\kappa)$ is a geodesic of $\mathbb{M}^3(\kappa)$ which must coincide with $\zeta_\tau$; see Definition \ref{def:zetau}. By Proposition \ref{pro:catenoidesFB}, we know that $\zeta_\tau$ and the rotation axis of the catenoid meet at a unique point $\hat p(\tau)$ in the half space $\{x_4 > 0\}$. This implies that $P_0 \cap \mathcal{P}$ is a timelike line in which $m(\tau)$ lies. In fact, the rescaling $c(\tau)$ of $m(\tau)$ must be $\hat p(\tau)$, as we wanted to prove.
\end{proof}

\begin{proposition}\label{pro:g} With the notations of Proposition \ref{pro:bthetabmu}, for each $\Theta_0 \in \mathcal{J}$, the function
$$\tilde{\mathfrak{g}}(\kappa):= \mathfrak{h}\left(1, b_{\Theta_0}(1,\kappa),\kappa\right)$$
    is well defined on a neighbourhood of $\kappa^*$ and changes sign at that point.
\end{proposition}
\begin{proof}
    We first show that the function $\tilde{\mathfrak{g}}(\kappa)$ is well defined. By Corollary \ref{cor:gcambiasigno}, it holds $\tilde u(\kappa^*) = \tau\left(1, b_{\Theta_0}(1,\kappa^*),\kappa^*\right)$. In particular, for values of $\kappa$ near $\kappa^*$, $\tau = \tau(1,b_{\Theta_0}(1,\kappa),\kappa)$ will be close to $\tilde u(\kappa)$, and so $\tau$ will belong to the interval $\mathcal{I}$ in Proposition \ref{pro:catenoidesFB}. By Lemma \ref{lem:hI}, $\tilde{\mathfrak{g}}(\kappa)$ is well defined on a neighbourhood of $\kappa^*$.

\medskip

    In order to prove the Proposition, we will show that the signs of $\tilde{\mathfrak{g}}(\kappa)$ and $g(\kappa)$ in \eqref{eq:g} coincide on a neighbourhood of $\kappa^*$ and then apply Corollary \ref{cor:gcambiasigno}. To simplify notation, let $\tau = \tau\left(1,b_{\Theta_0}(1,\kappa),\kappa\right)$ and $\tilde u = \tilde u(\kappa)$. By \eqref{eq:hp3} and the definition of $\tilde{\mathfrak{g}}(\kappa)$, it holds $\tilde{\mathfrak{g}}(\kappa) = \hat p_3(\tau)$. Since the function $\hat p_3(u)$ is strictly increasing and $\hat p_3(\tilde u) = 0$, the sign of $\hat p_3(\tau)$ coincides with that of $\tau - \tilde u = g(\kappa)$, as we wanted to show. In particular, note that $\tilde{\mathfrak{g}}(\kappa^*) = g(\kappa^*)= 0$.
\end{proof}

\subsection{Proof of Theorem \ref{thm:main}}\label{subsec:83}

Fix $q \in \mathcal{J} \cap \mathbb{Q}$ and define
\begin{equation}
    \mathfrak{G}(a,\kappa):=\mathfrak{h}(a,b_q(a,\kappa),\kappa),
\end{equation}
where $b_q(a,\kappa)$ is the function in Remark \ref{rem:btheta} for $\Theta_0 = q$. Observe that if $a = 1$, $\mathfrak{G}(1,\kappa) = \tilde{\mathfrak{g}}(\kappa)$. By Proposition \ref{pro:g}, $\mathfrak{G}(a,\kappa)$ is well defined and analytic on a neighbourhood of $(a,\kappa) =(1,\kappa^*)$.
Arguing as in Remark \ref{rem:mu}, observe that $\kappa \mapsto \mathfrak{G}(1,\kappa)$ changes sign at $\kappa^*$. Hence, applying Lojasiewicz's Structure Theorem \cite[Theorem 5.2.3]{KP} to the map $\mathfrak{G}(a,\kappa)$, it follows that there exists $\varepsilon = \varepsilon(q)$ and a curve $\hat \mu_q:[0,\varepsilon) \to \mathbb{R}^2$, $\hat \mu_q(\eta):= \left(a(\eta), \kappa(\eta)\right)$, satisfying $\hat \mu_q(0) = (1,\kappa^*)$, $a(\eta) > 1$ for every $\eta > 0$ and $\kappa(\eta) < 0$, $\mathfrak{G}(\hat\mu_q(\eta)) \equiv 0$ for all $\eta \geq 0$. Consider the 1-parameter family of surfaces
$$\Sigma_q(\eta) := \Sigma_0\left(\tau;a(\eta),b_q(a(\eta),\kappa(\eta)),\kappa(\eta)\right),$$
where $\tau = \tau\left(a(\eta),b_q(a(\eta),\kappa(\eta)),\kappa(\eta)\right)$; see Definition \ref{def:Sigma_0}. Let us express the rational number $q$ as an irreducible fraction $q = -m/n$, $m,n \in \mathbb{N}$. From Remark \ref{rem:anillos} it follows that the surfaces $\Sigma_q(\eta)$ are compact minimal annuli once we identify $(u,v) \sim (u, v + 2n\sigma)$, where $\sigma = \sigma\left(a(\eta),b_q(a(\eta),\kappa(\eta)),\kappa(\eta)\right)$ depends analytically on $\eta$. Moreover, by definition of $\tau$ and since the height map vanishes along the parameters associated to $\Sigma_q(\eta)$, it follows from Remark \ref{rem:FBnegativo} that each of these annuli are free boundary in some geodesic ball $B = B(q,\eta)$ of $\mathbb{M}^3(\kappa(\eta))$. Furthermore, the following properties are satisfied:
\begin{enumerate}[(i)]
    \item The balls $B(q,\eta)$ are centered at $c(\tau) = 
 {\bf e}_4 \in \mathbb{M}^3(\kappa(\eta))$ by item (4) of Theorem \ref{thm:simetrias}.
    \item Each annulus $\Sigma_q(\eta)$ is invariant under the symmetry with respect to the totally geodesic surface $\mathbf{S}_{\kappa(\eta)} \subset \mathbb{M}^3({\kappa(\eta)})$ and also with respect to the $n$ totally geodesic surfaces $\Omega_j\subset \mathbb{M}^3({\kappa(\eta)}) $ which meet equiangularly along $\mathcal{L}_{\kappa(\eta)}$. This follows from item (1) of Theorem \ref{thm:simetrias}.
    \item The rotation index of the planar geodesic $\Sigma_q(\eta) \cap \mathbf{S}_{\kappa(\eta)}$ is $-m$; see Lemma \ref{lem:simetriaS} and Proposition \ref{pro:simetriasGamma}.
    \item The annuli $\Sigma_q(\eta)$ are foliated by spherical curvature lines, by Lemma \ref{lem:lineasesfericas}.
    \item If $\eta > 0$, then $a(\eta) > 1$, and so the annulus $\Sigma_q(\eta)$ has a prismatic symmetry group of order $4n$; see item (5) of Theorem \ref{thm:simetrias}.
    \item For $\eta = 0$ it holds $a(\eta) = 1$, so the annulus $\Sigma_q(0)$ is an $m$-cover of a free boundary hyperbolic catenoid. This follows from the discussion in Subsection \ref{sec:1bk} and Proposition \ref{pro:simetriasGamma}.
\end{enumerate}

We will now define the set of annuli $\mathbb{A}_q(\eta) \subset \mathbb{H}^3$ in Theorem \ref{thm:main}. First, recall that for every $\kappa < 0$, $\mathbb{M}^3(\kappa) \subset \mathbb{R}_\kappa^4$ is a hyperbolic space of sectional curvature $\kappa$; see Remark \ref{rem:M3k}. In fact, the linear change $x_4 \mapsto x_4/\sqrt{-\kappa}$ yields an isometry between $\mathbb{M}^3(\kappa)$ and the three-dimensional hyperboloid model $\mathbb{H}^3(\kappa) \subset \mathbb{L}^4$ of curvature $\kappa$, that is,
$$\mathbb{H}^3(\kappa) = \{(x_1,x_2,x_3,x_4) \in \mathbb{L}^4 \; : \; x_1^2 + x_2^2 + x_3^2 - x_4^2 = 1/\kappa, x_4> 0\}.$$
Now, a homothety of ratio $\sqrt{-\kappa}$ in Lorentz space $\mathbb{L}^4$ sends the hyperboloid $\mathbb{H}^3(\kappa)$ to the usual hyperbolic space $\mathbb{H}^3 \subset \mathbb{L}^4$ of curvature $-1$; see \eqref{eq:espaciohiperbolico1}. We then define $\mathbb{A}_q(\eta) \subset \mathbb{H}^3$ as the image of the annulus $\Sigma_q(\eta)$ under the isometry and homothety that sends the space $\mathbb{M}^3(\kappa(\eta))$ to $\mathbb{H}^3$. By the properties (i)-(vi) of the annuli $\Sigma_q(\eta)$ listed above, it is immediate to check that the the surfaces $\mathbb{A}_q(\eta)$ are free boundary minimal annuli in geodesic balls of $\mathbb{H}^3$ satisfying items (1)-(6) of Theorem \ref{thm:main}. This completes the proof.

\appendix 

\section{Appendix}\label{sec:appendix}
This Appendix is devoted to proving Proposition \ref{pro:catenoidesFB}. We will study the surfaces $\Sigma(1,b,\kappa)$ in terms of the parametrizations $\varphi = \varphi(\mathfrak{s},\theta)$ in \eqref{eq:paramesfera}, \eqref{eq:paramhiperb}, \eqref{eq:parameuclideo}. For this purpose, we first introduce the following definitions:

\begin{definition}\label{def:C0}
    For any $(1,b,\kappa) \in \mathcal{O}$ and $\mathfrak{s}_0 > 0$, we define $\mathfrak{C}_0(\mathfrak{s}_0;\kappa)$ as the restriction of $\varphi(\mathfrak{s},\theta)$ to $[-\mathfrak{s}_0,\mathfrak{s}_0]\times \mathbb{R}$ under the identification $(\mathfrak{s},\theta) \sim (\mathfrak{s},\theta + 2\pi)$.
\end{definition}
\begin{remark}\label{rem:C0S0}
    For every $\mathfrak{s}>0$ and $\kappa \in \left(-\frac{1}{4},\frac{1}{4}\right)$, the surface $\mathfrak{C}_0(\mathfrak{s};\kappa) $ is a rotational minimal annulus. The surface $\Sigma_0 = \Sigma_0(u_0;1,b,\kappa)$ in Definition \ref{def:Sigma_0} corresponds with the restriction of $\varphi(\mathfrak{s},\theta)$ to $[-\mathfrak{s}_0,\mathfrak{s}_0]\times \mathbb{R}$, where $\mathfrak{s}_0 := \mathfrak{s}(u_0)$; see Lemma \ref{lem:stuv}. Hence, $\Sigma_0$ is an infinite cover of the annulus $\mathfrak{C}_0$.    
\end{remark}

\begin{remark}\label{rem:zetas}
From Lemma \ref{lem:stuv} it follows that the geodesics $\zeta_{u(\mathfrak{s})}$ in Definition \ref{def:zetau} pass through $\varphi(\mathfrak{s},0)$ with tangent vector $\varphi_{\mathfrak{s}}(\mathfrak{s},0)$. For simplicity, we set $\zeta_{\mathfrak{s}}:= \zeta_{u(\mathfrak{s})}$.
\end{remark}

We will show next that there exists an analytic function $\tilde {\mathfrak{s}} = \tilde {\mathfrak{s}}(\kappa) > 0$ such that the annulus $\mathfrak{C}_0(\tilde {\mathfrak{s}};\kappa)$ is free boundary in a ball contained in the half-space $\{x_4 > 0\}$. In particular, setting $\tilde u(\kappa) := u(\tilde{\mathfrak{s}}(\kappa))$ we deduce items (1) and (3) of Proposition \ref{pro:catenoidesFB}. Item (2) will appear naturally along the proof.

\subsection{Proof of item (1) of Proposition \ref{pro:catenoidesFB}}\label{sec:appendixA1}

Let $\kappa \in \left(-\frac{1}{4},\frac{1}{4}\right)$ and $\mathfrak{s}_0 > 0$. By Remark \ref{rem:parametrosstheta}, we know that the annulus $\mathfrak{C}_0(\mathfrak{s}_0;\kappa)$ in Definition \ref{def:C0} is invariant under the reflection with respect to the totally geodesic surface ${{\bf S}_\kappa}$ and also under rotations with axis $\Upsilon$. So, if $\mathfrak{C}_0(\mathfrak{s}_0;\kappa)$ were free boundary in a geodesic ball $B$, the following properties should hold:
\begin{enumerate}
    \item $B$ must be invariant under these two isometries. In particular, the geodesic center of $B$ has to be the point ${\bf e}_4$, where $\Upsilon$ and ${{\bf S}_\kappa}$ meet.
    \item  The profile curve $\mathfrak{s} \mapsto \varphi(\mathfrak{s},0)$ must meet the boundary of $B$ orthogonally at $\mathfrak{s} = \mathfrak{s}_0$. Since $B$ is a geodesic ball, this is equivalent to the fact that the geodesic $\zeta_{\mathfrak{s}_0}$ in Remark \ref{rem:zetas} meets ${\bf e}_4$, i.e. the center of $B$.
\end{enumerate}
As we will see in Lemma \ref{lem:Fbordelibre}, these two geometric properties can be reduced to studying whether the function
\begin{equation}\label{eq:F}
        F(\mathfrak{s}) := x_3(\mathfrak{s}) x'(\mathfrak{s}) - x_3'(\mathfrak{s}) x(\mathfrak{s})
    \end{equation}
vanishes at $\mathfrak{s} = \mathfrak{s}_0$. Here, $x(\mathfrak{s})$ and $x_3(\mathfrak{s})$ denote the first and third coordinates of the profile curve $\varphi(\mathfrak{s},0)$ in \eqref{eq:paramesfera}, \eqref{eq:paramhiperb}, \eqref{eq:parameuclideo}.

\begin{lemma}\label{lem:F}
   For any $\kappa \in \left(-\frac{1}{4},\frac{1}{4}\right)$, the function $F(\mathfrak{s})$ in \eqref{eq:F} has a first positive root $\tilde{\mathfrak{s}} = \tilde{\mathfrak{s}}(\kappa) > 0$. In addition,
    \begin{enumerate}
        \item $x(\mathfrak{s})$ is strictly increasing in $(0,\tilde{\mathfrak{s}}]$ and $x'(\tilde{\mathfrak{s}}) > 0$.
        \item $x_3'(\mathfrak{s}) > 0$ for every $\mathfrak{s} \in [-\tilde{\mathfrak{s}},\tilde{\mathfrak{s}}]$. In particular, $x_3(\tilde{\mathfrak{s}})$ is strictly positive in $(0,\tilde{\mathfrak{s}}]$.
        \item If $\kappa \neq 0,$ $\kappa x_4'(\mathfrak{s})< 0$ for all $\mathfrak{s} \in (0,\tilde{\mathfrak{s}}]$ and $x_4(\tilde{\mathfrak{s}}) > 0$, where $x_4(\mathfrak{s})$ is the fourth coordinate of the profile curve $\varphi(\mathfrak{s},0)$.
    \end{enumerate}

\end{lemma}
\begin{proof}
    We will prove that $F(0) < 0$ and that $F(\mathfrak{s}_0) >0 $ for some $\mathfrak{s}_0 > 0$, deducing the existence of a first positive root $\tilde{\mathfrak{s}} \in (0, \mathfrak{s}_0)$. The properties on the functions $x(\mathfrak{s}), x_3(\mathfrak{s}), x_4(\mathfrak{s})$ will appear naturally in the process.
    
    \medskip

    It is straightforward that $F(0) < 0$ using \eqref{eq:paramesfera}, \eqref{eq:x}, \eqref{eq:phiesferico}, \eqref{eq:paramhiperb}, \eqref{eq:phihip} and \eqref{eq:parameuclideo}: indeed, for all $\kappa \in \left(-\frac{1}{4},\frac{1}{4}\right)$, it holds $x_3(0) = 0$, $x(0) = \frac{2}{\sqrt{4\kappa + 1}} > 0$ and $x_3'(0) > 0$. It just remains to prove that $F(\mathfrak{s}_0) > 0$ for some $\mathfrak{s}_0 > 0$. We will deal with three cases depending on the sign of $\kappa$. 

    \medskip

    \textbf{First case: $\kappa = 0$.} We have an explicit parametrization of the catenoid given by \eqref{eq:parameuclideo}, and so
    $$F(\mathfrak{s}) = 2 \arcsinh\left(\frac{\mathfrak{s}}{2}\right) \frac{\mathfrak{s}}{\sqrt{\mathfrak{s}^2 + 4}} - 1.$$
    The first (and unique) positive root of $F(\mathfrak{s})$ is given by the equation
    $$\arcsinh\left(\frac{\mathfrak{s}}{2}\right) = \frac{\sqrt{\mathfrak{s}^2 + 4}}{\mathfrak{s}}.$$
Items (1) and (2) are trivial in this situation.
    \medskip

    \textbf{Second case: $\kappa < 0$.} Let us study the functions $x(\mathfrak{s})$ and $\phi(\mathfrak{s})$ in \eqref{eq:phihip}. Observe that the polynomial $h(x)$ that appears in \eqref{eq:x} vanishes at $x(0)$ and is positive for all $x \in (x(0),\infty)$. This implies that $x(\mathfrak{s})$ is strictly increasing for all $\mathfrak{s} > 0$ with $\lim_{\mathfrak{s} \to \infty}x(\mathfrak{s}) = \infty$. In particular, the function $\phi(\mathfrak{s})$ in \eqref{eq:phihip} must be increasing as well. Nevertheless, $\phi(\mathfrak{s})$ is bounded: indeed, applying the change of variables $y = x(\mathfrak{s})$, it holds
    \begin{align*}
        \phi(\mathfrak{s}) &< \lim_{\mathfrak{s} \to \infty} \phi(\mathfrak{s}) = \int_{x(0)}^{\infty} \frac{\delta}{\sqrt{-\kappa} (y^2 - 1/\kappa)\sqrt{-\kappa y^4 + y^2 - \delta^2}}dy < \infty.
    \end{align*}
    We can now show that $F(\mathfrak{s}) > 0$ for $\mathfrak{s}$ large enough. Consider the function $G(\mathfrak{s}) := \frac{F(\mathfrak{s})\sqrt{x(\mathfrak{s})^2 - 1/\kappa}}{x'(\mathfrak{s})}$. If $\mathfrak{s} > 0$, the sign of $F(\mathfrak{s})$ coincides with that of $G(\mathfrak{s})$. A straightforward computation using \eqref{eq:paramhiperb}, \eqref{eq:phihip}, \eqref{eq:F} shows that
    $$\lim_{\mathfrak{s} \to \infty} G(\mathfrak{s}) = -\frac{1}{\kappa}\sinh(\phi_M) > 0,$$
    where $\phi_M:= \lim_{\mathfrak{s} \to \infty}\phi(\mathfrak{s})$. In particular, $G(\mathfrak{s})$ is positive for large $\mathfrak{s} > 0$, and so is $F(\mathfrak{s})$. Items (1) and (3) of the Lemma are immediate, while (2) is a consequence of the fact that $\phi(\mathfrak{s})$ is anti-symmetric and strictly increasing.

    \medskip

    \textbf{Third case: $\kappa > 0$.} The polynomial $h(x)$ in \eqref{eq:x} now has two positive roots $y_1 <  y_2$ with $y_1 = x(0)$. In particular, the function $x(\mathfrak{s})$ will be periodic and taking values on the interval $[y_1,y_2]$. Let us denote by $\mathfrak{s}_1$ the first positive value for which $x(\mathfrak{s}_1) = y_2$. The function $\phi(\mathfrak{s})$ is increasing in $[0,\mathfrak{s}_1]$, according to \eqref{eq:phiesferico}. We now claim that
    \begin{equation}\label{eq:phinu}
        \phi(\mathfrak{s}_1) > \frac{\pi}{2}
    \end{equation}
holds; we will prove this later. From \eqref{eq:phinu} we deduce that there exists $\mathfrak{s}_0 \in (0,\mathfrak{s}_1)$ for which $\phi(\mathfrak{s}_0) = \frac{\pi}{2}$. Since $x'(\mathfrak{s}_0) > 0$,
    $$F(\mathfrak{s}_0) = \sqrt{1/\kappa - x(\mathfrak{s}_0)^2}x'(\mathfrak{s}_0) + \frac{x'(\mathfrak{s}_0)}{2\sqrt{1/\kappa - x(\mathfrak{s}_0)^2}}x(\mathfrak{s}_0) > 0,$$
    as we wanted to show. Items (1) and (3) of the Lemma are immediate. Regarding item (2), suppose by contradiction that  $x_3'(\mathfrak{s})$ is not positive for all $\mathfrak{s} \in [0,\tilde{\mathfrak{s}}]$: in such a case, since $x_3'(0) > 0$, there exists a first root $s^* \in (0,\tilde{\mathfrak{s}}]$ of $x_3'(\mathfrak{s})$. As a consequence, $x_3(s^*) > 0$, and so $F(s^*) = x_3(s^*)x'(s^*) > 0$. This is impossible because $F(0) < 0$ and $\tilde{\mathfrak{s}}$ is by definition the first root of $F(\mathfrak{s})$. Item (2) is then a consequence of the anti-symmetry of $x_3(\mathfrak{s})$, proving Lemma \ref{lem:F}.
    
    \medskip
    
    We now prove \eqref{eq:phinu}. Consider the change of variables $w = x(\mathfrak{s})^2$ on the integral that defines $\phi(\mathfrak{s}_1)$. Letting $w_m = y_1^2$, $w_M = y_2^2$ and by using \eqref{eq:x}, we see that
$$\phi(\mathfrak{s}_1) = \int_{w_m}^{w_M} \frac{\delta}{2\sqrt{\kappa}\left(1/\kappa -w\right)\sqrt{w(w_M - w)(w-w_m)}}dw.$$

We can study this integral by using the residue theorem: indeed, let $f(z)$ be the function
$$f(z) := \frac{i\delta}{2\sqrt{\kappa}\left(1/\kappa -z\right)\sqrt{z(z - w_M)(z-w_m)}},$$
\begin{figure}
\centering
\includegraphics[width=0.45\textwidth]{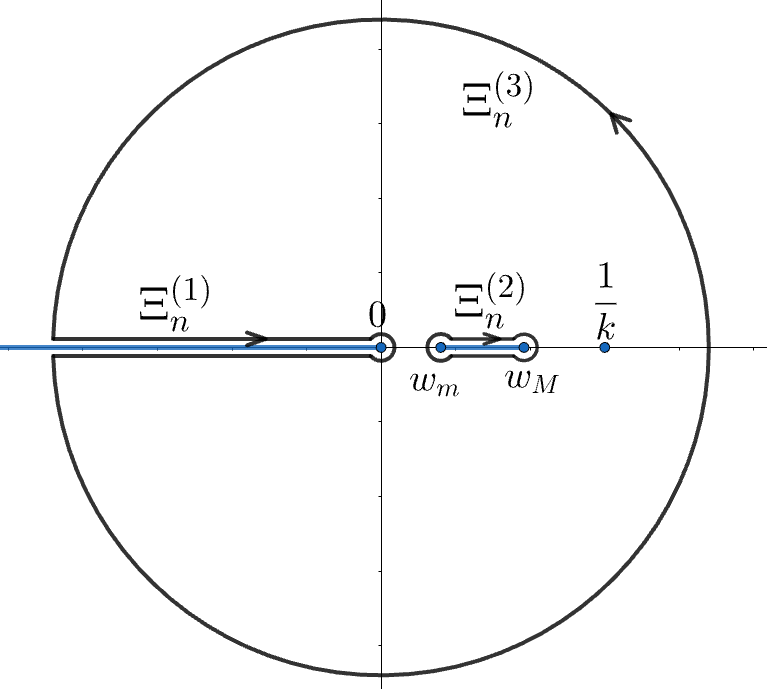}
\caption{Integration path $\Xi_n$.}\label{fig:integracion2}
\end{figure}
which is meromorphic on $\mathbb{C} \setminus ((-\infty,0]\cap [w_m,w_M])$ and has a pole at $z = 1/\kappa$. For $n$ large enough, let $\Xi_n$ be the integration path of Figure \ref{fig:integracion2}. By the residue theorem,
$$\int_{\Xi_n}f(z)dz = 2\pi i \text{Res}\left(f,\frac{1}{\kappa}\right) = \pi.$$
We can alternatively decompose $\Xi_n$ into three pieces, namely: the path $\Xi^{(1)}_n$ around the interval $(-\infty,0]$, the cycle $\Xi^{(2)}_n$ enclosing the interval $[-w_m,w_M]$ and the circular arc $\Xi^{(3)}_n$ of radius $n$. We take $\Xi^{(1)}_n$ and $\Xi^{(2)}_n$ so that they converge to their respective intervals. It can be checked that
\begin{align*}
\lim_{n \to \infty}& \int_{\Xi^{(1)}_n}f(z)dz = -\int_{-\infty}^{0} \frac{\delta}{2\sqrt{\kappa}\left(1/\kappa -w\right)\sqrt{-w(w - w_M)(w-w_m)}} = \mathcal{C} < 0, \\
 \lim_{n \to \infty} &\int_{\Xi^{(2)}_n}f(z)dz  = 2\phi(\mathfrak{s}_1),\\
   \lim_{n \to \infty}& \int_{\Xi^{(3)}_n}f(z)dz  = 0,
\end{align*}
for some $\mathcal{C} < 0$. As a consequence, $2\phi(\mathfrak{s}_1)  + \mathcal{C} = \pi$, and so $$\phi(\mathfrak{s}_1) = \frac{\pi - \mathcal{C}}{2} > \frac{\pi}{2},$$
proving \eqref{eq:phinu}. This completes the proof.
\end{proof}
We next prove item (1) in Proposition \ref{pro:catenoidesFB}:
\begin{lemma}\label{lem:Fbordelibre}
    For any $\kappa \in (-1,4/1,4)$, the compact minimal annulus $\mathfrak{C}_0(\tilde{\mathfrak{s}};\kappa)$ in Definition \ref{def:C0} is embedded and free boundary in a geodesic ball $B$ centered at ${\bf e}_4$. More specifically, for $\kappa \neq 0$,
\begin{equation}\label{eq:bolanoeuclideaFB}
        B = B[{\bf e}_4, x_4(\tilde{\mathfrak{s}})/\kappa],
    \end{equation}
    while for $\kappa = 0$, $B$ is the ball centered at ${\bf e}_4$ with radius
   \begin{equation}\label{eq:radiobola}
       R = \sqrt{x(\tilde{\mathfrak{s}})^2 + x_3(\tilde{\mathfrak{s}})^2}.
   \end{equation}
    In any of the cases, $B$ is contained in $\{x_4 > 0\} \cap \mathbb{M}^3(\kappa)$.
\end{lemma}
\begin{proof}
    The embeddedness of the annulus $\mathfrak{C}_0(\tilde{\mathfrak{s}};\kappa)$ follows from the fact that the coordinate function $x_3(\mathfrak{s})$ is strictly increasing on the interval $[-\tilde{\mathfrak{s}},\tilde{\mathfrak{s}}]$; see Lemma \ref{lem:F}.
    
\medskip

    By our discussion at the beginning of Section \ref{sec:appendixA1}, if $\mathfrak{C}_0(\tilde{\mathfrak{s}};\kappa)$ were free boundary in a geodesic ball $B$, its center should necessarily be ${\bf e}_4$. This will happen if the following conditions hold:
    \begin{enumerate}
        \item $\mathfrak{C}_0(\tilde{\mathfrak{s}};\kappa)$ must be contained in $B$ and its two boundary components, given by $\varphi\left(\{\tilde{\mathfrak{s}}\} \times \mathbb{R}\right)$, $\varphi\left(\{-\tilde{\mathfrak{s}}\} \times \mathbb{R}\right)$, should lie in the geodesic sphere $\partial B$.
        \item The geodesic $\zeta_{\tilde{\mathfrak{s}}}$ meets the point ${\bf e}_4$; see Remark \ref{rem:zetas}. This ensures that the annulus $\mathfrak{C}_0(\tilde{\mathfrak{s}};\kappa)$ meets the geodesic sphere $\partial B$ orthogonally.
    \end{enumerate}
The first condition holds if we choose the balls given by \eqref{eq:bolanoeuclideaFB}, \eqref{eq:radiobola}: indeed, if $\kappa \neq 0$ and $\mathfrak{s} \in [0,\tilde{\mathfrak{s}}],$
$$\langle {\bf e}_4, \varphi(\mathfrak{s},\theta)\rangle_\kappa = \frac{x_4(\mathfrak{s})}{\kappa} \geq \frac{x_4(\tilde{\mathfrak{s}})}{\kappa},$$
since $\kappa x_4(\mathfrak{s})$ is decreasing by Lemma \ref{lem:F}. The inequality also holds for $\mathfrak{s} \in [-\tilde{\mathfrak{s}},0]$, as $x_4(\mathfrak{s})$ is symmetric. This implies by \eqref{eq:bolanoeuclideaFB} that $\mathfrak{C}_0(\tilde{\mathfrak{s}}) \subset  B = B[{\bf e}_4, x_4(\tilde{\mathfrak{s}})/\kappa]$. Similarly, if $\kappa = 0$, we take the ball $B$ centered at ${\bf e}_4$ of radius $R$ as in \eqref{eq:radiobola}. Since the function $\mathfrak{s} \mapsto \sqrt{x(\mathfrak{s})^2 + x_3(\mathfrak{s})^2}$ is strictly increasing in $\mathfrak{s} \in [0,\tilde{\mathfrak{s}}]$ (see \eqref{eq:parameuclideo}), we deduce that the catenoid $\mathfrak{C}_0(\tilde{\mathfrak{s}};0)$ is contained in $B$, as expected.
 
\medskip

We will now show the second condition, i.e., that $\varphi(\mathfrak{s},0)$ meets the geodesic sphere $\partial B$ orthogonally. For any $\kappa \in \left(-\frac{1}{4},\frac{1}{4}\right)$, we define $\mathbb{M}^2(\kappa):= \mathbb{M}^3(\kappa) \cap \{x_2 = 0\}$. This surface has constant sectional curvature $\kappa$, so it is isometric to either a 2-sphere, a hyperbolic plane or a Euclidean plane. Now, let $P: \mathbb{M}^2(\kappa)\cap \{x_4 > 0\} \to \mathbb{R}^2$ be the totally geodesic projection
\begin{equation}\label{eq:proyP}
    (\overline x, \overline y) = P(x_1,x_3,x_4) := \left(\frac{x_1}{x_4},\frac{x_3}{x_4}\right).
\end{equation}
Observe that if $\kappa \leq 0$, $P$ is defined on the whole surface $\mathbb{M}^2(\kappa)$, while for $\kappa > 0$ it is just defined on a hemisphere. Moreover, $P$ sends the rotation axis $\Upsilon$ to the line $\{\overline x = 0\}$, and $P({\bf e}_4) = (0,0)$. Now, let $\eta(\mathfrak{s})$ be the projection of the profile curve $\varphi(\mathfrak{s},0)$, that is, $\eta(\mathfrak{s}) = P(\varphi(\mathfrak{s},0))$. Since $P$ is totally geodesic, the projection of any geodesic $\zeta_{\mathfrak{s}_0}$ coincides with the tangent line $L_{\mathfrak{s}_0}$ to the curve $\eta(\mathfrak{s})$ at $\eta(\mathfrak{s}_0)$. This line will either be parallel to the axis $\{\overline x = 0\}$ or meet it at a single point $(0, \overline y)$. The condition for $L_{\mathfrak{s}_0}$ not to be parallel to the axis is that $\left(\frac{x}{x_4}\right)'\big|_{\mathfrak{s}_0} \neq 0$. In that case, it follows that  
\begin{equation}\label{eq:yF}
    \overline y(\mathfrak{s}_0) = \frac{x_3(\mathfrak{s}_0)}{x_4(\mathfrak{s}_0)} - \frac{x(\mathfrak{s}_0)\left(\frac{x_3}{x_4}\right)'\big|_{\mathfrak{s}_0}}{x_4(\mathfrak{s}_0)\left(\frac{x}{x_4}\right)'\big|_{\mathfrak{s}_0}} = \frac{F(\mathfrak{s}_0)}{x'(\mathfrak{s}_0)x_4(\mathfrak{s}_0) - x(\mathfrak{s}_0) x_4'(\mathfrak{s}_0)},
\end{equation}
see \eqref{eq:F}. Assume for a moment that
\begin{equation}\label{eq:condnecesaria}
    \left(\frac{x}{x_4}\right)' = \frac{x'x_4 - x x_4'}{x_4^2} > 0
\end{equation}
holds on a neighbourhood of $\mathfrak{s} = \tilde{\mathfrak{s}}$.
Then, \eqref{eq:yF} implies that $\overline y(\tilde{\mathfrak{s}}) = 0$, that is, the line $L_{\tilde{\mathfrak{s}}}$ passes through the origin of $\mathbb{R}^2$, which means that $\zeta_{\tilde{\mathfrak{s}}} \subset \mathbb{M}^2(\kappa)$ meets ${\bf e}_4$. This is exactly the condition that we were looking for. Hence, it just remains to check \eqref{eq:condnecesaria} to prove Lemma \ref{lem:Fbordelibre}. Suppose by contradiction that
\begin{equation}\label{eq:contradict}
    x'(\tilde{\mathfrak{s}}) x_4(\tilde{\mathfrak{s}}) \leq  x(\tilde{\mathfrak{s}}) x_4'(\tilde{\mathfrak{s}}).
\end{equation}
This is impossible if $\kappa \geq 0$ as $x(\tilde{\mathfrak{s}}), x'(\tilde{\mathfrak{s}}), x_4(\tilde{\mathfrak{s}}) > 0$ but $x_4'(\tilde{\mathfrak{s}}) \leq 0$; see \eqref{eq:parameuclideo} and Lemma \ref{lem:F}. Assume then that $\kappa < 0$. The fact that $\varphi(\mathfrak{s},0) \in \mathbb{M}^2(\kappa)$ implies that
\begin{equation}\label{eq:estarenM3}
    x^2 + x_3^2 +\frac{x_4^2}{\kappa} = \frac{1}{\kappa}
\end{equation}
for all $\mathfrak{s}$, so in particular
\begin{equation}\label{eq:desigualdadderivada}
x(\tilde{\mathfrak{s}})x'(\tilde{\mathfrak{s}}) + x_3(\tilde{\mathfrak{s}})x_3'(\tilde{\mathfrak{s}}) = -\frac{x_4(\tilde{\mathfrak{s}}) x_4'(\tilde{\mathfrak{s}})}{\kappa}.
\end{equation}
By definition of $\tilde{\mathfrak{s}}$, it holds $x_3(\tilde{\mathfrak{s}})x'(\tilde{\mathfrak{s}}) = x(\tilde{\mathfrak{s}})x_3'(\tilde{\mathfrak{s}})$. Moreover, by \eqref{eq:contradict}, $x_4'(\tilde{\mathfrak{s}})  \geq \frac{x'(\tilde{\mathfrak{s}})x_4(\tilde{\mathfrak{s}})}{x(\tilde{\mathfrak{s}})}.$ Substituing in \eqref{eq:desigualdadderivada} for $\mathfrak{s} = \tilde{\mathfrak{s}}$,
$$xx' + \frac{x_3^2x'}{x} = -\frac{x_4x_4'}{\kappa} \geq -\frac{x'x_4^2}{\kappa x} = -\frac{x'}{\kappa x}\left(1 
 - \kappa (x^2 + x_3^2)\right).$$
 Multiplying both sides by $\frac{x}{x'} > 0$,
 $$x^2 + x_3^2 \geq  x^2 + x_3^2 - \frac{1}{\kappa}.$$
 This implies that $\frac{1}{\kappa} > 0$, which is impossible by hypothesis. Hence, \eqref{eq:condnecesaria} holds, proving Lemma \ref{lem:Fbordelibre} and item (1) of Proposition \ref{pro:catenoidesFB}.

\end{proof}


\subsection{Proof of item (2) in Proposition \ref{pro:catenoidesFB}}
    In the proof of Lemma \ref{lem:Fbordelibre}, we showed that \eqref{eq:condnecesaria} holds on a neighbourhood $I \subset \mathbb{R}$ of $\tilde{\mathfrak{s}}$. In particular, by \eqref{eq:yF} we see that $\overline y = \overline{y}(\mathfrak{s}):I\to \mathbb{R}$ is analytic. On the other hand, by definition of $\overline y$, the point $P^{-1}(0,\overline y(\mathfrak{s})) \in \mathbb{M}^2(\kappa)$ is the (unique) point in in the half-space $\{x_4 > 0\}$ in which the geodesics $\Upsilon$ and $\zeta_{\mathfrak{s}}$ meet. This shows that the map $\hat p(u)$ introduced in item (3) of Proposition \ref{pro:catenoidesFB} is well defined and analytic on a neighbourhood $\mathcal{I} \subset \mathbb{R}$ of $\tilde u$. The fact that $\hat p_3(\tilde u) = 0$ is immediate since $\overline{y}(\tilde{\mathfrak{s}}) = 0$. It is also clear by \eqref{eq:yF} that the sign of the derivative $\hat p_3'(\tilde u)$ coincides with that of $F'(\tilde{\mathfrak{s}})$. We will show in the following Section that $F'(\tilde{\mathfrak{s}}) > 0$, proving item (2) of Proposition \ref{pro:catenoidesFB}.

\subsection{Proof of item (3) of Proposition \ref{pro:catenoidesFB}}\label{sec:appendixA2}
We will now prove that the map $\tilde{\mathfrak{s}} = \tilde{\mathfrak{s}}(\kappa):\left(-\frac{1}{4},\frac{1}{4}\right) \to \mathbb{R}^+$ is analytic. By definition, $\tilde{\mathfrak{s}}$ is the first positive root of the function $F$ in \eqref{eq:F}. Hence, it suffices to show that
\begin{equation}\label{eq:derivadaF}
    F'(\tilde{\mathfrak{s}}) > 0,
\end{equation}
that $F = F(\mathfrak{s};\kappa)$ is analytic as a function of $\mathfrak{s}$ and $\kappa$, and apply the implicit function theorem. We first show that $F(\mathfrak{s};\kappa)$ is analytic. According to \eqref{eq:F}, we just need to check that the coordinate functions $x(\mathfrak{s};\kappa)$, $x_3(\mathfrak{s};\kappa)$ are analytic. We emphasize that the analytic dependence on $\kappa$ is not immediate, as $x_3$ adopts different expressions depending on the sign of this parameter; see \eqref{eq:paramesfera}, \eqref{eq:paramhiperb}, \eqref{eq:parameuclideo}. Nevertheless, for any $\kappa \in \left(-\frac{1}{4},\frac{1}{4}\right)$ the tuple $(x,x_3,x_4)$ can be expressed as a solution of the following differential system:

\begin{equation}\label{eq:xx3x4}
\left\{\def\arraystretch{1.3} \begin{array}{lll} x'' & =- \kappa x  + \frac{\delta^2}{x^3}, \\
x_3' &= \kappa\frac{x x'}{\kappa x^2 -1} x_3 + \frac{\delta  x_4}{x(1 - \kappa x^2)}, \\
        x_4' &= \kappa\frac{x x'}{\kappa x^2 - 1} x_4 + \kappa \frac{\delta x_3}{x(\kappa x^2 - 1)},\end{array} \right.
\end{equation}
where $\delta = \frac{2}{4\kappa + 1}$ and we set the initial conditions
\begin{align*}
    x(0) &= \frac{2}{\sqrt{4\kappa + 1}}, \\
    x'(0) &= 0, \\
    x_3(0) &= 0, \\
    x_4(0) &= \frac{1}{\sqrt{4\kappa + 1}}.
\end{align*}
The equations that define \eqref{eq:xx3x4} are immediate from \eqref{eq:paramesfera}, \eqref{eq:paramhiperb} and \eqref{eq:parameuclideo}. Both this system and its initial conditions depend analytically on $\kappa$, and so will do the solution $(x,x_3,x_4)$. We deduce then that $F = F(\mathfrak{s},\kappa)$ is analytic.

\medskip

Let us now prove \eqref{eq:derivadaF}. A direct computation using \eqref{eq:xx3x4} and \eqref{eq:x} shows that
\begin{equation}\label{eq:derivadaFcalculo}
    F'(\mathfrak{s}) = \frac{2 \left(2 x_3 + (1 + 4 \kappa) x^2 x' x_4  \right)}{\left(1 + 4 \kappa\right)^2 x^3 \left(1 - \kappa x^2\right)}
\end{equation}
for all $\mathfrak{s} \in \mathbb{R}$. At $\mathfrak{s} = \tilde{\mathfrak{s}}$, we have that $x(\tilde{\mathfrak{s}}), x'(\tilde{\mathfrak{s}}), x_3(\tilde{\mathfrak{s}}), x_4(\tilde{\mathfrak{s}}) > 0$. Moreover, $1 - \kappa x(\tilde{\mathfrak{s}})^2 > 0$; this is trivial if $\kappa \leq 0$, while for $\kappa > 0$ is an immediate consequence of \eqref{eq:estarenM3}. In any case, \eqref{eq:derivadaF} holds. This shows that $\tilde{\mathfrak{s}}= \tilde{\mathfrak{s}} (\kappa)$ is an analytic function.


\bibliographystyle{amsalpha}

\vskip 0.2cm

\noindent Alberto Cerezo

\noindent Departamento de Geometría y Topología \\ Universidad de Granada (Spain) \\ Instituto de Matemáticas IMUS \\ Departamento de Matemática Aplicada I \\ Universidad de Sevilla (Spain)

\noindent  e-mail: {\tt cerezocid@ugr.es}

\vskip 0.4cm

\noindent This research has been finantially supported by Grant PID2020-118137GB-I00 funded by MICIU/AEI/10.13039/501100011033 and by ESF+.

\vskip 0.2cm

\end{document}